\documentclass{article}

\usepackage[preprint]{neurips_2021}

\usepackage[utf8]{inputenc}
\usepackage[T1]{fontenc}
\usepackage{hyperref}
\usepackage{url}
\usepackage{booktabs}
\usepackage{nicefrac}
\usepackage{microtype}

\usepackage{footmisc}
\usepackage{graphicx}
\usepackage{subcaption}
\usepackage{amsfonts}
\usepackage{amsmath,amssymb,amsthm}
\usepackage{bbm}
\usepackage{xcolor}
\usepackage{pifont}

\usepackage{enumitem}

\usepackage[utf8]{inputenc} 
\usepackage[T1]{fontenc}    
\usepackage{hyperref}       
\usepackage{url}            
\usepackage{booktabs}       
\usepackage{amsfonts}       
\usepackage{nicefrac}       
\usepackage{microtype}      
\usepackage{xcolor}         

\usepackage{xspace}

\usepackage{algorithm,algpseudocode}
\usepackage{tabularx}
    \newcolumntype{L}{>{\raggedright\arraybackslash}X}

\newcommand{\argmin}{\mathop{\arg\min}}

\newtheorem{theorem}{Theorem}
\newtheorem{definition}{Definition}

\newtheorem{lemma}{Lemma}

\newtheorem{proposition}{Proposition}

\newcommand{\cmark}{\ding{51}}%
\newcommand{\xmark}{\ding{55}}%
\newcommand{\picwidth}{0.35}

\makeatletter
\let\@fnsymbol\@arabic
\makeatother

\title{Distributed Machine Learning with Sparse Heterogeneous Data}

\author{%
  Dominic Richards \\
  Department of Statistics\\
  University of Oxford\\
  24-29 St Giles’, Oxford, OX1 3LB \\
  \texttt{Dominic.Richards94@gmail.com} \\
  \And
  Sahand N Negahban \\
  Department of Statistics and Data Science\\
  Yale University\\
  24 Hillhouse Ave., New Haven, CT 06510\\
  \texttt{Sahand.Negahban@Yale.edu}
   \And
  Patrick Rebeschini  \\
  Department of Statistics\\
  University of Oxford\\
  24-29 St Giles’, Oxford, OX1 3LB \\
  \texttt{Patrick.Rebeschini@stats.ox.ac.uk} \\
}

\begin{document}

\maketitle

\begin{abstract}
Motivated by distributed machine learning settings such as Federated Learning, we consider the problem of fitting a statistical model across a distributed collection of heterogeneous data sets whose similarity structure is encoded by a graph topology. Precisely, we analyse the case where each node is associated with fitting a sparse linear model, and edges join two nodes if the difference of their solutions is also sparse. We propose a method based on Basis Pursuit Denoising with a total variation penalty, and provide finite sample guarantees for sub-Gaussian  design matrices. Taking the root of the tree as a reference node, we show that if the sparsity of the differences across nodes is smaller than the sparsity at the root, then recovery is successful with fewer samples than by solving the problems independently, or by using methods that rely on a large overlap in the signal supports, such as the group Lasso. We consider both the noiseless and noisy setting, and numerically investigate the performance of distributed methods based on Distributed Alternating Direction Methods of Multipliers (ADMM) and hyperspectral unmixing.
\end{abstract}

\section{Introduction}
The wide adoption of mobile phones, wearable and smart devices, has created a influx of data which requires processing and storage. Due to the size of these datasets and communication limitations, it is then often not feasible to send all the data to a single computer cluster for storage and processing. This has motivated the adoption of decentralised methods, such as Federated Learning \citep{mcmahan2017communication,smith2017federated}, which involves both storing and processing data locally on edge devices.

This increase in data sources has led to applications that are increasingly high-dimensional.
To be both statistically and computationally efficient in this setting, it is then important to develop approaches that can exploit the structure within the data. A natural assumption in this case is that the data is sparse in some sense. For instance, a subset of features is assumed to be responsible for determining the outcome of interest or, in the case of compressed sensing \citep{donoho2006compressed}, the data is assumed to be generated from a sparse signal.

Fitting statistical models on data collected and stored across a variety of devices presents a number of computational and statistical challenges. Specifically, the distributed data sets can be heterogeneous owing to data coming from different population distributions e.g.\ each device can have different geographic locations, specifications and users. Meanwhile, from a computational perspective, it is often unfeasible, due to network limitations and robustness, to have a single central ``master'' device collect and disseminate information. This creates a demand for statistical methodologies which are both: flexible enough to model potential statistical differences in the underlying data; and can be fit in a decentralised manner i.e.\ without the need for a centralised device to collate information. 

In this work, we set to investigate the statistical efficiency of a method for jointly fitting a collection of \emph{sparse} models across a collection of heterogeneous datasets. More precisely, models are associated to nodes within a graph, with edges then joining two models if their \emph{difference} is also assumed to be \emph{sparse}. The approach of penalising the differences between models (in an appropriate norm) has, for instance, been applied within both Federated Learning \citep{li2018federated} and Hyperspectral Denoising  \citep{iordache2012total,du2015learning} to encode heterogeneous data. In our case, we consider linear models and obtain precise insights into when jointly fitting models across heterogeneous datasets yield gains in statistical efficiency over other methods, such as, the group Lasso and Dirty Model for multi-task learning \citep{jalali2010dirty}. In addition to being solvable in a centralised manner with standard optimisation packages, the framework we consider is then directly applicable to decentralised applications, as  information only needs to be communicated across nodes/models/devices which share an edge within the graph.

\subsection{Our Contribution}
We consider a total variation scheme that penalises differences between models that share an edge in the graph. This scheme encodes the intuition that if the signal differences are sufficiently sparse then, to recover all signals in the graph, it is more statistically efficient to first recover a single signal associated to a particular reference node (root) and then recover the signal differences associated to edges.
Following the celebrated Basis Pursuit algorithm \citep{candes2006stable}, we consider the solution that minimises the $\ell_1$ norm of the model associated to a root node of the tree and the differences between models that share an edge. We refer to this method as \emph{Total Variation Basis Pursuit (TVBP)}.  A noisy variant similar to Basis Pursuit Denoising \citep{chen2001atomic} is then considered, where the linear constraint is substituted for a bound on the $\ell_2$ norm of the residuals. We call this method as \emph{Total Variation Basis Pursuit Denoising (TVBPD)}.  Note that variants of TVBPD have been successfully applied within the context of hyperspectral data \citep{iordache2012total,du2015learning} and thus, this work now provides theoretical justification in this case.

Given this framework and assuming sub-Gaussian design matrices, we show that statistical savings can be achieved by TVBP and TVBPD as opposed to solving them either independently or with methods that consider the union of supports (see Table \ref{tab:samples}). In the noiseless case, TVBP requires a total sample complexity of $O(s + n^2 s^{\prime})$ where $s$ is the root sparsity, $s^{\prime}$ is the maximum sparsity of the difference along edges and $n$ is the number of nodes. In contrast, methods like the group Lasso and Dirty Model \citep{jalali2010dirty} have an $O(n s)$ total sample complexity, which matches the case when nodes do not communicate with each other. Moreover, note that the TVBP approach does \emph{not} need to know the true underlying graph $G$, whereas the optimal stepwise approach does. If the true graph $G$ is known,  TVBP can incorporate this information to yield improved sampled complexity (see Table \ref{tab:samples:TVBP} in Section \ref{sec:Noiseless:TreeBased}).  In the noisy setting, we demonstrate that TVBPD has an $\ell_1$ estimation error bounded as $O(\sqrt{s} + \sqrt{n s^{\prime}})$, where as the stepwise approach scales as  $O(s + n \sqrt{s^{\prime}})$, and thus, achieves an $O(\sqrt{n} s^{\prime})$ saving in this case. While the step wise approach achieves optimal total sample complexity in the noiseless case, in the noisy setting its estimation scales sub-optimally compared to TVBPD.
\begin{table}[!h]
  \centering
  \renewcommand{\arraystretch}{1.4} 
  \begin{tabularx}{\linewidth}{|L|c|c|c|}
    \hline
   Method  & Total Sample Complexity $\sum_{v \in V} N_{v} $ & Sim. Recovery & Know $G$  \\
     \hline  
   Independent Basis Pursuit  & $n s + \text{Diam}(G)^2 s^{\prime}$  & \cmark & \xmark \\
   \hline
         Stepwise Basis Pursuit & $s+ n s^{\prime}$ & \xmark &\cmark 
 				\\
   \hline
   Group Lasso \citep{negahban2008joint,obozinski2011support} \newline
   Dirty Model \citep{jalali2010dirty}  \newline
   GSP \citep{feng2013generalized} & $ns + \text{Diam}(G)^2 s^{\prime}$ & 
    \cmark & \xmark \\
    \hline
    TVBP (this work) & $s + n^2 \text{Diam}(G) s^{\prime}$ 
    & \cmark & \xmark \\
    \hline
\end{tabularx}
\caption{Worst case total sample complexities (up-to logarithmic and constant factors) for recovering a collection of sparse signals $\{x^{\star}_v\}_{v \in V}$ on a tree graph $G = (V,E)$ in the noiseless case with sub-Gaussian design matrices. Sparsity of the root signal $|\text{Supp}(x^{\star}_1)|\leq s$, sparsity of difference along edges $e = (v,w) \in E$ $|\text{Supp}(x^{\star}_v - x^{\star}_w)|\leq s^{\prime}$. 
  \textit{Sim. Recovery}: Whether the method simultaneously recovers the collection of signals $\{x_v^{\star}\}_{v \in V}$. \textit{Know G}: Whether the sample complexity listed requires knowledge of relationship graph $G$. \textit{TVBP}: Total Variation Basis Pursuit \eqref{equ:TVBP}. When the algorithm \emph{does not} depend on $G$ and $s^{\prime}$, the best case of $G,s^{\prime}$ can taken.
  } 
  \label{tab:samples}
  \vspace{-0.5cm}
\end{table}

Alongside total sample complexity savings over other joint recovery methods, 
we also show that TVBP is amenable to decentralised machine learning. Specifically, the objective can be reformulated with constraints that reflect the graph topology allowing a Distributed Alternating Direction Methods of Multipliers (ADMM) algorithm \citep{Boyd:2011:DOS:2185815.2185816} to be applied. Theoretical findings are also supported by experiments (Section \ref{sec:NoisySimulated}) which demonstrate both: TVBP can outperform group Lasso methods \citep{jalali2010dirty} when model differences are sparse; and TVBPD yields qualitative improvements in hyperspectral unmixing with the real-world AVIRIS Cuprite mine data set.

A key reason for TVBP achieving sample complexity savings over variants of the group Lasso and Dirty Model \citep{jalali2010dirty}, is that the matrices associated to each tasks \emph{do not} need to satisfy an \emph{incoherence condition} with respect to their support sets. The incoherence condition is a common assumption within the analysis of sparse methods like the Lasso (see for instance \citep{wainwright2019high}) and, in short, it requires the empirical covariance of the design matrix to be invertiable when restricted to co-ordinates in the signal support set. This immediately implies that the sample size at each task is larger than the support set of that given task. In contrast, for TVBP an alternative analysis  is conducted by reformulating the problem into a standard basis pursuit objective with an augmented matrix and support set. In this case, the problem structure can be leveraged to show a \emph{Restricted Null Space Property} holds when the sample size at each task scales with the sparsity along the graph edges. This highlights that the Total Variation penalty encodes a different structure when compared to methods like the group Lasso and Dirty Model \citep{jalali2010dirty}. As shown in the noisy case, this approach can then be generalised through the \emph{Robust Null Space Property} \citep{chandrasekaran2012convex,foucart2013invitation}.

\subsection{Related Literature}
\label{sec:RelatedLit}

Learning from a collection of heterogeneous datasets can be framed as an instance of multi-task learning \citep{caruana1997multitask}, with applications in distributed contexts gaining increased attention recently. We highlight the most relevant to our setting. The works \citep{Wang2018distributed,vanhaesebrouck2017decentralized} have considered models penalised in an $\ell_2$ sense according to the network topology to encode prior information. The $\ell_2$ penalty is not appropriate for the sparse setting of our work. 
A number of distributed algorithms have been developed for the sparse setting, for a full review we refer to \citep{azarnia2017cooperative}.   
The works \citep{jalali2010dirty,sundman2016design,wang2016distributed,li2015weighted,wang2016robust,oghenekohwo2017low} have developed distributed algorithms following the group Lasso setting, in that, the signals are assumed to be composed of a common shared component plus an individual component. Within \citep{jalali2010dirty,sundman2016design,wang2016distributed} this requires each node to satisfy an incoherence condition, while the setting in \citep{li2015weighted,oghenekohwo2017low} is a specific case of a star topology within our work. For details on how the incoherence condition influences the sample complexity see discussion in Section \ref{sec:Noiseless:Setup}. The work \citep{park2012geometric} develops a manifold lifting algorithm to jointly recover signals in the absence of an incoherence assumption, although no theoretical guarantees are given.

Federated machine learning \citep{konevcny2016federated,smith2017federated,mcmahan2017communication,li2018federated} is a particular instance of distributed multi-task learning where a central node (root) holds a global model and other devices collect data and update their model with the root. Data heterogeneity can  negatively impact the performance of methods which assume homogeneous data, see \citep{mcmahan2017communication} for the case of Federated Averaging. This motivates modelling the data heterogeneity, as recently done within \citep{li2018federated} where the difference in model parameters is penalised at each step. Our work follows this approach in the case of learning sparse models, and thus, it provides insights into when an improvement in statistical performance can be achieved over other baseline methods.

Simultaneously recovering a collection of sparse vectors can also be framed into the  multiple measurement vectors framework \citep{duarte2011structured}, which has been precisely investigated for $\ell_1/\ell_q$ regularisation for $q > 1$. Specifically,  $\ell_1 / \ell_{\infty}$ was investigated within \citep{zhang2008sparsity,negahban2008joint,turlach2005simultaneous} and $\ell_1/\ell_2$ in \citep{lounici2009taking,obozinski2011support}. Other variants include the Dirty Model of \citep{jalali2010dirty}, multi-level Lasso of \citep{lozano2012multi} and tree-guided graph Lasso of \citep{kim2010tree}. In the same context, a number of works have investigated variants of greedy pursuit style algorithms \citep{feng2013generalized,chen2011hyperspectral,dai2009subspace,tropp2006algorithms}. 
These methods assume a large overlap between the signals, with their analysis often assuming each task satisfies an incoherence condition \citep{negahban2008joint,jalali2010dirty,obozinski2011support} (see Section \ref{sec:Noiseless:Setup}).

The total variation penalty is linked with the fused Lasso \citep{wang2016trend,hutter2016optimal,tibshirani2005sparsity,chambolle2004algorithm,rudin1992nonlinear} and has been widely applied to images due to it promoting piece-wise continuous signals which avoids blurring. As far as we are aware, the only work theoretically investigating the total variation penalty as a tool to link a collection of sparse linear recovery problems has been \citep{chen2010graph}.  This work considers the penalised noisy setting and gives both asymptotic statistical guarantees and an optimisation algorithm targeting a smoothed objective. In contrast, we give finite sample guarantees as well as settings where statistical savings are achieved. The application of hyperspectral unmixing \citep{iordache2011sparse,iordache2012total,du2015learning} has successfully integrated the total variation penalty within their analysis. Here, each pixel in an image can be associated to its own sparse recovery problem,  for instance, the presence of minerals \citep{iordache2012total} or the ground class e.g.\ trees, meadows etc. \citep{du2015learning}.

\section{Noiseless Setting}
This section presents results for the noiseless setting. Section \ref{sec:Noiseless:Setup} introduces the setup and details behind the comparison to other methods in Table \ref{tab:samples}. Section \ref{sec:Noiseless:TreeBased} presents analysis for \emph{Total Variation Basis Pursuit} alongside descriptions of how refined bounds can be achieved. Section \ref{sec:Noiseless:SampleComplexity} presents experimental results for the noiseless setting. Section \ref{sec:ProofSketch} gives a sketch proof of the main theorem.

\subsection{Setup}
\label{sec:Noiseless:Setup}
Consider an undirected graph $G = (V,E)$ with nodes $|V|=n$ and edges $E \subseteq V \times V$. Denote the degree of  a node $v \in V$ by $\text{Deg}(v) = |\{ (i,j) \in E: i=v \text{ or } j=v\}|$ and index the nodes $V=\{1,\dots,n\}$ with a root node associated to $1$. To each node $v \in V$, associate a signal vector $x_v^{\star} \in \mathbb{R}^{d}$. The objective is to estimate the signals $\{x_v^{\star}\}_{v \in V}$ through measurements $\{y_v\}_{v \in V}$ defined as $y_v = A_{v} x^{\star}_v \in \mathbb{R}^{N_v}$ where $A_v \in \mathbb{R}^{N_v \times d}$ is a design matrix. As we now go on to describe, we will assume the signals are both: sparse $x_v^{\star}$ for $v \in V$, and related through the graph $G$. For instance, the graph $G$ can encode a collection of wearable devices connected through a network. Each node holds a collection of data $(y_v,A_v)$ that can, for example, represent some sensor outputs. Alternatively, in Hyperspectral Denoising, each node $v \in V$ is associated to a pixel in a image and the  signal $x^{\star}_v$ indicates the presence of a mineral or classification of land type. The graph can then encode the landscape topology.

Following these examples, it is natural that the collection of signals $\{x^{\star}_v\}_{v \in V}$ will have a sparsity structure related to the graph $G$. Specifically, if two nodes share an edge $e=(v,w) \in E$ then it is reasonable that only a few co-ordinates will change from $x_v^{\star}$ to $x^{\star}_w$. For instance, in Hyperspectral Imaging we expect the composition of the ground to change by a few minerals when moving to an adjacent pixel. This can then be explicitly encoded by assuming difference in the underlying signals can also be sparse. We encode the structural assumption on the signal within the following definition. 
\begin{definition}[$(G,s,s^{\prime})$ Sparsity]
\label{ass:SparsityStructure}
A collection of signals $\{x_v^{\star}\}_{v \in V}$ is $(G,s,s^{\prime})$-sparse if the following is satisfied. The root-node signal support is bounded $|\text{Supp}(x_1^{\star})| \leq s$. For any edge $e =(v,w) \in E$ the support of the difference is bounded $|\text{Supp}(x_v-x_w)| \leq s^{\prime}$. 
\end{definition}
We are interested in the total number of samples $N_{\text{Total Samples}} := \sum_{v \in V} N_{v}$ required to recover all of the signals $\{x_v^{\star}\}_{v \in V}$. 
We begin by describing the total number of samples $N_{\text{Total Samples}}$ required by baseline methods, a summary of which in Table \ref{tab:samples}.

\textbf{Independent Basis Pursuit} For an edge $e = (v,w) \in E$, denote the support of the difference as $S_{e} = \text{supp}(x_v^{\star} - x_w^{\star})$. Let us  suppose for any pair of edges $e,e^{\prime} \in E$ the supports of the differences are disjoint from each other $S_{e} \cap S_{e^{\prime}}  = \emptyset$ and the support of the root $S_{e} \cap S_{1} = \emptyset$. Let $G$ be a tree graph and the integer $i_v \in \{0,\dots,n-1\}$ denote the graph distance from node $v$ to the root agent $1$. If each node has sub-Gaussian matrices $A_v$ and performed Basis Pursuit independently, then the number of samples required by agent $v$ to recover $x^{\star}_v$ scales as $N_{v} \geq s + i_{v} s^{\prime}$. The total sample complexity is at least $N_{\text{Total Samples}} = \sum_{v \in V} N_{v} \geq n s + s^{\prime} \sum_{v \in V} i_v = O(ns + \text{Diam}(G)^2 s^{\prime})$ where we lower bound $\sum_{v \in V} i_v$ by considering the longest path in the graph including agent $1$. The optimisation problem associated to this approach is then 
\begin{align*}
    \min_{x_v} \|x_v\|_1 \text{ subject to } A_{v}x_v = y_v \text{ for  } v \in V,
\end{align*}
which can be solved independently for each agent $v \in V$.

\textbf{Stepwise Basis Pursuit} Consider the support set structure described in \textbf{Independent Basis Pursuit}. The signals can then be recovered in a stepwise manner with a total sample complexity of $O(s+n s^{\prime})$. Precisely, order $s$ samples can recover the root signal $x_1^{\star}$, meanwhile order $n \times s^{\prime}$ samples can recover each of the differences associated to the edges. Any node's signal can then be recovered by summing up the differences along the edges. This yields a saving from  $O(n s + n^2 s^{\prime})$  to  $O(s + n s^{\prime})$, which is significant when the difference sparsity $s^{\prime}$ is small and the network size $n$ is large. This embodies the main intuition for the statistical savings that we set to unveil in our work. The associated optimisation problem in this case requires multiple steps. Let the solution at the root be denoted $\widehat{x}_1 = \argmin_{x_1 : A_1 x_1 = y_1} \|x_1\|_1$. For a path graph we recursively solve, with the notation $\sum_{i=2}^{1}\cdot = 0$, the following sequence of optimisation problems for $j=2,\dots,n$ 
\begin{align*}
    \min_{\Delta} \|\Delta\|_1 \text{ subject to } A_j \Delta = y_j - A_j \widehat{x}_1 - \sum_{i=2}^{j-1}A_{k} \widehat{\Delta}_i,
\end{align*}
with the sequence of solutions here denoted $\{\widehat{\Delta}_j\}_{j=2}^{n}$.

\textbf{Group Lasso / Dirty Model / GSP} The group Lasso \citep{negahban2008joint,obozinski2011support}, Dirty Model \citep{jalali2010dirty} and  Greedy Pursuit Style algorithms \citep{feng2013generalized,chen2011hyperspectral,dai2009subspace,tropp2006algorithms} from the Multiple Measurement Vector Framework \citep{duarte2011structured}, assume an incoherence condition  on the matrices $A_v$ for $v \in V$, which impacts the total sample complexity. Namely, for $v \in V$ denote the support $S_v = \text{Supp}(x^{\star}_v) \subseteq \{1,\dots,d\}$ alongside the design matrix restricted to the co-ordinates in $S_v$ by $(A_v)_{S_v} \in \mathbb{R}^{N_{v} \times |S_v|}$. The incoherence assumption for the Dirty Model \citep{jalali2010dirty} then requires  $(A_v)^{\top}_{S_v}(A_v)_{S_v} $ to be full-rank (invertibility), and thus, $N_v \geq |S_v|$. Since $|S_v| = s + i_v s^{\prime}$ the lower bound on $N_{\text{Total Samples}}$ then comes from \textbf{Independent Basis Pursuit}. 
One of the associated optimisation problem for the group lasso then takes the form for some $\lambda > 0$
\begin{align*}
    \min_{x_1,x_2,\dots,x_n} \sum_{v \in V} \|y_v - A_v x_v\|_2^2 + \lambda \sum_{j=1}^{d} 
    \|( (x_{1})_j,(x_{2})_j,\dots,(x_{n})_{j})\|_2.
\end{align*}

\subsection{Total Variation Basis Pursuit}
\label{sec:Noiseless:TreeBased}
To simultaneously recover the signals $\{x_v^{\star}\}_{v \in V}$ we consider the \emph{Total Variation Basis Pursuit} (TVBP) problem. Specifically for a tree-graph $\widetilde{G} = (V,\widetilde{E})$ with edges $\widetilde{E} \subset V \times V$, consider: 
\begin{align}
\label{equ:TVBP}
    \min_{x_1,x_2,\dots,x_n} & \|x_1\|_1 + \sum_{e  =(v,w) \in \widetilde{E} } \|x_v - x_w \|_1 \quad \text{subject to}
    \quad 
    A_v x_v = y_v \text{ for } v \in V.
\end{align}
Let us denote a solution to \eqref{equ:TVBP} as $\{x^{TVBP}_{v}\}_{v \in V}$. Note that $\widetilde{G}$ does not have to be equal to the graph associated to the sparsity of $\{x^{\star}_v\}_{v \in V}$. For instance, we can consider a star graph for $\widetilde{G}$ whilst $G$ is a more complex unknown graph. Furthermore, we highlight that \eqref{equ:TVBP} couples the agents solutions together in a distinctly different manner when compared to the methods in Section \ref{sec:Noiseless:Setup} e.g. Group Lasso.

We now upper bound on the number of samples $N_{1},N_{\text{Non-root}}$ required for TVBP to recover the signals.  For the following, we say that if $A$ has independent and identically distributed (i.i.d.) sub-Gaussian entries, then the $i,j$th entry $A_{ij}$ satisfies $\mathbb{P}(|A_{ij}| \! \geq\! t ) \! \leq \! \beta e^{-\kappa t^2}$ for all $t \! \geq \! 0$ for sub-Gaussian parameters $\beta$ and $\kappa$.
Let us also denote the root sample size as $N_{\text{Root}} = N_1$, with all non-root agents  having the same sample size $N_{\text{Non-root}} = N_{v}$ for $v \in V \backslash \{1\}$.  The proof for the following theorem can then be found in Appendix \ref{sec:Starlike}.  
\begin{theorem}
\label{thm:GeneralTree}
Suppose the signals $\{x^{\star}_{v}\}_{v \in V}$ are $(G,s,s^{\prime})$-sparse and the matrices satisfy $A_v = \frac{1}{\sqrt{N_v}} \widetilde{A}_v$ where  $\{\widetilde{A}_v\}_{v \in V}$ each have i.i.d.\ sub-Gaussian entries. Fix $\epsilon > 0$. If  
\begin{align*}
    N_{\text{Root}} & \gtrsim 
    \max\{s, n^2 \text{Diam}(G) s^{\prime}\} \big( \log(d) + \log(1/\epsilon) \big)
    \text{ and }\\
    N_{\text{Non-root}} & \gtrsim n \text{Diam}(G) s^{\prime} \big( \log(d) + \log(n/\epsilon) \big),
\end{align*}
then with probability greater than $1-\epsilon$ the TVBP solution with a star graph $\widetilde{G}$ is unique and satisfies $x^{TVBP}_{v} = x_v^{\star}$ for any $v \in V$. 
\end{theorem}
Theorem \ref{thm:GeneralTree} provides conditions on the number of samples held by each agent in order for TVBP to recover the signals when $\widetilde{G}$ is a star topology. As seen in Table \ref{tab:samples}, the total number of samples in this case satisfies $N_{\text{Total Samples}} = O(s + n^2 \text{Diam}(G) s^{\prime})$. As we now go on to describe, the sample complexities in Theorem \ref{thm:GeneralTree} are  \emph{worst case} since we have assumed no prior knowledge of $G$. 

Incorporating knowledge of the signal graph $G$ into the TVBP problem \eqref{equ:TVBP}  naturally influences the total sample complexity required for recovery. Table \ref{tab:samples:TVBP} provides a summary of the total complexity required in two  different cases. Precisely, an improved total sample complexities can be achieved when: $G$ is a known tree graph, as well as when the non-root design matrices $\{A_v\}_{v \in V \backslash \{1\}}$ are the same. When $G$ is a  known tree graph, the sample complexity is reduced to $s + \max\{n^2,n \text{Deg}(V\backslash \{1\})^2\text{Diam}(G)^2\}s^{\prime}$ from $s + n^2 \text{Diam}(G)$ previously. In the case $\text{Diam}(G) = O(\sqrt{n})$, then an order $\sqrt{n}s^{\prime}$ saving in achieved. Meanwhile, if $\{A_{v}\}_{v \in V \backslash \{1\}}$ are also equal, then the sample complexity reduces to $s + n \text{Deg}(1)^2 s^{\prime}$ which, for constant degree root nodes $\text{Deg}(1)$, matches the optimal stepwise method. This precisely arises due to the null-spaces of the non-root nodes matrices being equal in this case, allowing the analysis to be simplified. Details of this are provided within the proof sketch. The assumption that the sensing matrices $\{A_{v}\}_{v \in V \backslash \{1\}}$ are equal is then natural for compressed sensing and Hyperspectral applications as $A_v$ represents the library of known spectra, and thus, can be identical across the nodes. 
\begin{table}[!h]
  \centering
  \renewcommand{\arraystretch}{1.4} 
  \begin{tabularx}{\linewidth}{|L|c|c|c|}
    \hline
   Method \& Assumptions  & Total Sample Complexity $\sum_{v \in V} N_{v} $  & Know $G$  \\
     \hline  
    TVBP  & $s + n^2 \text{Diam}(G) s^{\prime}$ 
      & \xmark \\
    \hline
    TVBP + $G$ Tree   & $s + \max\{n^2,n \text{Deg}(V\backslash \{1\})^2\text{Diam}(G)^2\}  s^{\prime}$ 
    & 
    \cmark\\
    \hline
    TVBP + $G$ Tree + $\{A_{v}\}_{v \in V \backslash \{1\}}$ equal
    & 
    $s + n \text{Deg}(1)^2 s^{\prime}$
    & 
    \cmark\\
    \hline
\end{tabularx}
\caption{Setting as described in Table \ref{tab:samples}. Comparing total sample complexity for TVBP with different assumptions: whether $G$ is a known tree; or $G$ is a known tree and the design matrices $\{A_v\}_{v \in V \backslash \{ 1\}}$ are identical.
Formal results can be found in Theorems  \ref{thm:Starlike} and \ref{thm:IdenticalMatrices} within the Appendix \ref{sec:RefinedResults}. When the algorithm depends upon the choice of graph $G$, the rate given depends upon this choice of  $G$ and the associated sparsity along edges $s^{\prime}$.
  } 
  \label{tab:samples:TVBP}
  \vspace{-0.8cm}
\end{table}

\subsection{Experiments for Noiseless Case}
\label{sec:Noiseless:SampleComplexity}
This section presents numerical experiments for Total Variation Basis Pursuit problem \eqref{equ:TVBP}. The paragraph  \textbf{Statistical Performance} focuses on the statistical performance of TVBP, supporting the results summarised in Table \ref{tab:samples:TVBP}. Paragraph  \textbf{Distributed Algorithm} outlines how the objective \eqref{equ:TVBP} can be solved in a decentralised manner. 

 \textbf{Statistical Performance} Figure \ref{fig:SampleComplexity1} plots the probability of recovery against the number of  samples held by non-root nodes $N_v$ for $v \in V \backslash \{1\}$ with a fixed number of root agent samples $N_{1} = \lfloor 2s \log(ed/s)\rfloor$. Observe, for a path topology and  balanced tree topology, once the non-root nodes have beyond approximately $30$ samples, the solution to TVBP finds the correct support for all of graph sizes. In contrast, the number of samples required to recover a signal with Basis Pursuit at the same level of sparsity and dimension considered would require at least $80$ samples, i.e.\ $2s\log(ed/s)$. We therefore save approximately $50$ for each non-root problem. 
\begin{figure}[!h]
\centering
\includegraphics[width = \picwidth\textwidth]{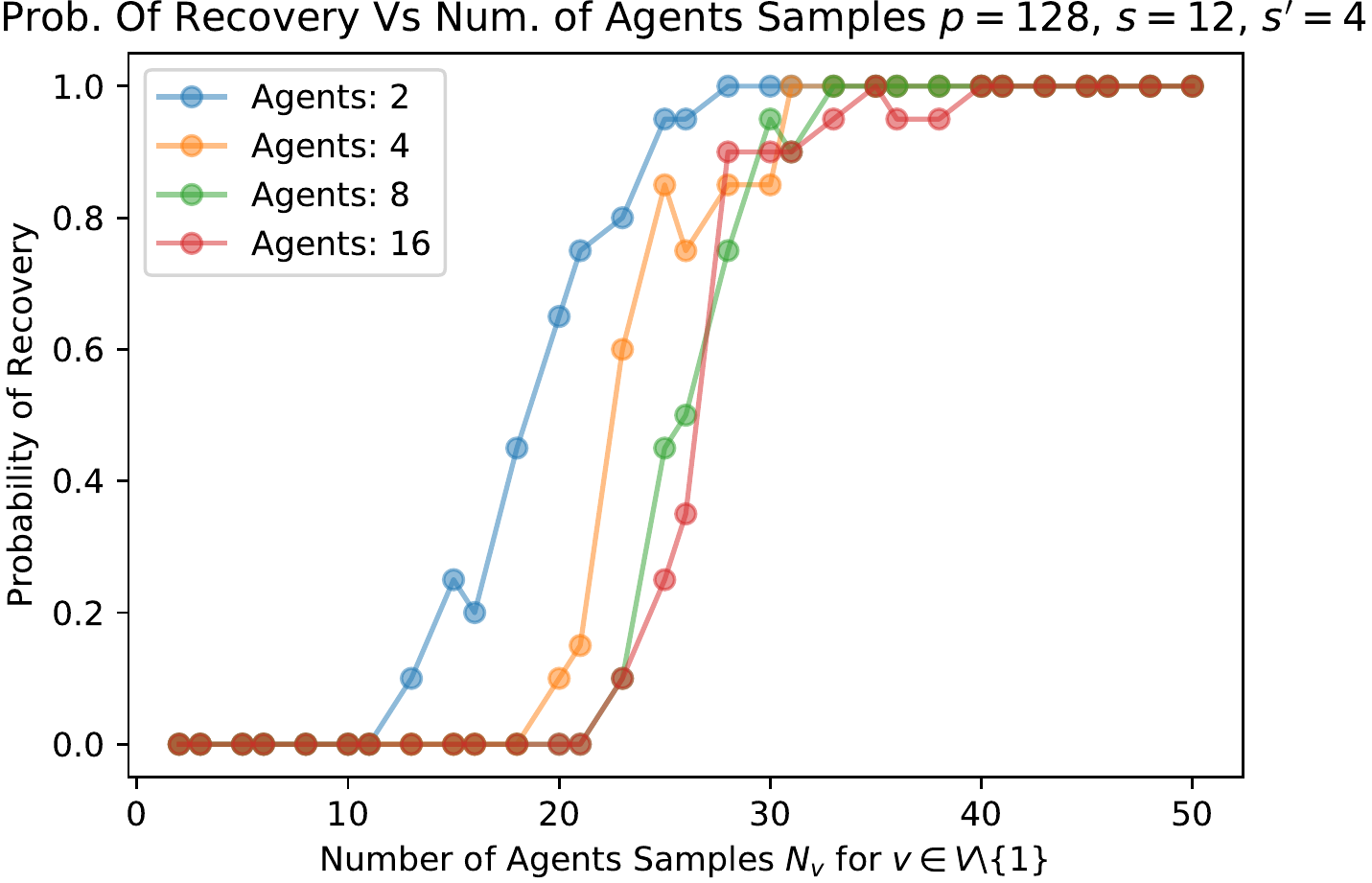}
\hspace{0.2cm}
\includegraphics[width = \picwidth\textwidth]{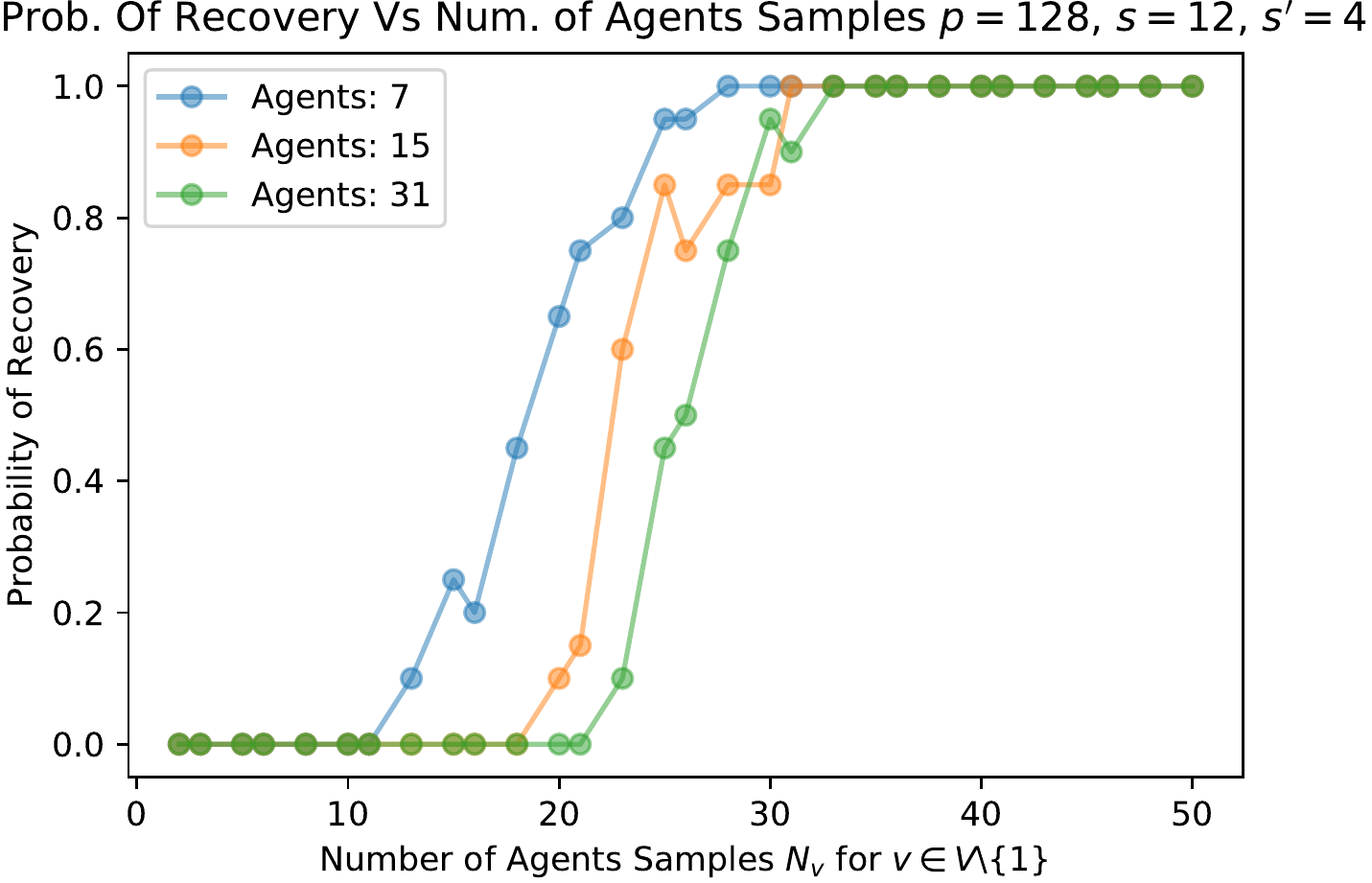}
\caption{Probability of recovery vs number of non-root node samples $N_{v}$ for $v \in V \backslash \{1\}$. Problem setting $d = 128$, $s = 12$, $s^{\prime} = 4$ and $N_1 = \lfloor 2s\log(ed/s) \rfloor = 80$, for path (\textit{Left}) and balance tree with branches of size 2 (\textit{Right}). Lines indicates graph sizes with $n \in \{2,4,8,16\}$ for path and $n \in \{7,15,31\}$ for balanced tree with heights of $\{2,3,4\}$ respectively. Solution to reformulated problem  \eqref{equ:DecentralisedPursit:Reformulated} found using CVXOPT. Each point is an average of 20 replications. Signal sampled from $\{1,-1\}$, differences concatenation of $s^\prime$  values. $\{A_v\}_{v \in V}$ standard Gaussian and $\widetilde{G} = G$.  }
\label{fig:SampleComplexity1}
\vspace{-0.3cm}
\end{figure} 

\textbf{Distributed Algorithm}
To solve the optimisation problem \eqref{equ:TVBP} in a decentralised manner an ADMM algorithm can be formulated, the details of which are given in Appendix \ref{sec:Simulations:Decentralised}. The Optimisation Error for the method is plotted in Figure \ref{fig:OptError}, which is seen to converge with a linear rate. The convergence for a path topology is slower, reaching a precision of $10^{-8}$ in 300 iterations for $7$ nodes, while the same size balanced tree topology reaches a precision of $10^{-15}$. This is expected as the balanced trees considered are more connected than a path, and thus, information propagates around the nodes quicker. Larger tree topologies also require additional iterations to reach the same precision, with a size $63$ tree reaching a precision of $10^{-7.5}$ in 300 iterations. 
\begin{figure}[!h]
\centering
\includegraphics[width = \picwidth\textwidth]{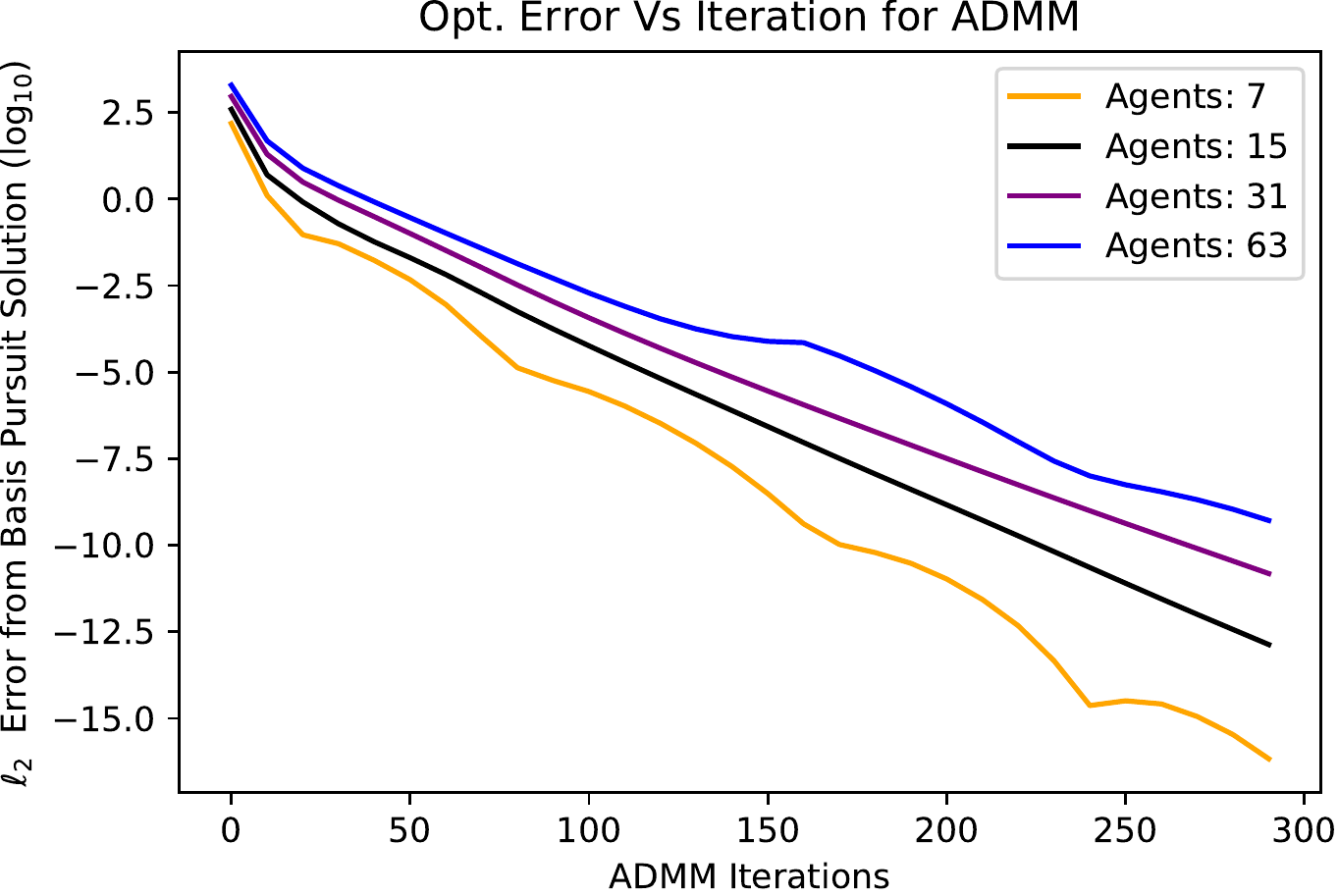}
\hspace{0.2cm}
\includegraphics[width = \picwidth\textwidth]{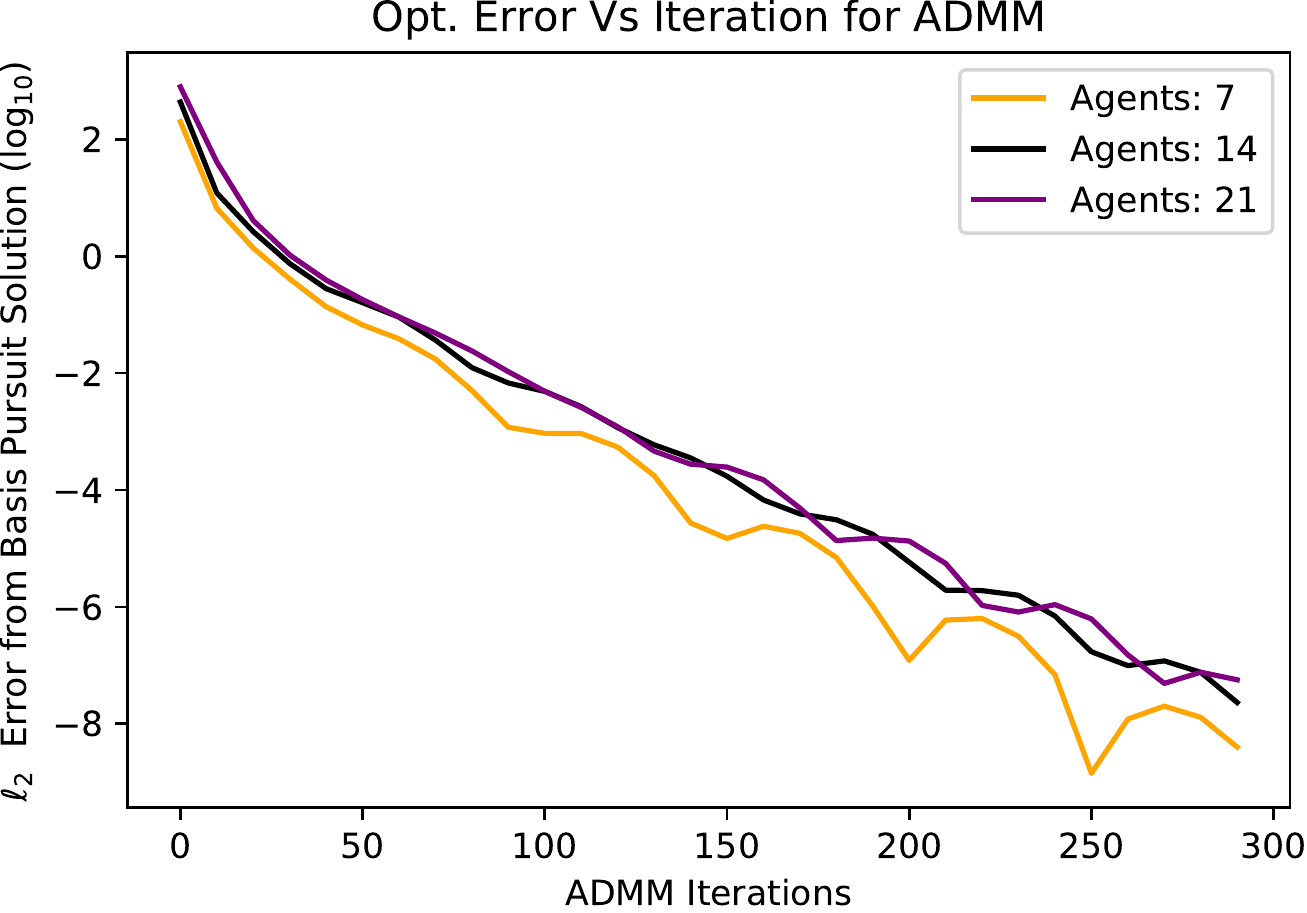}
\caption{Optimisation error $\|x^{t} - x^{\star}_{\text{BP}}\|_2^2$ (Log scale) vs Iterations for ADMM method with $\rho = 10$ for different graph sizes (lines) and topologies (plots). Here $x^{\star}_{\text{BP}}$ is solution to \eqref{equ:TVBP}.  Problem parameters $d = 2^{9}$, $s = \lfloor 0.1 d \rfloor$ and $s^{\prime} = 4$. \textit{Left}: Balanced trees, branch size $2$ and heights $\{2,3,4,5\}$. \textit{Right}: Path topology.  Agent sample size $N_1 = 2s\log(ep/2s)$ and $N_v = 150$ for $v \not= 1$. Matrices $\{A_v\}_{v \in V}$ i.i.d.\ standard Gaussian  entries, $x^{\star}_1$ has $s$ values randomly drawn from $\{+1,-1\}$ and 
$\{\Delta^{\star}_e\}_{e \in E}$ each have $s^{\prime}$ i.i.d.\ standard Gaussian entries, locations chosen at random. }
\label{fig:OptError}
\vspace{-0.5cm}
\end{figure}

\subsection{Proof Sketch for Theorem \ref{thm:GeneralTree}}
\label{sec:ProofSketch}
This section provides a proof sketch for Theorem \ref{thm:GeneralTree}. For a response $y \in \mathbb{R}^{N}$ and design matrix $A \in \mathbb{R}^{N \times d}$, recall that the \emph{Basis Pursuit} problem is given by 
\begin{align}
\label{equ:BasisPursuit}
    \min \|x\|_1 \text{ subject to } Ax = y.
\end{align}
Let us denote the solution to \eqref{equ:BasisPursuit} as $x^{\text{BP}}$.
It is known that if $y = A x^{\star}$ for some sparse vector $x^{\star}$ supported on $S$, then the solution to \eqref{equ:BasisPursuit} recovers the signal $x^{\text{BP}} = x^{\star}$ if and only if $A$ satisfies the \emph{Restricted Null Space Property} with respect to $S$, that is, 
\begin{align}
\label{equ:RSP}
    2 \|(x)_{S}\|_1 \leq  \|x\|_1 \text{ for } x \in \text{Ker}(A) \backslash \{0\}.
\end{align}
In the case $A$ has i.i.d.\ sub-Gaussian entries and $N \gtrsim |S|c^{-2} \log(d/\delta)$ for $c \in (0,1/2)$, we then have $\|(x)_S\|_1 \leq c \|x\|_1$ with probability greater than $1-\delta$ \citep{foucart2013invitation}. 

The proof for Theorem \ref{thm:GeneralTree} proceeds in two steps. Firstly, the TVBP problem \eqref{equ:TVBP} is reformulated into a standard basis pursuit problem \eqref{equ:BasisPursuit} with an augmented matrix $A$, signal $x^{\star}$ and support set $S$. Secondly, we show that the Restricted Null Space Property can be satisfied in this case by utilising the structure of $A$, $x^{\star}$ and $S$. Each of these steps is described within the following paragraphs. For clarity, we assume the TVBP problem with $\widetilde{G}=G$, the signal graph $G$ being a tree graph with agent $1$ as the root, and all non root agents with equal design matrix $A_v = \widehat{A}$ for $v \in \{2,3,4,5,\dots,n\}$. Discussion on weakening these assumptions is provided at the end, with fulls proofs given in Appendix \ref{sec:Proof:Noiseless}.  

\textbf{Reformulating TVBP problem}
Let us denote $x_1 \in \mathbb{R}^{d}$ and $\Delta_{i} \in \mathbb{R}^{d}$ for  $i=1,\dots,n-1$ where edges $e \in  E$ are associated to an integer $e \rightarrow i$. The TVBP problem \eqref{equ:TVBP} can then be reformulated as 
\begin{align*}
    \min_{x_1, \{\Delta_e\}_{e \in E} } \|x_1\|_1 + \sum_{i=1}^{n-1} \|\Delta_i\|_1 \quad
    \text{subject to}
    \quad
    \begin{pmatrix}
    A_1 & 0 & 0 & \dots & 0 \\
    \widetilde{A} & \widetilde{A} & H_{13} & \dots & H_{1n} \\
    \widetilde{A} & H_{22} & \widetilde{A} & \dots & H_{2n} \\
    \vdots & \ddots & \ddots & \ddots & \vdots \\
    \widetilde{A} & H_{n2} & H_{n3} & \dots & \widetilde{A} \\
    \end{pmatrix}
    \begin{pmatrix}
    x_1 \\
    \Delta_{1}\\
    \Delta_{2}\\
    \vdots \\
    \Delta_{n-1}
    \end{pmatrix}
    = 
    \begin{pmatrix}
    y_1 \\
    y_2\\
    y_3\\
    \vdots \\
    y_n
    \end{pmatrix}
\end{align*}
where the matrices $\{H_{ij}\}_{i,j=1,\dots,n-1}$ take values $H_{ij} = \widetilde{A}$ if the $j$th agent is on the path from agent $i$ to the root node $1$, and $0$ otherwise. The above is equivalent to a Basis Pursuit problem \eqref{equ:BasisPursuit} with  $x = (x_1,\Delta_{1},\Delta_{2},\dots,\Delta_{n})$, $y = (y_1,y_2,\dots,y_n)$, an $A \in \mathbb{R}^{nd \times nd}$ encoding the linear constraint above, and sparsity structure $S = S_1 \cup \big\{ \cup_{e,i} ( S_{e} + i) \big\}$. That is for an edge $e = (v,w) \in E$ the support $S$ is constructed by off-setting $S_e = \text{Supp}(x_v^{\star} - x^{\star}_w)$ by an integer $(S_e + i) = \{k + i: k \in S_e\}$. 

\textbf{Showing Restricted Null Space Property} We begin by noting that if $x \in \text{Ker}(A)\backslash \{0\}$ then 
\begin{align}
\label{equ:KernelCondition}
    A_1 x_1 = 0 \quad\quad \widetilde{A} \Delta_e = 0 \text{ for } e= (v,w) \in E \text{ such that } v,w \not= 1,
\end{align}
the second equality being over edges $e$ not connected to the root node. To see this, suppose the edge $e \in E$ is both: associated to the integer $i$; and not directly connected to the root $1 \not\in e$. Consider the edge neighbouring $e$ closest to the root, say, $e^{\prime} \in E$ with integer $j$.  We have from $x \in \text{Ker}(A) \backslash \{0\}$
\begin{align*}
    \widetilde{A} x_1 + \widetilde{A} \Delta_i  + \sum_{k \not= i} H_{ik} \Delta_k = 0
    \quad\quad \text{ and } \quad\quad 
    \widetilde{A} x_1 + \widetilde{A} \Delta_j  + \sum_{k \not= j} H_{jk} \Delta_k = 0.
\end{align*}
Taking the difference of the two equations we get $\widetilde{A}\Delta_i = 0 $ since both: $j$ is on the path from $i$ to the root so $H_{ij} = \widetilde{A}$; and the path from $j$ to the root is shared i.e.\ $\sum_{k \not=i,j} H_{ij}\Delta_{k} = \sum_{k \not= j} H_{jk}\Delta_k $. 

In a similar manner to Basis Pursuit, the constraints \eqref{equ:KernelCondition} are used to control the norms $\|(x_1)_{S_1}\|_1$ and $\|(\Delta_{e})_{S_{e}}\|_1$ for $e \in E \backslash \{ e \in E : 1 \in e \}$. Precisely, if $A_1 \in \mathbb{R}^{N_1 \times d}$ and $\widetilde{A} \in \mathbb{R}^{N_{\text{Non-root}} \times d}$ both i.i.d.\ sub-Gaussian with  $N_1 \gtrsim s \log(1/\delta)$  and $N_{\text{Non-root}} \gtrsim s^{\prime}\log(1/\delta)$, then $\|(x_1)_{S_1}\|_1 \leq \|x_1\|_1/4$ and $\|(\Delta_{e})_{S_{e}}\|_1 \leq  \|\Delta_{e}\|_1/4$ with high probability. Controlling the norm for the edges $e \in E$ connected to the root $1 \in e$ is then more technical. In short, if $N_1,N_{\text{Non-Root}} \gtrsim \text{Deg}(1)^2 s^{\prime} \log(1/\delta) $ then $\|(\Delta_e)_{S_{e}}\|_1 \leq  \big( \|x_1\|_1 + \|\Delta_e\|_1\big)/4\text{Deg}(1) $. Summing up the bounds gives 
\begin{align*}
   \|(x)_{S}\|_1 = \|(x_1)_{S_1}\|_1 + \sum_{e \in E}\|(\Delta_e)_{S_e}\|_1 
    \leq 
    \frac{1}{2} \|x_1\|_1 + \frac{1}{2} \sum_{e \in E} \|\Delta_{e}\|_1
    =\frac{1}{2} \|x\|_1
\end{align*}
as required. When the matrices $\{A_v\}_{v \in V \backslash \{1\}}$ are different the condition \eqref{equ:KernelCondition} may no longer be satisfied, and thus, an alternative analysis is required. Meanwhile, if $\widetilde{G}$ is a star and does not equal $G$, a different sparsity set $\widetilde{S}$ is considered where the support along edges are swapped $s^{\prime} \rightarrow \text{Diam}(G) s^{\prime}$.

\section{Noisy Setting}
This section demonstrates how the TVBP problem can be extended to the noisy setting,
Section \ref{sec:TVBasisPursuitDenoisingProblem} introduces Total Variation Basis Pursuit Denoising (TVBPD) alongside theoretical guarantees. 
Section \ref{sec:NoisySimulated} presents experiments investigating the performance of TVBPD.  

\vspace{-0.1cm}
\subsection{Total Variation Basis Pursuit Denoising}
\label{sec:TVBasisPursuitDenoisingProblem}
Let us assume that $y_v = A_v x_v^{\star} + \epsilon_v$ for $v \in V$. In this case the equality constraint in the TVBP problem \eqref{equ:TVBP} is swapped for a softer penalisation, leading to the \emph{Total Variation Basis Pursuit Denoising} (TVBPD) problem for a graph $\widetilde{G} = (V,\widetilde{E})$ and penalisation $\eta > 0$
\begin{align}
\label{equ:DecentralisedDenoising}
\min_{\{x_v\}_{v \in V} } \|x_1\|_1 & + \sum_{e=(v,w) \in \widetilde{E}} \|x_v - x_w\|_1 \text{ subject to } \sum_{v \in V} \| A_v x_v - y_v\|_2^2  \leq \eta^2.
\end{align}
The equality constraint $A_v x_v = y_v$ at each agent $v \in V$ in \eqref{equ:TVBP} is now swapped with an upper bound on the deviation $\|A_v x_v - y_v\|_2^2$. Given this problem, we now provide gaurantees on the $\ell_1$ estimator error for the solution of \eqref{equ:DecentralisedDenoising}. The proof of the following Theorem is in Appendix \ref{sec:TVBasisPursuitDenoisingProblem}.   
\begin{theorem}
\label{thm:TVBPD}
Suppose $G$ is a tree graph, the signals $\{x_v^{\star}\}_{v \in V}$ are $(G,s,s^{\prime})-$sparse and $y_v = A_v x^{\star}_v + \epsilon_v$ for $v \in V$. Assume that $A_v = \widetilde{A}_v/\sqrt{N_{v}}$ where $\{\widetilde{A}_v\}_{v \in V}$ each have i.i.d.\ sub-Gaussian entries.  Fix $\epsilon > 0$. If $\eta^2 \geq  \sum_{v \in V} \|\epsilon_{v}\|_2^2/n$ and 
\begin{align*}
    N_{\text{Root}}  \gtrsim s \big(\log(d) + \log(1/\epsilon))
    \quad \text{ and } \quad
    N_{\text{Non-root}}  \gtrsim n^2 s^{\prime} \big(\log(d) + \log(n/\epsilon)\big),
\end{align*}
then with probability greater than $1-\epsilon$ the solution to \eqref{equ:DecentralisedDenoising} with $\widetilde{G} = G$ satisfies
\begin{align*}
    & \|x_1 - x_1^{\star}\|_1 
    + 
    \sum_{e = (v,w) \in E} \|(x_v - x_w) - (x^{\star}_v - x^{\star}_w) \|_1
    \lesssim
    \big(\sqrt{s} + \text{Deg}(G) \sqrt{n s^{\prime}}\big)\eta.
\end{align*}
\vspace{-0.5cm}
\end{theorem}
Theorem \ref{thm:TVBPD} gives conditions on the sample size so a bound on the $\ell_1$ estimation error can be achieved. For the stepwise approach in Section \ref{sec:Noiseless:Setup} the $\ell_1$ estimation error scales as $\sqrt{s} + n \times \sqrt{s^{\prime}}$. Therefore, TVBPD yields an order $\sqrt{n}$ saving in $\ell_1$ estimation error over the step wise approach. This highlights two sample size regimes. When the total sample size is $O(s + n s^{\prime})$, the step wise approach is provably feasible and the estimation error is $O(\sqrt{s} + n \sqrt{s^{\prime}})$. Meanwhile, when the total sample size is $O(s + n^{3} s^{\prime})$, TVBPD is provably feasible and the estimation is $O(\sqrt{s} + \sqrt{n s^{\prime}})$. The gap in sample size requirements between the noisy case with TVBPD $O(s + n^3 s^{\prime})$ and the noiseless  case with TVBP $O(s + n^2\text{Diam}(G)s^{\prime})$ is due to a different proof technique for TVBPD. We leave extending the techniques for analysing TVBP to the case of TVBPD to future work.

\vspace{-0.1cm}
\subsection{Experiments for Total Variation Basis Pursuit Denoising}
\label{sec:NoisySimulated}
This section present simulation results for the Total Variation Basis Pursuit Denoising problem \eqref{equ:DecentralisedDenoising}. The following paragraphs, respectively, describe results for synthetic and real data.

\textbf{Synthetic Data.} Figure \ref{fig:GroupLasso1} plots the $\ell_1$ estimation error for Total Variation Basis Pursuit Denoising, group Lasso and the Dirty Model of \citep{jalali2010dirty}, against the number of agents for both path and balanced tree topologies. As the number of agents grows, the estimation error for the group Lasso methods grows quicker than the total variation approach. The group Lasso variants perform poorly here due to the union of supports growing with the number of agents, and thus, the small overlap between agent's supports. 

\begin{figure}[!h]
\centering
\includegraphics[width = \picwidth\textwidth]{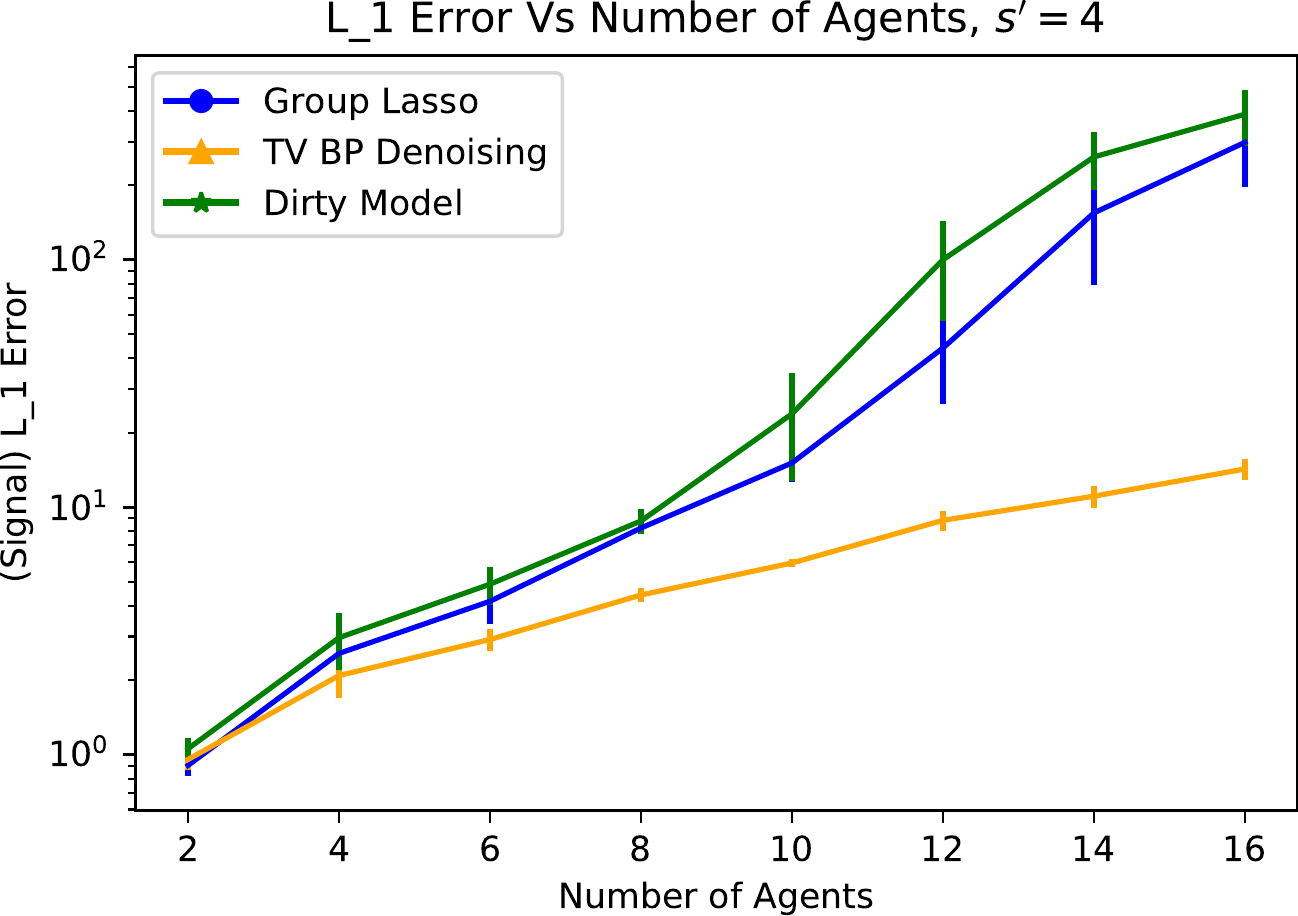}
\hspace{0.2cm}
\includegraphics[width = \picwidth\textwidth]{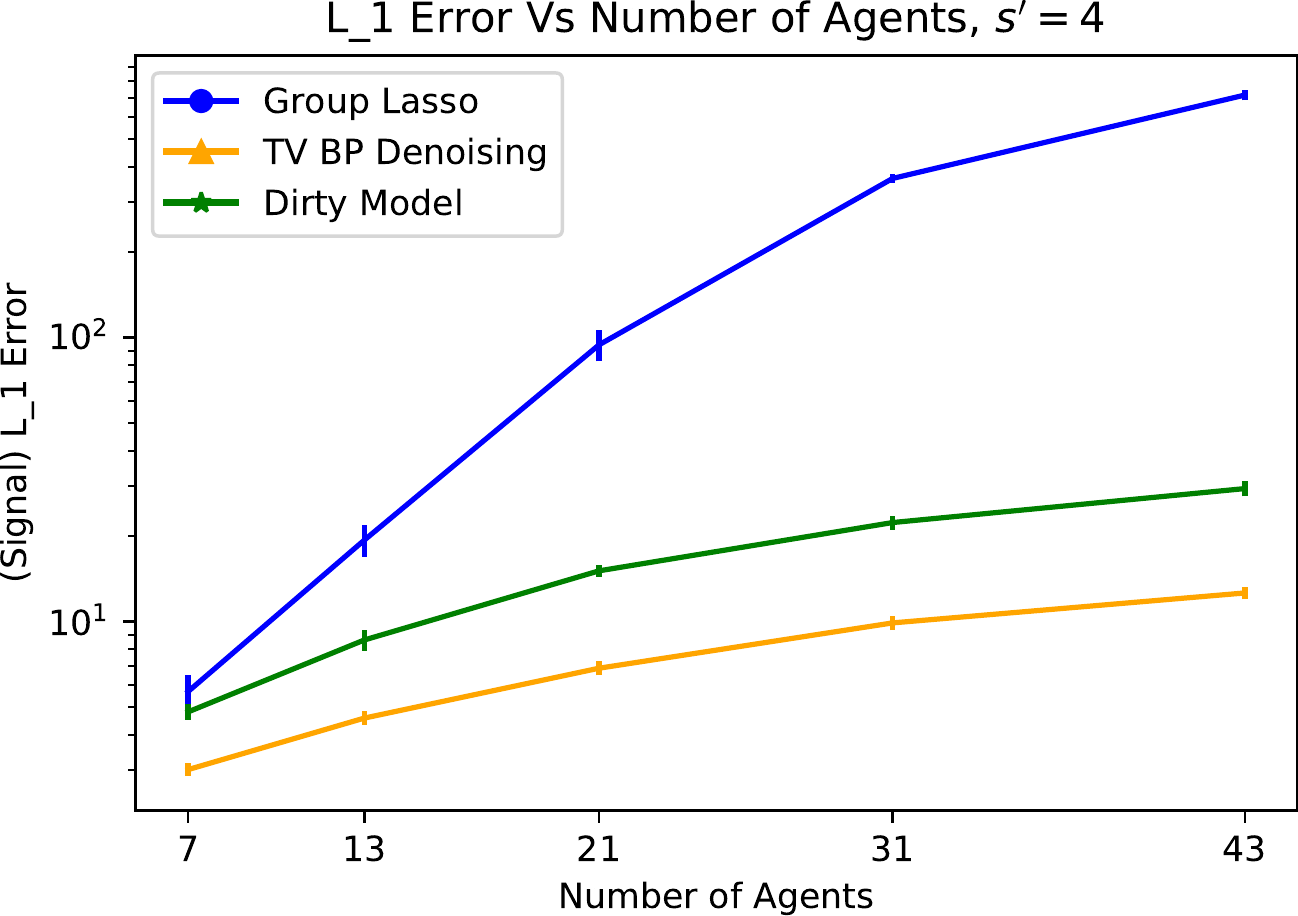}
\caption{ 
$\ell_1$ estimation error $\sum_{v \in V} \|x_v - x^{\star}_v\|_1$ ($\log_{10}$ scale) against number of agents for Total Variation Basis Pursuit Denoising solved using SPGL1 Python package (\textit{Yellow}), group Lasso (\textit{blue}) and Dirty Model of \citep{jalali2010dirty} (\textit{Green}). \textit{Left}: Path topology. \textit{Right}: Balanced tree topology height $2$ branching rate $\{2,3,4,5,6\}$. The same i.i.d.\ standard Gaussian matrix was associated to each node with $N_v = 200$ for $v \in V$,  with parameters were $d=2^{9}$, $s = 25$  and $s^{\prime} = 4$. Signal at the root $x_1^{\star}$ and differences  $\{x^{\star}_v - x{\star}_w \}_{(v,w) \in E}$ sampled from $\{+1,-1\}$ with no overlap in supports. 
}
\label{fig:GroupLasso1}
\end{figure}

\textbf{Hyperspectral Unmixing.}  
We apply Total Variation Basis Pursuit Denoising to the popular AVIRIS Cuprite mine reflectance dataset \url{https://aviris.jpl.nasa.gov/data/free_data.html} with a subset of the USGS library splib07 \citep{kokaly2017usgs}. As signals can be associated to pixels in a 2-dimensional image, it is natural to consider the total variation associated with a grid topology. Computing the total variation explicitly in this case can be computationally expensive, see for instance  \citep{padilla2017dfs}. We therefore simplify the objective by tilling the image into groups of $n=4$ pixels arranged in a 2x2 grid, with each group considered independently. This is common approach within parallel rendering techniques, see for instance \citep{molnar1994sorting}, and is justified in our case as the signals are likely most strongly correlated with their neighbours in the graph.  Note that this also allows our approach to scale to larger images as the algorithm can be run on each tile in an embarrassingly parallel manner.  
More details of the experiment are in Appendix \ref{sec:AVIRIS:Extra}. 

We considered four methods: applying Basis Pursuit Denoising to each pixel independently; Total Variation Denoising \eqref{equ:DecentralisedDenoising} applied to the groups of 4 pixels as described previously; the group Lasso applied to the groups of $4$ pixels described previously.; and a baseline Hyperspectral algorithm SUNnSAL \citep{bioucas2010alternating}. Figure \ref{fig:AVIRIS_small} then gives plots of the coefficients associated to two minerals for three of the methods. Additional plots associated to four minerals and the four methods have been   Figure \ref{fig:AVIRIS:1} Appendix \ref{sec:AVIRIS:Extra}. Recall, by combining pixels the aim is to estimate more accurate coefficients than from denoising them independently. 
Indeed for the Hematite, Andradite and Polyhalite minerals, less noise is present for the total variation approach, alongside larger and brighter clusters. This is also in comparison to SUNnSAL, where the images for Andradite and Polyhalite from the total variation approach have less noise and brighter clusters. Although, we note that combining groups of pixels in this manner can cause the images to appear at a lower resolution. 
\begin{figure}[!h]
\centering
\includegraphics[width = 0.19\textwidth]{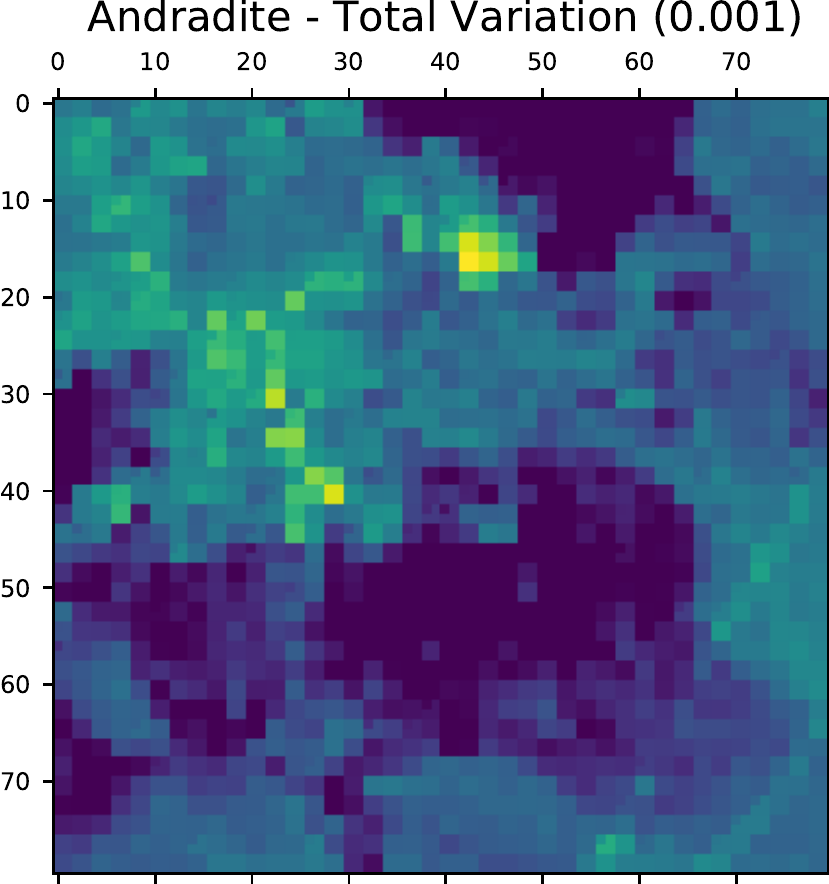}
\includegraphics[width = 0.19\textwidth]{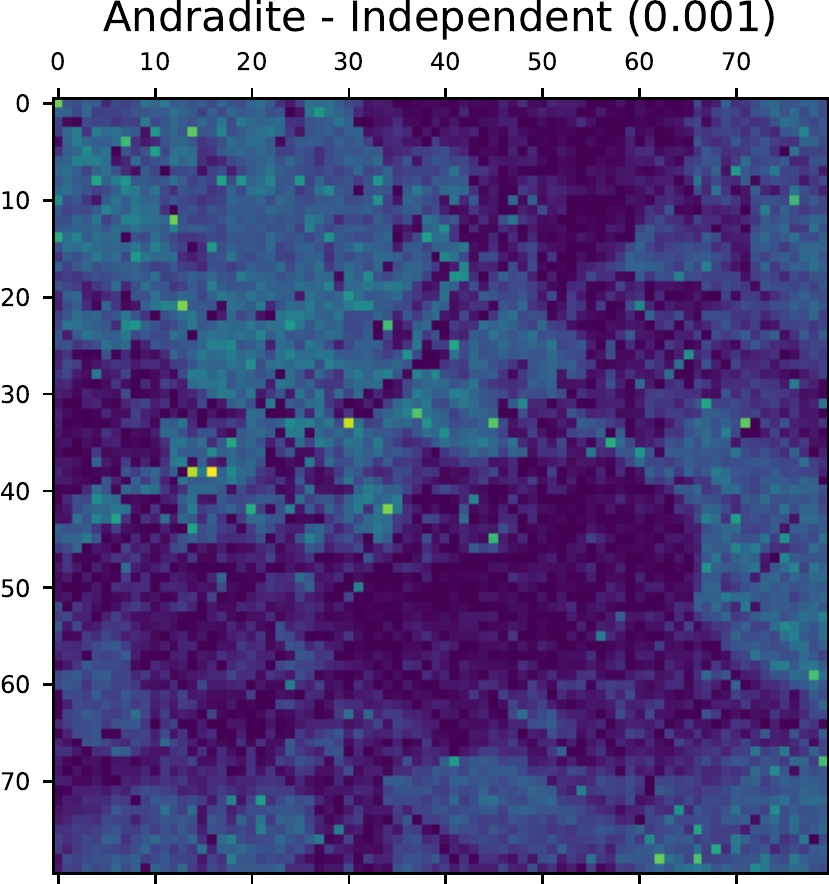}
\includegraphics[width = 0.19\textwidth]{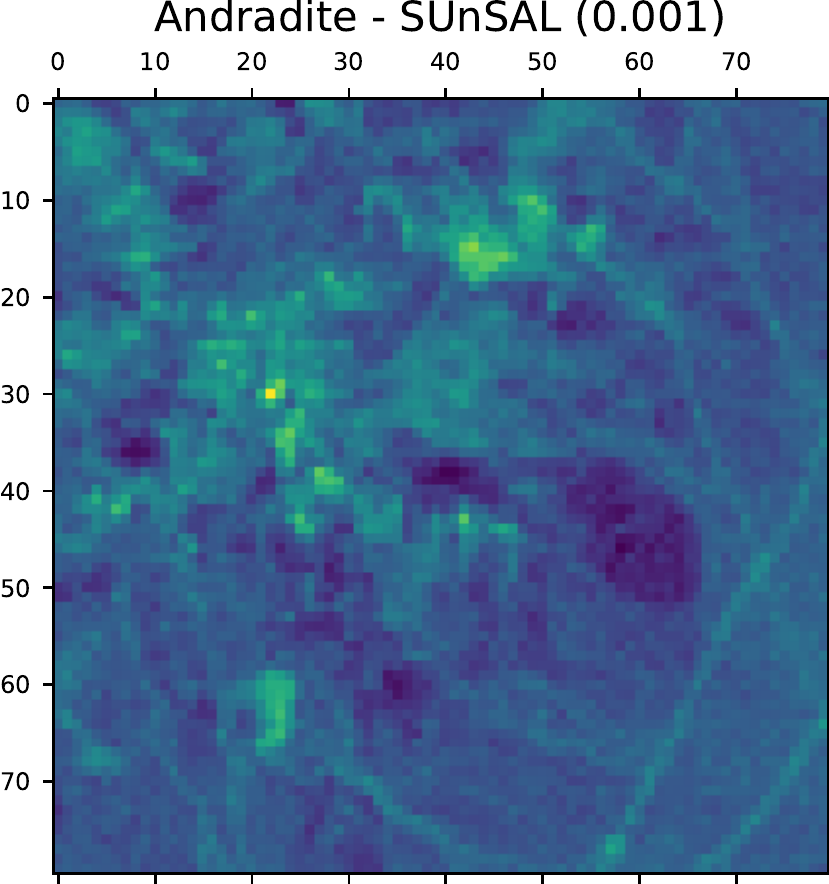}\\
\includegraphics[width = 0.19\textwidth]{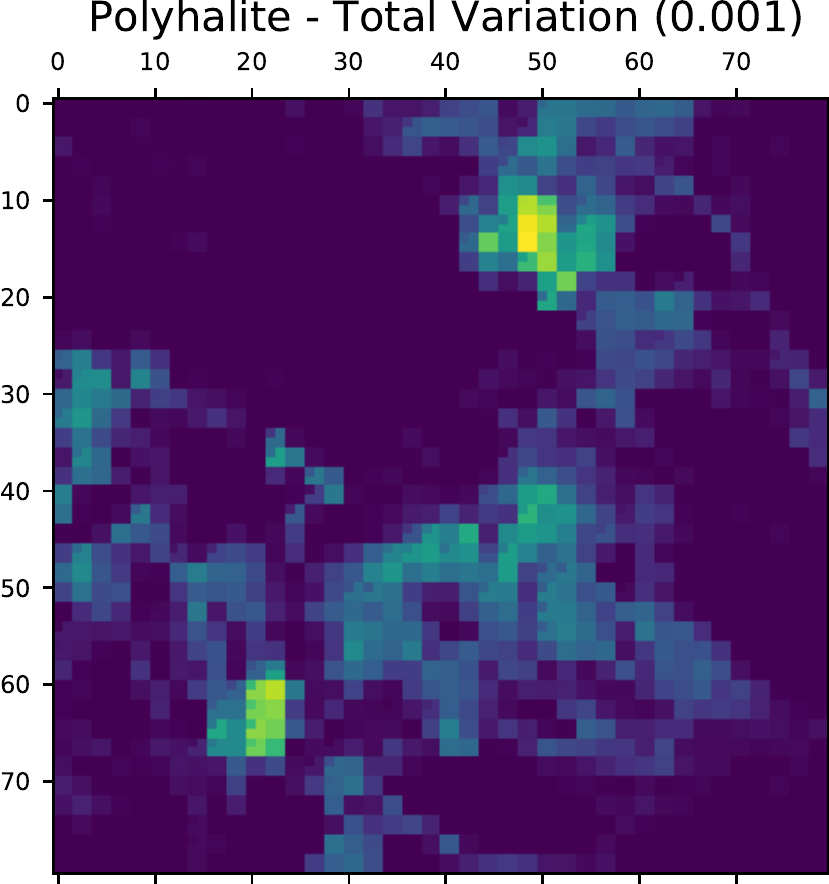}
\includegraphics[width = 0.19\textwidth]{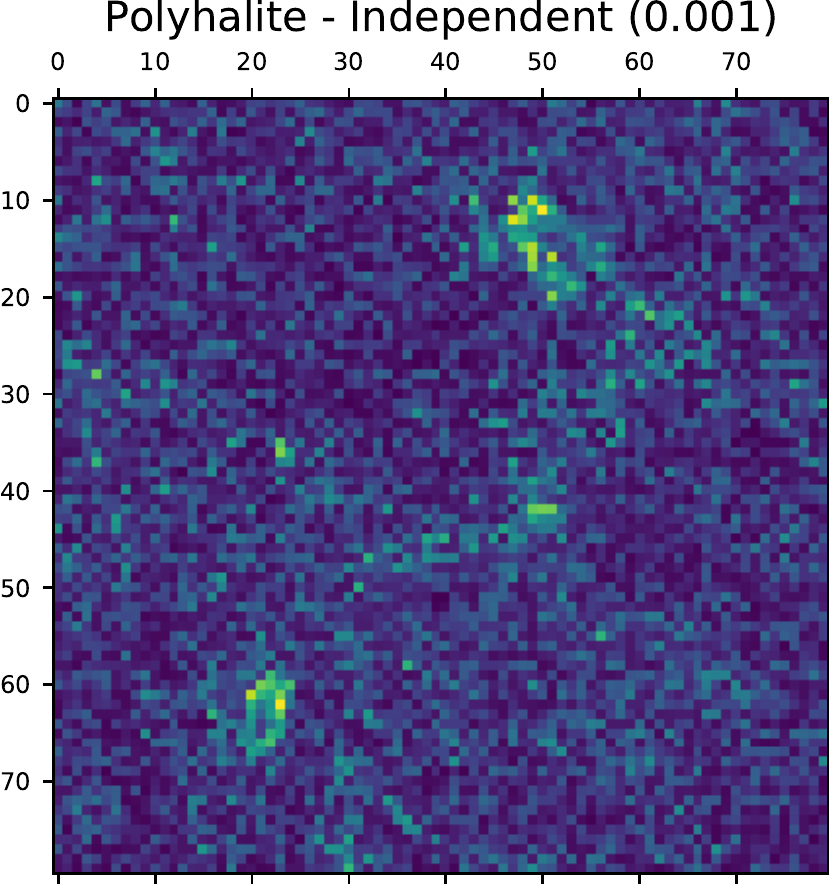}
\includegraphics[width = 0.19\textwidth]{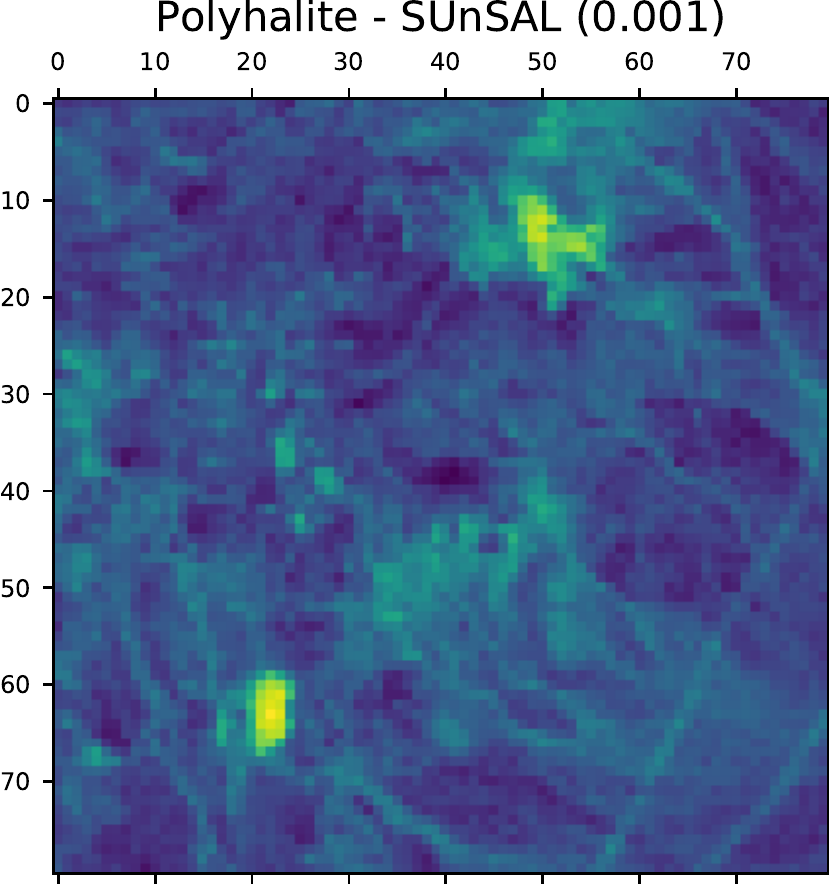}
\caption{Coefficients associated to the minerals Andradite (\textit{top}), and Polyhalite (\textit{bottom}). Methods are, \textit{left}: Total Variation Basis Pursuit Denoising applied to 2x2 pixel tiles with $\eta  = 0.001$; \textit{Middle}: Basis Pursuit Denoising applied independently to each pixel with $\eta = 0.001$. \textit{right}: SUNnSAL with regularisation of $0.001$. Yellow pixels indicate higher values. }
\label{fig:AVIRIS_small}
\vspace{-0.7cm}
\end{figure}

\section{Conclusion}
In this work we investigated total variation penalty methods to jointly learn a collection of sparse linear models over heterogeneous data. We assumed a graph-based sparse structure for the signals, where the signal at the root and the signal differences along edges are sparse. This setting differs from previous work on solving collections of sparse problems, which assume large overlapping supports between signals. We demonstrated (in noiseless and noisy settings) that statistical savings can be achieved over group Lasso methods as well as solving each problem independently, in addition to developing a distributed ADMM algorithm for solving the objective function in the noiseless case. 

The theoretical results currently suggest having identical matrices for non-root agents is more sample efficient over having different matrices (Table \ref{tab:samples:TVBP} and proof sketch Section \ref{sec:ProofSketch}). A natural direction is to investigate whether this is a limitation of the analysis, or fundamental to using the Total Variation penalty. Following this work, a distributed ADMM algorithm can also be developed for Total Variation Basis Pursuit Denoising optimisation problem \eqref{equ:DecentralisedDenoising}.

\section*{Acknowledgements}
D.R. is supported by the EPSRC and MRC through the
OxWaSP CDT programme (EP/L016710/1), and the London Mathematical Society ECF-1920-61. S.N. was supported by NSF DMS 1723128. P.R. was supported in part by the Alan Turing Institute under the EPSRC grant EP/N510129/1.

\bibliographystyle{plain}
\bibliography{References}

\appendix
\onecolumn

\section{Additional Material - Noiseless setting}
In this section we present additional material associated to the noiseless setting. 
Section \ref{sec:RefinedResults} presents refined sample complexity results when additional assumptions are placed on the graph $\widetilde{G}$ or the design matrices $\{A_v\}_{v \in V \backslash \{1\}}$.
Section \ref{sec:Simulations:Decentralised} presents details related to the distributed ADMM algorithm presented within Section \ref{sec:Noiseless:SampleComplexity} of the main body.  

\subsection{Refined Theoretical Results for Total Variation Basis Pursuit} 
\label{sec:RefinedResults}
In this section we present refined results for Total Variation Basis Pursuit. 
Section \ref{sec:Starlike} consider the case of a tree graph with different design matrices at each node. Section \ref{sec:MatchingNonRoot} considers the case of a tree graph with identical design matrices at non-root agents. 

\subsubsection{Total Variation Basis Pursuit with known Tree Graph }
\label{sec:Starlike}
Let us now consider the case where $G$ is a known tree graph. The sample complexity in this case is summarised within the follow Theorem. 
\begin{theorem}
\label{thm:Starlike}
Suppose $G$ is a tree graph, the signals $\{x^{\star}_{v}\}_{v \in V}$ are $(G,s,s^{\prime})$-sparse and matrices satisfy $A_v = \frac{1}{\sqrt{N_{v}}} \widetilde{A}_v$ where $\{\widetilde{A}_v\}_{v \in V}$ have i.i.d.\ sub-Gaussian entries. Fix $\epsilon > 0$. If
\begin{align*}
    N_{\text{Root}} & \gtrsim 
    \max\{s, n^2s^{\prime}  \} \big( \log(d) + \log(1/\epsilon) \big) 
    \text{ and }\\
    N_{\text{Non-root}} & \gtrsim \max\big\{ n ,\text{Deg}(V \backslash \{1\})^2 \text{Diam}(G)^2 \big\}s^{\prime} \big(   \log(d) + \log(n/\epsilon) \big)
\end{align*}
then with probability greater than $1-\epsilon$ the solution $\{x^{TVBP}_{v}\}_{v \in V}$ with $G = \widetilde{G}$ is unique and satisfies $x^{TVBP}_{v} = x^{\star}_{v}$ for all $v \in V$.
\end{theorem}
\begin{proof}
    See Appendix \ref{sec:proof:starlike}. 
\end{proof}
The sample complexity listed within the second row of Table \ref{tab:samples:TVBP} is then arrived at by simply summing up the above bound to arrive at $N_{\text{Total Samples}} = O(s + \max\{n^2 s^{\prime},n \text{Deg}(V \backslash \{1\})^2\text{Diam}(G)^2\}s^{\prime})$. We then note that Theorem \ref{thm:GeneralTree} within the main body of the work is a direct consequence of Theorem \ref{thm:Starlike}. We formally presented these steps within the following proof. 

\begin{proof}[Proof of Theorem \ref{thm:GeneralTree}]
    Let us consider a signal $\{x^{\star}_v\}_{v \in V}$ that is $(G,s,s^{\prime})-$sparse with respect to a general graph $G$. We then see that the signal is then $(\widetilde{G},s,\text{Diam}(G)s^{\prime})$-sparse when $\widetilde{G}$ is a star topology. Using Theorem \ref{thm:Starlike} there after (swapping $G$ for $\widetilde{G}$ and $s^{\prime}$ for $\text{Diam}(G) s^{\prime}$) and noting that both $\text{Diam}(\widetilde{G}) = 1$ and the degree of the non-root agents is $\text{Deg}(V \backslash \{1\}) =1 $ yields the result. 
\end{proof}

\subsubsection{Tree Graph with Identical Matrices for Non-root Agents}
\label{sec:MatchingNonRoot}
Let us now consider when non-root agents have the \emph{same} sensing matrices  i.e $A_v = A_w$ for $v,w \in V \backslash \{1\}$. The result is then summarised within the follow Theorem \ref{thm:IdenticalMatrices}. 
\begin{theorem}
\label{thm:IdenticalMatrices}
Suppose $G$ is a tree graph, $\{x^{\star}_{v}\}_{v \in V}$ is $(G,s,s^{\prime})$-sparse, $A_v = \frac{1}{\sqrt{N_{\text{Non-root}}}}A_{\text{Non-root}}$ for $v \not= 1$ and $A_1 = \frac{1}{\sqrt{N_1}} A_{Root}$. Assume that $A_{\text{Root}},A_{\text{Non-root}}$ each have i.i.d.\ sub-Gaussian entries. Fix $\epsilon >0$. If 
\begin{align*}
    N_{\text{Root}} & \gtrsim \max\{s,\text{Deg}(1)^2 s^{\prime} \} 
    \big( \log(d) + \log(1/\epsilon) \big) \text{ and } \\
    N_{\text{Non-root}} & \gtrsim \text{Deg}(1)^2 s^{\prime}\big( \log(d) + \log(1/\epsilon) \big)
\end{align*}
then with probability greater than $1-\epsilon$ the solution to TVBP with $\widetilde{G}=G$ is unique and satisfies $x^{TVBP}_{v} = x^{\star}_{v}$ for all $v \in V$. 
\end{theorem}
\begin{proof}
    See Appendix \ref{sec:EqualMatrices}. 
\end{proof}
The entry within the third row of Table \ref{tab:samples:TVBP} is then arrived at by summing up the above bound to arive at $N_{\text{Total Samples}} = O( s + n \text{Deg}(1)^2 s^{\prime})$.

\subsection{Distributed ADMM Algorithm}
\label{sec:Simulations:Decentralised}
In this section present the Distributed ADMM algorithm for solving the Total Variation Basis Pursuit problem \eqref{equ:TVBP}. We begin by reformulating the problem into an consensus optimisation form. 
Specifically, with $\Delta_{e} = x_v - x_w$ for $e = \{v,w\} \in E$, we consider 
\begin{align*}
& \min_{x_v, v \in V} \|x_1\|_1  + \sum_{e \in V} \|\Delta_e\|_1 \text{ subject to } \\
 & A_v x_v = Y_v \text{ for all } v \in V  \text{ and }  x_v - x_w = \Delta_e \text{ for all } e = \{v,w\} \in E.
\end{align*}
We then  propose the Alternating Direction Method of Multipliers (ADMM) to solve the above.  The key step is consider the augmented Lagrangian from dualizing the consensus constraint which, with $\|x\|_1 = \|x_1\|_1 + \sum_{e \in E} \|\Delta_{e}\|_1$, is  for $\rho > 0$
\begin{align*}
&\mathcal{L}_{\rho}(\{x_v\}_{v \in V}, \!\{\Delta_e\}_{e \in E}, \{\gamma_{e}\}_{e \in E} ) 
 \!\! = \!\!
\|x\|_1 +  \!\!\!
\sum_{e=\{v,w\} \in E} \frac{\rho}{2} \|x_v-x_w - \Delta_e\|_2 + \langle \gamma_e,x_v-x_w - \Delta_e\rangle.
\end{align*}
The ADMM algorithm then proceeds to minimise $\mathcal{L}_{\rho}$ with respect to $\{ x_v \}_{v \in V}$, then $\{\Delta_{e}\}_{e \in E}$, followed by a ascent step in the dual variable $\{\gamma_{e} \}_{e \in E}$. Full details of the ADMM updates have been given in Appendix \ref{ADMM:Updates}. Each step can be computed in closed form, expect for the update for $x_1$ which requires solving a basis pursuit problem with an $\ell_2$ term in the objective. This can be solved to a high precision efficiently by utilising a simple dual method, see \citep[Appendix B]{mota2011distributed}. The additional computational required by the root node in this case aligns with the framework we consider, since we assume the root node also has an additional number of samples $N_1$. 

The theoretical convergence guarantees of ADMM have gained much attention lately due to the wide applicability of ADMM to distributed optimisation problems \citep{boyd2011distributed,he20121,hong2017linear}. While a full investigation of the convergence guarantees of ADMM in this instance is outside the scope of this work, we note for convex objectives with proximal gradient steps computed exactly, ADMM has been shown to converge at worst case a polynomial rate of order $1/t$ \citep{he20121}. A number of works have shown linear convergence under additional assumptions which include full column rank on the constraints or strong convexity, which are not satisfied in our case \footnote{
The constraint dualised by ADMM, $x_v - x_w = \Delta_e$ for $e =\{v,w\} \in E$, can be denoted in terms the signed incident matrix of the graph. This is a linear constraint, but the signed incident matrix does not have full column rank.
}. Although, if one considers a proximal variant of ADMM with an additional smoothing term, linear convergence can be shown in the absence of the column rank constraint \citep{hong2017linear}. The convergence of ADMM can be sensitive hyperparameter choice $\rho$, which has motivated a number of adaptive schemes, see for instance \citep{he2000alternating}.

\section{Additional Material - Noisy Setting}
In this section we present additional material associated to the noisy setting within the main body of the manuscript. Section \ref{sec:AdditionalPlotsTVDenoising} presents details for experiments on simulated data. Section \ref{sec:AVIRIS:Extra} details for the experiments on real data.

\subsection{Details for Total Variation Basis Pursuit Denoising - Simulated Data}
\label{sec:AdditionalPlotsTVDenoising}

We now provide some details related to the \textbf{Simulated Data} experiments in Section \ref{sec:TVBasisPursuitDenoisingProblem} the manuscript. Group Lasso used best regularisation from between $[10^{-6},10^{-2}]$. Dirty model regularisation followed \citep{jalali2010dirty} with (in their notation) $5 \times 5$ ($\log$ -scale) grid search for $\lambda_g$ and $\lambda_b$ with $\lambda_g/\lambda_b \in [10^{-3},10]$, $\lambda_b = c \sqrt{7/200}$ and $c \in [10^{-2},10]$. Dirty model was fit using MALSAR \citep{zhou2011malsar}. The group Lasso variants used normalised matrices $A_{v}/\sqrt{N_v}$ and responses $y_{v}/\sqrt{N_{v}}$.
Total Variation Basis Pursuit Denoising parameter was $\eta = \sqrt{200 \times n} 0.1$. Each point and error bars from $5$ replications.

\subsection{Data Preparation and Experiment Parameters for AVIRIS Application}
\label{sec:AVIRIS:Extra}

In this section we present details associated to the application of Total Variation Basis Pursuit Denoising TVBPD (\ref{equ:DecentralisedDenoising}) to the AVIRIS Cuprite dataset. We begin with Figure \ref{fig:AVIRIS:MarkedArea}, which presents the sector of the AVIRIS Cuprite dataset used, as well as the 80 x 80 pixel subset portion sub-sampled for our experiment. We note each pixel in the dataset is associated to 224 spectral bands between 400 and 2500 nm and, in short, the objective is to decompose the spectrum of each pixel into a sparse linear combination known mineral spectra. The specific bandwidth presented in Figure \ref{fig:AVIRIS:MarkedArea} demonstrate that this area maybe a region of interest. Following \citep{iordache2011sparse,iordache2012total}, we construct a spectral library $ A_{\text{Lib}}$ by randomly sampling 240 mineral from the USGS library splib07 
\footnote{\url{https://crustal.usgs.gov/speclab/QueryAll07a.php}}. After cleaning the AVIRIS dataset and the library we are left with $N_v = 184$ spectral bands for each pixel $v \in V$, and thus, $A_v= A_{\text{Lib}} \in \mathbb{R}^{184 \times 240}$ and $y_v \in \mathbb{R}^{184}$. We now go on to describe more detail the experimental steps.
\begin{figure}[!h]
\includegraphics[width = 0.5\textwidth]{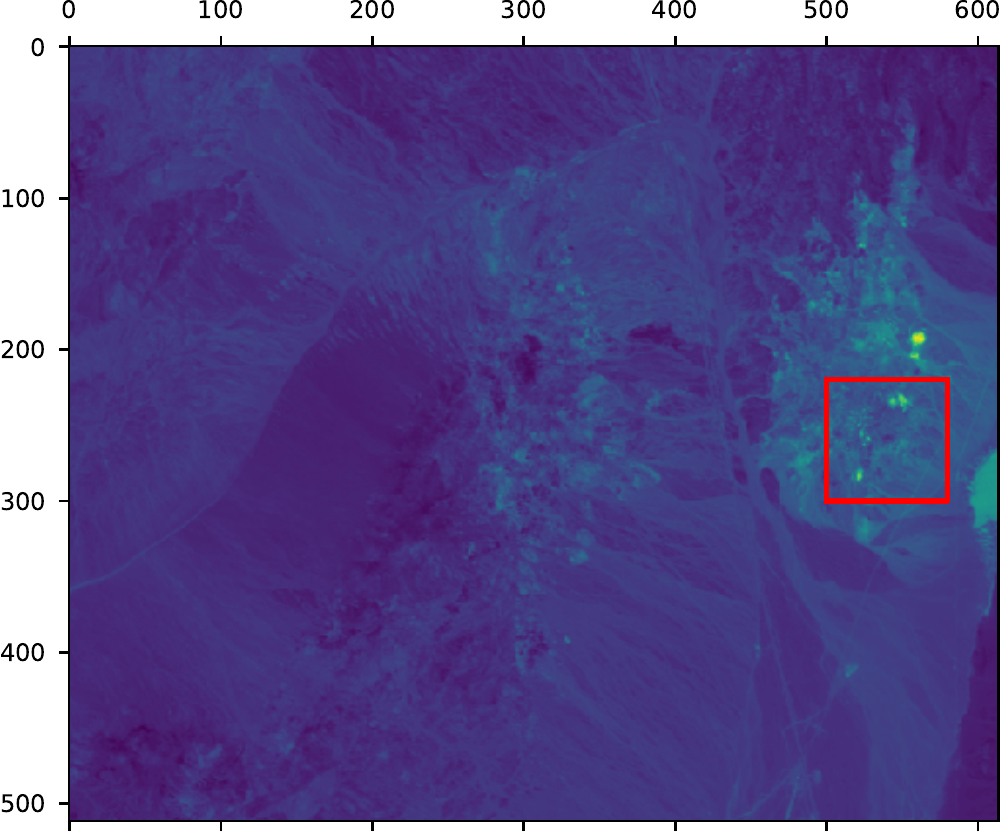}
\includegraphics[width = 0.42\textwidth]{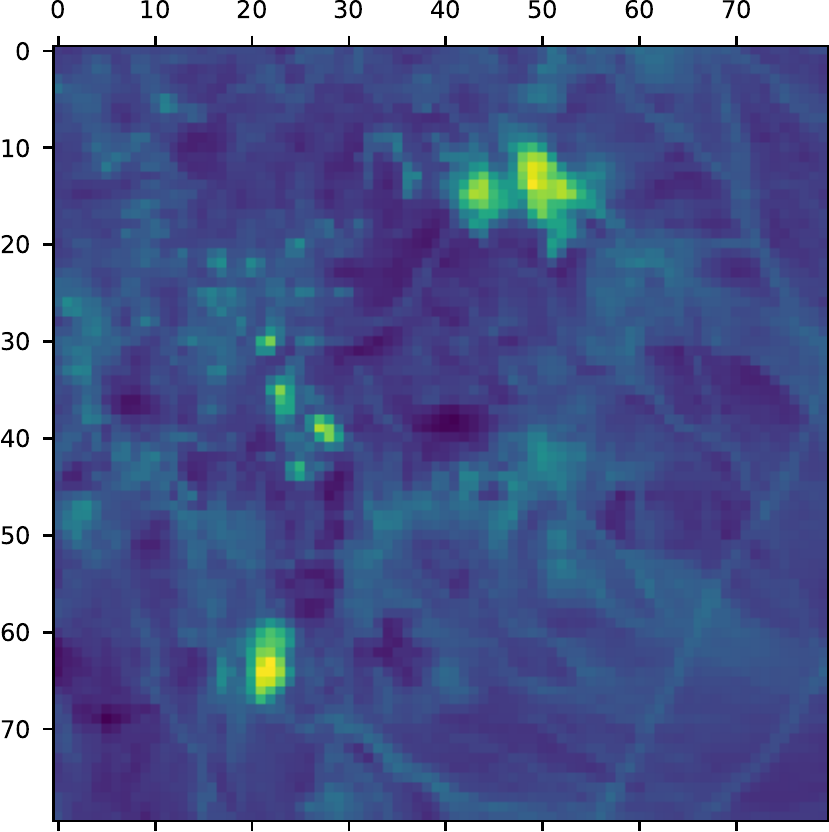}
\caption{\textit{Left}: Sector \texttt{f970619t01p02\_r04\_sc03.a.rfl} 
of AVIRIS data set at bandwidth of 557.07 nm. Red square indicates $80\times 80$ portion of the sector used as the data set. \textit{Right}: Red squared section zoomed in. }
\label{fig:AVIRIS:MarkedArea}
\end{figure}

\paragraph{Cleaning AVIRIS Cuprite Dataset} We followed \citep{iordache2012total} and removed the spectral bands 1-2, 105-115, 150-170 and 223-224, which are due to water absorption and low signal to noise. This would leave us with 188 spectral bands, although additional bands were removed due to large values within the USGS Library, see next paragraph. 
\paragraph{Sub-sampling USGS Library} We took a random sample of 240 minerals from splib07 library, that are specifically calibrated to the AVIRIS 1997 data set i.e.\ have been resampled at the appropriate bandwidths.  A number of the spectrum for the minerals were corrupted or had large reflectance values for particular wavelengths e.g. greater than $10^{34}$. We therefore restricted ourselves to minerals that had less than 10 corrupted wavelengths.  After sub-sampling, any wavelengths with a corrupted value (if it contained a value greater than $10$) were removed. This left us with 184 spectral bands. 
\paragraph{Algorithm Parameters} To apply Basis Pursuit Denoising independently to each pixel, we used the SPGL1 python package, which can be found at \url{https://pypi.org/project/spgl1/}. To solve the Total Variation Basis Pursuit Denoising problem \eqref{equ:DecentralisedDenoising}, we used the Alternating Direction Methods of Multiplers (ADMM) algorithm for $\ell_1$-problems in \citep{yang2011alternating}, specifically the inexact method (2.16). We applied this algorithm to the normalised data i.e.\ dividing by the matrix and response vector by the square root of the total number of samples (4 pixels $\times$ 184 spectral bands). We ran the algorithm for $500$ iterations with parameters (in the notation of \citep{yang2011alternating}) $\tau = 0.1$, $\beta = 2$, $\gamma = 0.1$ and $\delta=0.001$. We note that directly applying the SPGL1  python package to the Total Variation Basis Pursuit Denoising problem \eqref{equ:DecentralisedDenoising}, resulted in instabilities when choosing $\eta < 0.2$. We chose $\eta = 0.001$ for both independent Basis Pursuit Denoising case and the Total Variation Basis Pursuit Denoising \eqref{equ:DecentralisedDenoising}, following the regularisation choice in \citep{iordache2012total}. Meanwhile, the group Lasso was fit using scikit-learn with regularisation $0.001$, and the SUNnSAL algorithm \citep{bioucas2010alternating} with regularisation $0.001$ was applied using the python implementation which can be found at \url{https://github.com/Laadr/SUNSAL}. We note when using SUNnSAL it is common to perform a computationally expensive pre-processing step involving a non-convex objective, see \citep{iordache2011sparse,iordache2012total}. This was not performed in this case, as all of the other methods did not pre-process the data. 

\begin{figure}[!h]
\centering
\includegraphics[width = 0.24\textwidth]{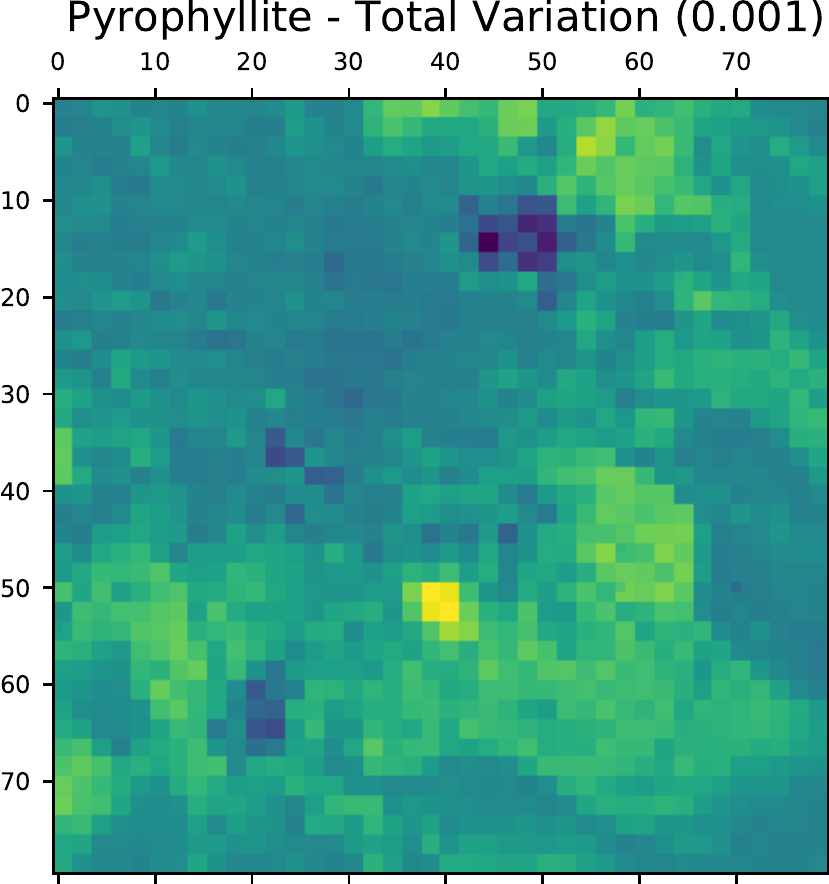}
\includegraphics[width = 0.24\textwidth]{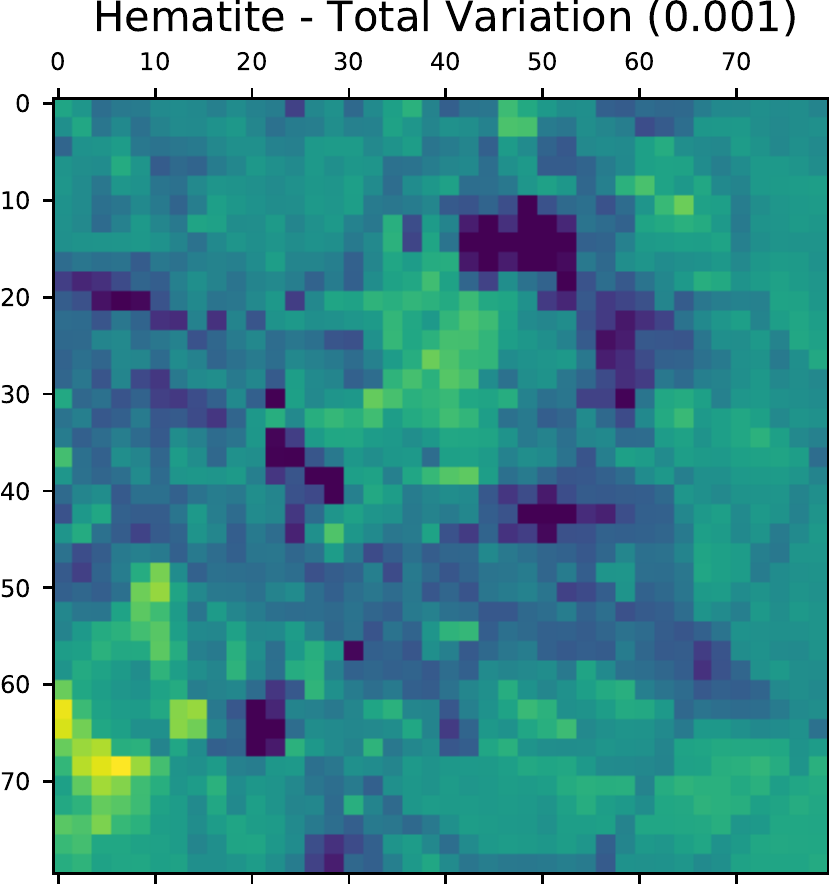}
\includegraphics[width = 0.24\textwidth]{Plots/Andradite_TV.pdf}
\includegraphics[width = 0.24\textwidth]{Plots/Polyhalite_TV.pdf}
\\
\includegraphics[width = 0.24\textwidth]{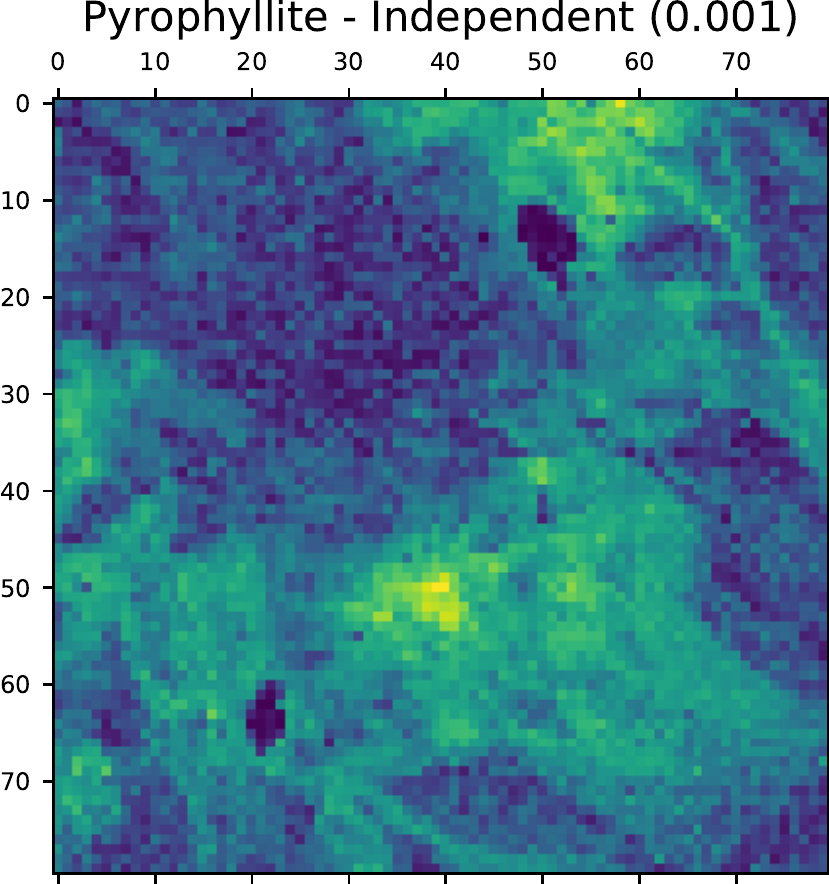}
\includegraphics[width = 0.24\textwidth]{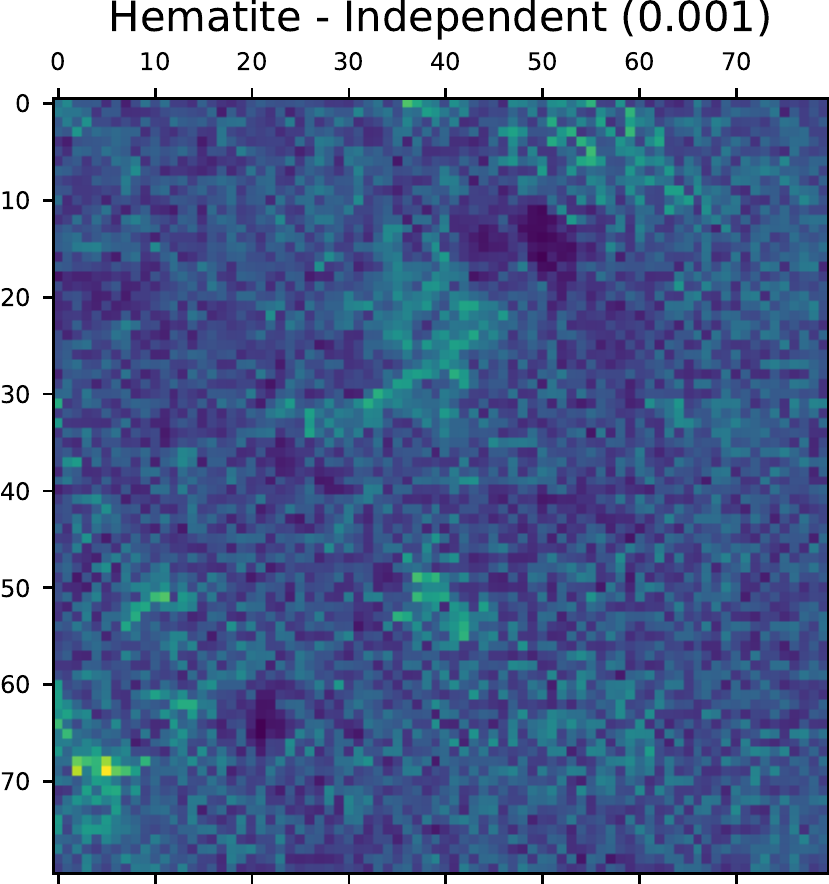}
\includegraphics[width = 0.24\textwidth]{Plots/Andradite_Indep_Noise1.pdf}
\includegraphics[width = 0.24\textwidth]{Plots/Polyhalite_Indep_Noise1.pdf}
\\
\includegraphics[width = 0.24\textwidth]{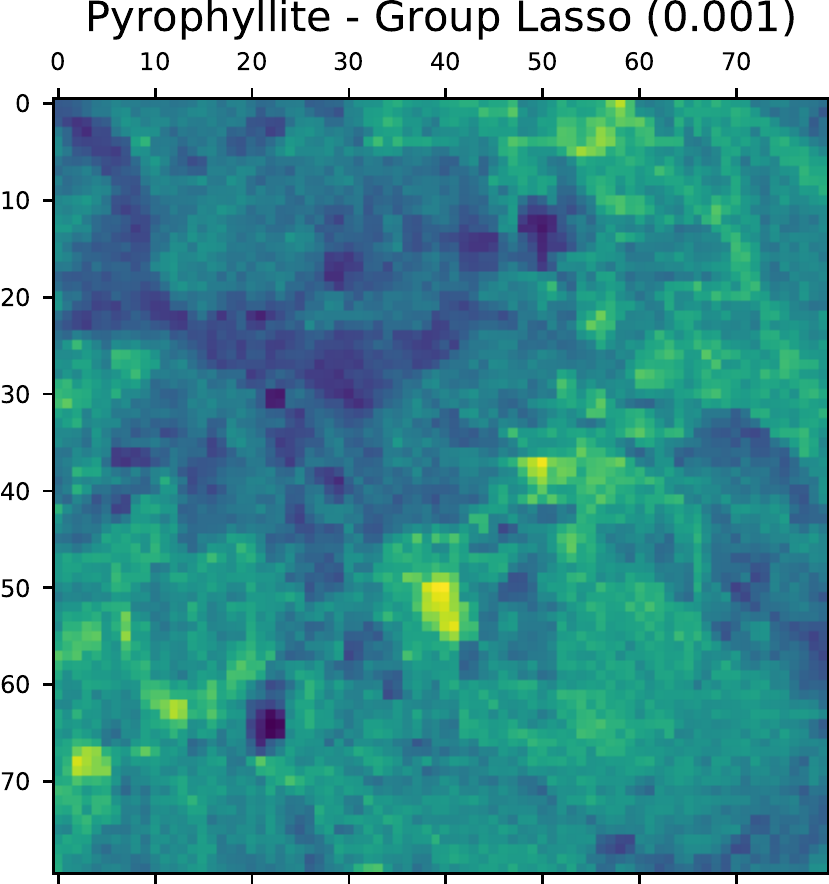}
\includegraphics[width = 0.24\textwidth]{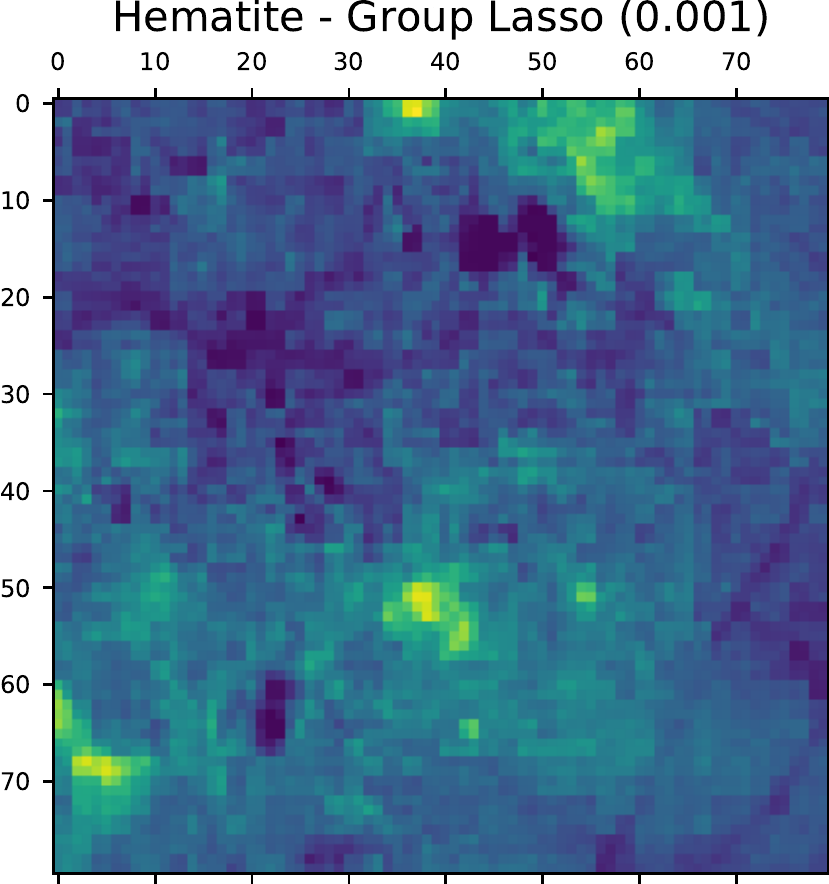}
\includegraphics[width = 0.24\textwidth]{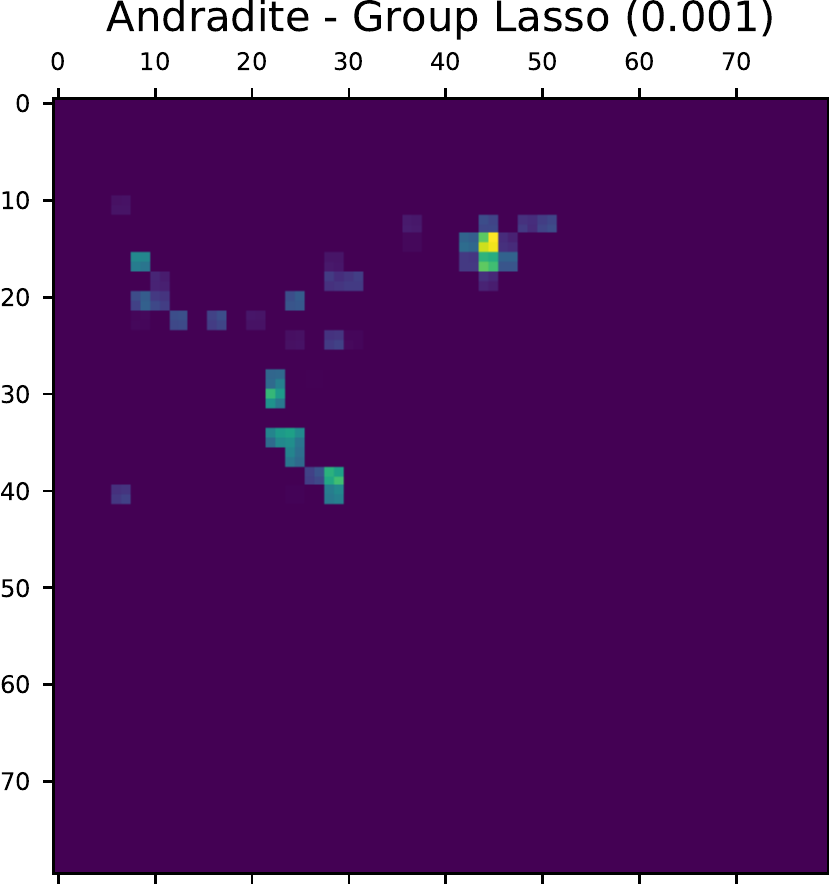}
\includegraphics[width = 0.24\textwidth]{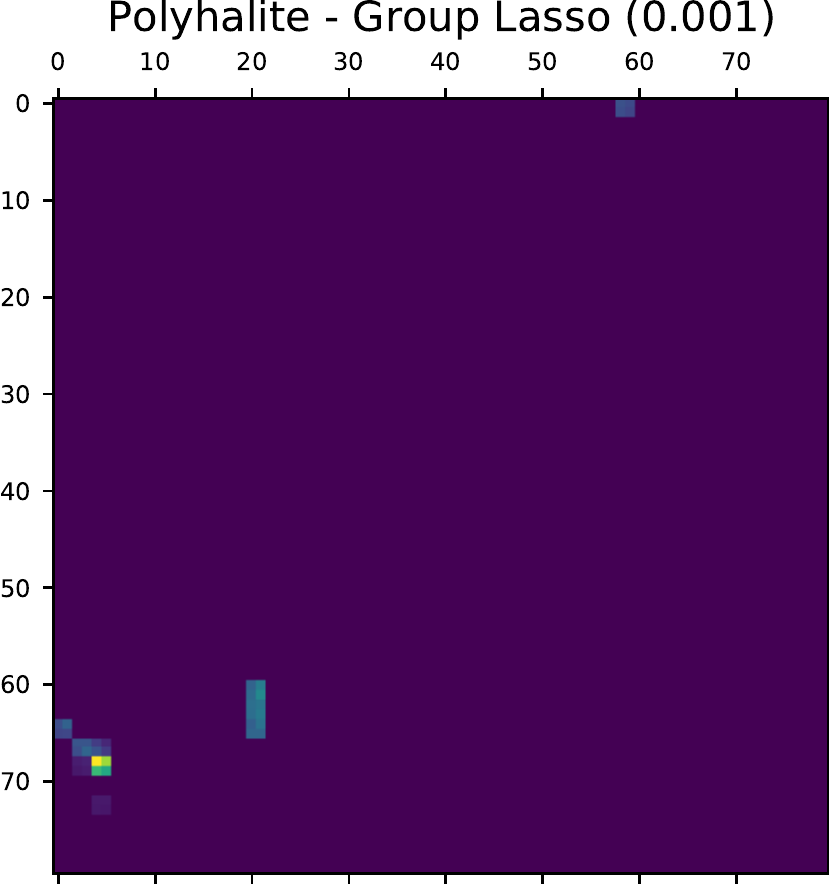}
\\
\includegraphics[width = 0.24\textwidth]{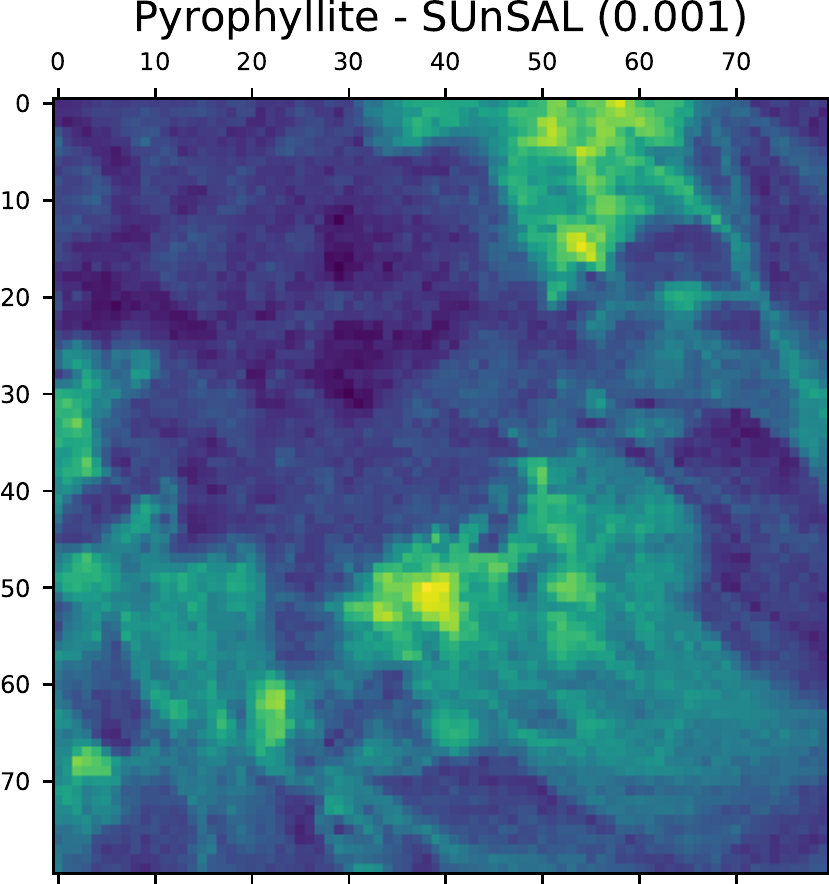}
\includegraphics[width = 0.24\textwidth]{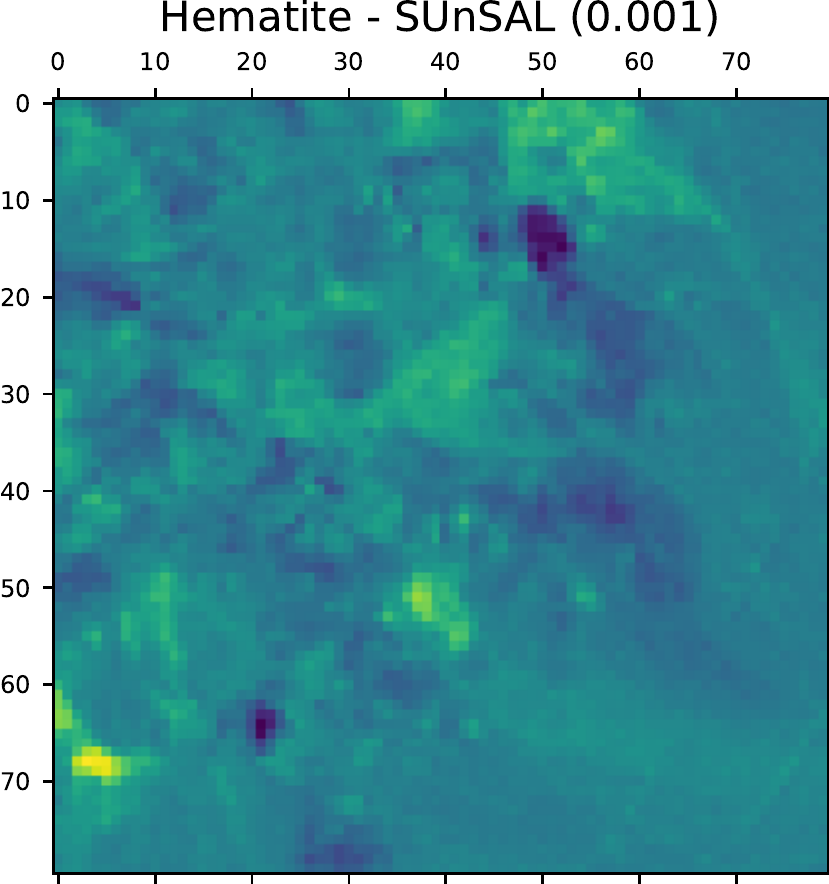}
\includegraphics[width = 0.24\textwidth]{Plots/Andradite_SUNSAL.pdf}
\includegraphics[width = 0.24\textwidth]{Plots/Polyhalite_SUNSAL.pdf}
\\
\caption{Coefficients associated to the mineral Pyrophyllite (\textit{Left}), Hematite (\textit{Left-Middle}), Andradite (\textit{right-middle}), and Polyhalite (\textit{Right}). Methods considered are: \textit{Top}: Total Variation Basis Pursuit Denoising applied to 2x2 pixels simultaneously with $\eta  = 0.001$; \textit{Middle-Top}: Basis Pursuit Denoising applied independently to each pixel with $\eta = 0.001$. \textit{Middle-Bottom}: group Lasso (jointly penalised all coefficients) applied to 2x2 pixels simultaneously with regularisation $0.001$.  \textit{Bottom}: SUNSAL with regularisation of $0.001$. Yellow pixels indicate higher values. }
\label{fig:AVIRIS:1}
\end{figure}

\section{Distributed ADMM Updates for Total Variation Basis Pursuit}
\label{ADMM:Updates}
In this section we more precisely describe the Distributed ADMM algorithm for fitting the Total Variation Basis Pursuit problem \eqref{equ:DecentralisedPursit:Reformulated}. We recall the consensus optimisation formulation of the Total Variation Basis Pursuit problem is as follows
\begin{align*}
& \min_{x_v, v \in V} \|x_1\|_1  + \sum_{e \in V} \|\Delta_e\|_1 \text{ subject to } \\
 & A_v x_v = Y_v \text{ for all } v \in V \\
& x_v - x_w = \Delta_e \text{ for all } e = \{v,w\} \in E.
\end{align*}
where  we consider the Augmented Lagrangian from dualizing the consensus constraint
\begin{align*}
&\mathcal{L}_{\rho}(\{x_v\}_{v \in V}, \{\Delta_e\}_{e \in E}, \{\gamma_{e}\}_{e \in E} ) 
 = 
\|x_1\|_1 \\
&  +  
\sum_{e=\{v,w\} \in E} \!\! \|\Delta_e\|_1 + 
\frac{\rho}{2} \|x_v-x_w - \Delta_e\|_2 + \langle \gamma_e,x_v-x_w - \Delta_e\rangle.
\end{align*}
Now the ADMM algorithm initialized at $\big(\{x_v^{1}\}_{v \in V}, \{\Delta^{1}_e\}_{e \in E}, \{\gamma_{e}^{1}\}_{e \in E}\big)$ then proceeds to update the iterates for $t \geq 1$ as 
\begin{align}
\label{equ:ADMMUpdates}
x^{t+1}_{v} & = \argmin_{x^{t}_v} \mathcal{L}_{\rho}(\{x^{t}_v\}_{v \in V}, \{\Delta^{t}_e\}_{e \in E}, \{\gamma^{t}_{e}\}_{e \in E} ) 
\text{ subject to } A_v x_v = Y_v  \text{ for } v \in V \\
\Delta^{t+1}_e & = \argmin_{\Delta^{t}_v} \mathcal{L}_{\rho}(\{x^{t+1}_v\}_{v \in V}, \{\Delta^{t}_e\}_{e \in E}, \{\gamma^{t}_{e}\}_{e \in E} ) 
\quad\quad\quad\quad\quad\quad\quad\quad \text{ for } e \in E \nonumber \\
\gamma^{t+1}_e & = \gamma_e^{t} + \rho \big( x_v - x_w - \Delta_e \big) 
\quad\quad\quad\quad\quad\quad\quad\quad\quad\quad\quad\quad\quad\quad\quad\quad\quad\,\,\,\,
 \text{ for } e \in E \nonumber 
\end{align}
We now set to show how each of the above updates can be implemented in a manner that respects the network topology due to the Augmented Lagrangian $\mathcal{L}_{\rho}$ decoupling across the network. These will be precisely described within the following sections. For clarity each update will be given its own subsection and the super script notation i.e.\ $x^{t}_{v}$ will be suppressed. 

\subsection{Updating $\{x_v\}$}
The updates for $\{x_v\}_{v \in V}$ take two different forms depending on whether $v$ is associated to the root node i.e.\ $v=1$ or otherwise. We begin with the case of a root note. 
\subsubsection{Root Node $x_1$}
The update for $x_1$ in the ADMM algorithm  \eqref{equ:ADMMUpdates} requires solving
\begin{align*}
	& \min_{x_1} \|x_1\|_1 + \sum_{e = (i,j) \in E : i = 1} \frac{\rho}{2} \|x_1 - x_j - \Delta_e \|_2^2 + \langle \gamma_e, x_1 - x_j - \Delta_e \rangle\\
	& \quad\quad\quad\quad\quad + 
	\sum_{e = (i,j) \in E : j = 1} \frac{\rho}{2} \|x_j - x_1 - \Delta_e \|_2^2 + \langle \gamma_e, x_j - x_1 - \Delta_e \rangle
	\\
	& \text{subject to }
	A_1 x_1 = y_1
\end{align*}
where we note the two summations in the objective arise from the orientation of the edges within the network.  This is then equivalent to considering solve a problem of the form
\begin{align}
\label{equ:LinPursuitWL2}
\min_{x} \|x\|_1 + \nu^{\top} x + c \|x\|_2^2 \text{ subject to }
Ax = b
\end{align}
with parameters $A = A_1$, $b=y_1$, $c = \text{Deg}(1) \frac{\rho}{2}$ where $\text{Deg}(1)$ is the degree of the root node $1$ and \\
$\nu = \sum_{e = (i,j) \in E : i = 1} \gamma_e - \rho(x_i + \Delta_e) + \sum_{e = (i,j) \in E : j = 1} - \gamma_e - \rho(x_j + \Delta_e) $.

To solve the problem \eqref{equ:LinPursuitWL2} we adopt the approach used in \citep[Appendix B]{mota2011distributed} to an optimisation problem of the same form. That is, we consider the dual problem
\begin{align*}
\max_{\lambda} \lambda^{\top}b + \sum_{i=1}^{p} \inf_{x_{i}} \big( |x_i| + u_i(\lambda) x_{i} + c x_i^2\big) 
\end{align*}
where the dual variable $\lambda \in \mathbb{R}^{n}$ and $u(\lambda) = \nu - A^{\top} \lambda$. 
The gradient of the above problem is then $b - A x(\lambda)$ where $x(\lambda) = (x(\lambda)_{1},\dots, x(\lambda)_{p})$ is constructed from the unique minimiser of $|x_i| + u_i(\lambda) x_{i} + c x_i^2$ for $i=1,\dots,p$ which is $x(\lambda)_i$. This can then be written in closed form as 
\begin{align*}
	x_i(\lambda) = 
	\begin{cases}
	0 & \text{ if } -1 \leq u_i(\lambda) \leq 1 \\
	-(u_i(\lambda) + 1)/2c & \text{ if } u_i(\lambda) < -1 \\
	-(u_i(\lambda) - 1)/2c & \text{ if } u_i(\lambda) > 1
	\end{cases}
\end{align*}
Given a solution $\lambda^{\star}$ the solution to the original problem is then $x(\lambda^{\star})$. To solve the Dual problem we use the Barzilai - Borwein algorithm \citep{raydan1997barzilai} with warm restarts using the dual variable from the previous iteration.

\subsubsection{Non-Root Node} 
In the case of $x_v$ which is not the root node i.e.\ $v \not=1$, we require solving the optimisation problem
\begin{align*}
	& \min_{x_v} \sum_{e = (i,j) \in E : i = v} \frac{\rho}{2} \|x_v - x_j - \Delta_e\|_2^2 + \langle \gamma_e, x_v - x_j - \Delta_e \rangle \\
	& \quad\quad\quad\quad + 
	\sum_{e = (i,j) \in E: j = v } \frac{\rho}{2} \|x_i - x_v - \Delta_e\|_2^2 + \langle \gamma_e, x_i - x_v - \Delta_e \rangle  \\
	& \text{subject to}
	\quad 
	A_v x_v = y_v
\end{align*}
This minimisation can be written in the form 
\begin{align}
\label{equ:NonPeanalisedzUpdate}
	&\min_{x} \|x\|_2^2 + \langle a,x\rangle \text{ subject to } \\
	& Ax = b \nonumber
\end{align}
with parameters $A = A_v$, $b= y_v$ and \\
$a = \frac{2}{\text{Deg}(v)} \Big( \big( \sum_{e \in \{i,j\}: i= v } - \Delta_e - x_j + \frac{\gamma_e}{\rho}\big) + \big( \sum_{e \in \{i,j\}: j = v} \Delta_e  - x_i  - \frac{\gamma_e}{\rho} \big) \Big)$. 
Since $\|x\|_2^2 + \langle a,x\rangle = \|x + \frac{a}{2} \|_2^2 - \frac{1}{2} \|a\|_2^2$,
This leads to the equivalent optimisation problem
\begin{align*}
	& \min_{u} \|u\|_2^2 \text{ subject to } \\
	& A u = b + A \frac{a}{2}.
\end{align*} 
This is exactly the least norm solution to a linear system, and is solved by $u = A^{\dagger}(b + A \frac{a}{2})$ where $A^{\dagger}$ is the Moore-Penrose pseudo-inverse. We then recover the solution to \eqref{equ:NonPeanalisedzUpdate} by setting  $x = A^{\dagger}(b + A \frac{a}{2}) - \frac{a}{2}$.

\subsection{Updating $\{\Delta_e\}_{e \in V}$}
For each edge $e = (i,j) \in E$ the updates require solving 
\begin{align*}
	\min_{\Delta_e} \|\Delta_e\|_1 + \frac{\rho}{2} \|x_i - x_j - \Delta_e\|_2^2 - \langle \gamma_e ,\Delta_e\rangle
\end{align*}
which is a equivalent to 
\begin{align*}
\min_{\Delta_e} \|\Delta_e\|_1 + \frac{\rho}{2} \|\Delta_e\|_2^2  - \langle \Delta_e, \gamma_e + z_i - z_j \rangle.
\end{align*}
This is a shrinkage step and thus the minimiser can be written as 
\begin{align*}
\Delta_e = \begin{cases}
0 & \text{ if } |\gamma_e + \rho(z_{i} - z_{j}) | < 1 \\
\frac{1}{\rho} \big( \gamma_e + \rho(z_i - z_{j}) - 1\big) & \text{ if }  \gamma_e + \rho(z_{i} - z_{j}) > 1 \\
\frac{1}{\rho} \big( \gamma_e + \rho(z_i - z_{j}) + 1\big) & \text{ if } \gamma_e + \rho(z_i - z_{j}) < -1
 \end{cases}
\end{align*}

\section{Proofs For Noiseless Case}
\label{sec:Proof:Noiseless}
In this section we present proofs for the results associated to Total Variation Basis Pursuit (TVBP) \eqref{equ:TVBP}.
This section is then structured as follows. Section \ref{sec:TechLemmas} presents technical lemmas associated to the Restricted Isometry Property of matrices. Section  \ref{sec:Proof:BP} introduces the Basis Pursuit problem. 
Section \ref{sec:Proof:reformulated} demonstrates how the Total Variation Basis Pursuit (TVBP) problem can be reformulated into a Basis Pursuit problem. Section \ref{sec:EqualMatrices} presents the proof of Theorem \ref{thm:IdenticalMatrices}. Section \ref{sec:proof:starlike} presents the proof of Theorem \ref{thm:Starlike}.

\subsection{Technical Lemmas for the Restricted Isometry Property}
\label{sec:TechLemmas}
Recall that a matrix $A \in \mathbb{R}^{N \times d}$ satisfies Restricted Isometry Property at level $k$ if there exists a constant $\delta_{k} \in [0,1)$ such that for any $k$-sparse vector $x \in \mathbb{R}^{p}, \|x\|_0 \leq k$ we have
\begin{align*}
    (1-\delta_k) \|x\|_2^2 \leq \|Ax\|_2^2 \leq (1+\delta_k)\|x\|_2^2 
\end{align*}
Now, Theorem 9.2 from \citep{foucart2013invitation}  demonstrates that a sub-Gaussian matrix can satisfy the Restricted Isometry Property in high probability provided the sample size is sufficiently large. This is presented within the following theorem. 
\begin{theorem}
\label{thm:RIPConc}
Let $A \in \mathbb{R}^{N \times p}$ be sub-Gaussian matrix with independent and identically distributed entries. Then there exists a constant $C > 0$ (depending on sub-Gaussian parameters $\beta$ and $\kappa$) such that the Restricted Isometry Constant of $A/\sqrt{N}$ satisfies $\delta_{k} \leq \delta$ with probability atleast $1-\epsilon$ provided 
\begin{align*}
    N \geq C \delta^{-2} \big( k \log(ep/k) + \log(\epsilon/2) \big).
\end{align*}
\end{theorem}
We will also make use of the following assumption. 
\begin{proposition}
\label{prop:DisjointRIP}
Let $\mathbf{u},\mathbf{v} \in \mathbb{R}^{p}$ be vectors such that $\|\mathbf{u}\|_0 \leq s$ and $\|\mathbf{v}\|_0 \leq k$, and matrix $A \in \mathbb{R}^{N \times p}$ satisfy Restricted Isometry Property up to $s+k$ with constant $\delta_{s+k}$. If the support of the vectors is disjoint $\text{Supp}(v) \cap \text{Supp}(u) = \emptyset$ then 
\begin{align*}
    |\langle A\mathbf{u},A \mathbf{v}\rangle| \leq 
    \delta_{s+k} \|\mathbf{u}\| \|\mathbf{v}\|_2
\end{align*}
\end{proposition}
We also make use of the following lemma which is can be found as Lemma 11 in the Supplementary material of \citep{loh2011high}. It will be useful to denote the $\ell_1$ ball as $\mathbb{B}_1(r) = \{x \in \mathbb{R}^{d} : \|x\|_1 \leq r \}$ and similarly for $\ell_2$ and $\ell_0$ balls as  $\mathbb{B}_{2}(r),\mathbb{B}_{0}(r)$, respectively.
\begin{lemma}
\label{lem:balls}
For any integer $s \geq 1$, we have 
\begin{align*}
    \mathbb{B}_1(\sqrt{s}) \cap \mathbb{B}_2(1) \subseteq  
    \mathrm{cl}\Big\{ \mathrm{conv}\Big\{\mathbb{B}_0(s) \cap \mathbb{B}_2\big(3\big) \Big\} \Big\}
\end{align*}
where $\mathrm{cl}$ and $\mathrm{conv}$ denote the topological closure and convex hull, respectively.
\end{lemma}

\subsection{Basis Pursuit}
\label{sec:Proof:BP}
For completeness we recall some fundamental properties of the Basis Pursuit problem. Consider a sparse signal $x^{\star}$, sensing matrix $A \in \mathbb{R}^{N \times d}$ and response $y  = Ax^{\star}$. The Basis Pursuit problem is then defined as 
\begin{align}
\label{equ:BasisPursuit:app}
    \|x\|_1  & \text{ subject to} \\
    & Ax = y.
    \nonumber
\end{align}
Denote the solution to the above as $x^{\text{BP}}$. Suppose that $x^{\star}$ is supported on the set $S \subset \{1,\dots,d\}$. Then it is well know \citep{foucart2013invitation} that $x^{\text{BP}}$ is both unique and satisfies $x^{\text{BP}}= x^{\star}$ if and only if $A$ satisfies the \emph{restricted null space property} with respect to $S$, that is, 
\begin{align}
\label{equ:RestrictedNullSpace:app}
    2 \|(x)_{S}\|_1 \leq  \|x\|_1 \text{ for any } x \in \text{Ker}(A) \backslash \{0\}.
\end{align}
Following the proof sketch in the manuscript, we now proceed to reformulate the Total Variation Pursuit Problem \eqref{equ:TVBP} into a Basis Pursuit problem \eqref{equ:BasisPursuit:app}.

\subsection{Reformulating Total Variation Basis Pursuit into Basis Pursuit}
\label{sec:Proof:reformulated}
We now describe the steps in reformulating the Total Variation Basis Pursuit Denoising problem \eqref{equ:TVBP} into a Basis Pursuit problem \eqref{equ:BasisPursuit:app}. We begin by introducing some notation. For node $v \in V$, denote the set of edges making a path from node $v$ to the root node $1$ by 
$\pi(v) = \{ \{v , w_1 \} , \{ w_1 , w_2 \},\dots,\{w_{k_v -1},w_{k_v}\},\{w_{k_v},1\} \} \subseteq E $  
where $k_v \geq 1$ is the number of intermediate edges. In the case $k_v =0$ there is only a single edge and so we write $\pi(v) = \{v,1\} \in E$. Meanwhile, for the root node itself $v=1$ we simply have the singleton $\pi(v)= \pi(1) = \{1\}$, and thus, we have the root node included $v \in \pi(v)$ but no edges i.e.\ $e \notin \pi(v)$ for any $e \in E$. For each edge $e = \{v,w\} \in E$ the difference is denoted $\Delta_e = x_v - x_w$, and so the vector associated to any node $x_v$ can be decomposed into the root node $x_1$ plus the differences along the path $x_v = x_1 + \sum_{e \in \pi(v)}  \Delta_{e}$. Similarly, the signal associated to each node $x^{\star}_v$ can be decomposed into differences of signals associated to the edges $e = \{v,w\} \in E$ with  $\Delta^{\star}_{e}  = x_v^{\star} - x_{w}^{\star}$.

With this notation we can then reformulate  \eqref{equ:TVBP} in terms of $x_1$ and $\{\Delta_e \}_{e \in E}$ as follows 
\begin{align}
\label{equ:DecentralisedPursit:Reformulated}
\min_{x_1, \{\Delta_e\}_{e \in E} } \|x_1\|_1 & + \sum_{e = (v,w) \in E } \|\Delta_e \|_1 \text{ subject to }  \\
 A_v\Big( x_1 & + \sum_{e \in \pi(v) } \Delta_e \Big) = y_v \quad \forall v \in V.
\nonumber
\end{align}
Optimisation problem \eqref{equ:DecentralisedPursit:Reformulated} is now in terms of a standard basis pursuit problem \eqref{equ:BasisPursuit} with, if edges are labeled with integers, the vector $x \! = \! (x_1,\Delta_1,\dots,\Delta_{|E|})$, true signal $x^{\star} \! = \! (x_1^{\star},\Delta_1^{\star},\dots,\Delta^{\star}_{|E|})$, and a matrix $A$. 
To be precise, the matrix $A$ can be defined in terms of blocks 
$A = (H_1^{\top} , \dots , H_n^{\top} )^{\top} \! \in \! \mathbb{R}^{ (\sum_{v \in V} N_{v}) \times  n p }$ with each block  
$H_v \! \in \! \mathbb{R}^{N_v \times n p}$ for $v \! \in \! V$. Each block then defined as 
$
H_v \! = \! \begin{pmatrix}
H_{v1},H_{v2},\dots,H_{vn}
\end{pmatrix}
$
with, for $i \! = \! 1,\dots,n$, the matrix $H_{vi}  \! = \! A_v$ if node $i$ is included on the path going from node $v$ to the root node $1$ i.e.\ $i \! \in \! \pi(v)$, and $0$ otherwise. 

The signal associated to the reformulated problem \eqref{equ:DecentralisedPursit:Reformulated} remains sparse and is supported on a set $S$ with a particular structure due to encoding the sparsity of the differences $\{\Delta^{\star}_{e} \}_{e \in E}$. Specifically, the set $S$ contains the entries from $\{1,\dots,p\}$ aligned with $S_1$ and, labeling the edges $e \in E$ with the integers $i =1,\dots,|E|$, the elements from $\{1,\dots,p\}$ associated to $S_e$ offset by $i \times p$.  Now that \eqref{equ:TVBP} is in terms of a Basis Pursuit problem, its success relies on the matrix $A$ satisfying the Restricted Null Space Property \eqref{equ:RestrictedNullSpace:app} with respect to the sparsity set $S$. This can be rewritten in terms of $x_1$ and $\{\Delta_{e}\}_{e \in E}$ as follows 
\begin{align}
\label{equ:RestrictedNullSpace:Reformulated}
    \|(x_1)_{S_1} \|_1 
    + 
    \sum_{e \in E} \|(\Delta_{e})_{S_e}\|_1
    & \leq 
    \frac{1}{2}
    \Big( 
    \|x_1\|_1 
    + 
    \sum_{e \in E} \|\Delta_{e} \|_1
    \Big) \\
    &\quad\quad\quad\quad\quad\quad \text{ for }
    x_1 + \sum_{e \in \pi(v) } \Delta_e  \in \text{Ker}(A_{v}) \backslash \{0\} 
    \text{ for } v \in V.
    \nonumber
\end{align}
From now on we will let $S_1$ denote the largest $s$ entries of $x_1$, and for $e \in E$ the set $S_{e}$ as the largest $s^{\prime}$ entries of $\Delta_{e}$. We that we begin with the proof of Theorem \ref{thm:IdenticalMatrices} in Section \ref{sec:EqualMatrices}, as the analysis of matching non-root matrices $A_v = A_w$ for $v,w \not= 1$ is simpler. This will then be followed by the proof of Theorem \ref{thm:Starlike} in Section \ref{sec:proof:starlike}.

\subsection{Proof of Theorem \ref{thm:IdenticalMatrices}}
\label{sec:EqualMatrices}
We now provide the proof of Theorem \ref{thm:IdenticalMatrices}. We begin with the following lemma which follows a standard shelling technique, see for instance \citep{foucart2013invitation}. 
\begin{lemma}
\label{lem:RIPtoL2}
Suppose the matrix $B \in \mathbb{R}^{N \times d}$ satisfies Restricted Isometry Property up-to sparsity level $d \geq k > 0$ with constant $\delta_{k} \in [0,1)$. If $x \in \text{Ker}(B)\backslash \{0\}$
 then for any $U \subseteq \{1,\dots,d\}$ such that $|U| = k$ we have 
 \begin{align*}
     \|(x)_{U} \|_2 \leq 
     \frac{\delta_{2k}}{1-\delta_{k}} \frac{1}{\sqrt{k}} \|x\|_1
 \end{align*}
\end{lemma}
\begin{proof}[Proof of Lemma \ref{lem:RIPtoL2}]
Noting that $Bx_{U}  = - B x_{U^{c}}$ and using the restricted isometry property of $B$, we can bound 
\begin{align*}
    (1 - \delta_{k})\|(x)_{U}\|_2^2 
    \leq \|B(x)_{U}\|_2^2 = -\langle B(x)_{U},B(x)_{U^c}\rangle.
\end{align*}
Now decompose $U^{c}$ into disjoint sets $\{U_{j}\}_{j=1,2,\dots}$ of size $k$ so that $U^{C} = U_1 \cup U_{2} \cup \dots$ . Structure the sets so that $U_{1}$ is the largest $k$ entries of $(x)_{U^{c}}$, $U_{2}$ is the largest $k$ entries of $(x)_{(U \cup U_1)^{c}}$ and so on. Note that for $j=2,\dots,$ that $\|(x)_{B_{j}}\|_2 \leq \sqrt{k} \|(x)_{B_{j}}\|_\infty \leq \frac{1}{\sqrt{k}} \|(x)_{B_{j-1}}\|_1$. While $\|(x)_{B_{1}}\|_2 \leq \frac{1}{\sqrt{k}} \|x_{U}\|_1$. Returning to the equation above this allows us to bound with Proposition \ref{prop:DisjointRIP}, 
\begin{align*}
    (1 - \delta_{k})\|(x)_{U}\|_2^2 
    & \leq \delta_{2k} \|(x)_{U}\|_2 \sum_{j \geq 1} \|(x)_{B_j}\|_2 \\
    & \leq 
    \frac{\delta_{2k}}{\sqrt{k}}\|(x)_{U}\|_2 \big( \|(x)_{U}\|_1 + \sum_{j \geq 1} \|(x)_{B_{j-1}}\|_1 \big) \\
    & = \frac{\delta_{2k}}{\sqrt{k}}\|(x)_{U}\|_2 \|x\|_1
\end{align*}
Dividing both sides by $(1-\delta_{k}) \|(x)_{U}\|_2$ yields the result.
\end{proof}
We now proceed to the proof of Theorem \ref{thm:IdenticalMatrices}.

\begin{proof}[Proof of Theorem \ref{thm:IdenticalMatrices}]
Recall that it is sufficient to demonstrate that the restricted null space property for the reformulated problem \eqref{equ:RestrictedNullSpace:Reformulated} is satisfied with high probability. 
In this case we then have  $A_{v} = \frac{1}{\sqrt{N_{\text{Non-root}}}}A_{\text{Non-Root}}$ for $v \in V \backslash \{0\}$, and $A_{1} = \frac{1}{\sqrt{N_{\text{Root}}}}A_{\text{Root}}$.  We then begin by assuming that $\frac{1}{\sqrt{N_{\text{Non-root}}}}A_{\text{Non-Root}}$ satisfies Restricted Isometry Property up-to sparsity level $k^{\prime}$ with constant $\delta^{\text{Non-root}}_{k^{\prime}} \in [0,1)$ and $\frac{1}{\sqrt{N_{\text{Root}}}}A_{\text{Root}}$ satisfies Restricted Isometry Property up-to sparsity level $k$ with constant $\delta^{\text{Root}}_{k} \in [0,1)$. Let us also suppose that $k \geq k^{\prime}$. We will then return to satisfying this condition with high-probability at the end of the proof.

The proof proceeds by bounding $\|(x_1)_{S_1}\|_1, \|(\Delta)_{S_{e}}\|_1$ by using that $x=(x_1,\Delta_1,\dots,\Delta_{|E|}) \in \text{Ker}(A) \backslash \{0\}$. We split into three cases: the root note $\|(x_1)_{S_1}\|_1$; the term $\|(\Delta)_{S_{e}}\|_1$ for edges $e = (v,w)$ not directly connected to the root  $v,w \not=1$; and the term $\|(\Delta)_{S_{e}}\|_1$ for edges $e=(v,w)$ joined to the root $v=1$ or $w=1$. Each is now considered in its own paragraph, with the combination in a fourth paragraph. 

\textbf{Root Node} Note that $x_1 \in \text{Ker}(A_{\text{Root}}) \backslash \{0\}$ therefore from  Lemma \ref{lem:RIPtoL2} and the inequality $\|(x_1)_{S_{1}}\|_1 \leq \sqrt{s} \|(x_1)_{S_1}\|_2 $  we get the upper bound 
\begin{align*}
    \|(x_1)_{S_1}\|_1 \leq 
    \frac{\delta^{\text{Root}}_{2s}}{1-\delta_{s}^{\text{Root}}} \|x_1\|_1,
\end{align*}
as required.

\textbf{Edges not connected to the root}
For any edge $\widetilde{e} = (v,w) \in E$ not connected to the root node so $v,w \not= 1$, note we have $\Delta_{\widetilde{e}} \in \text{Ker}(A_{\text{Non-root}}) \backslash \{0\}$.  To see this, each vector is in the same null-space $x_1 + \sum_{e \in \pi(v) } \Delta_e,x_1 + \sum_{e \in \pi(w) } \Delta_e \in \text{Ker}(A_{\text{Non-root}}) \backslash \{0\} $ so their difference is also. That is if  $w$ is the furthest from the root node we can write 
$x_1 + \sum_{e \in \pi(w) } \Delta_{e} = x_1 + \sum_{e \in \pi(v) } \Delta_{e} + \Delta_{\widetilde{e}}$ and therefore
\begin{align*}
    \Delta_{\widetilde{e}} = \Big(x_1 + \sum_{e \in \pi(w) } \Delta_{e}\Big)
    - 
    \Big(x_1 + \sum_{e \in \pi(v) } \Delta_{e}\Big)
    \in \text{Ker}(A_{\text{Non-root}}). 
\end{align*}
An identical calculation can be done for the case of when $v$ is furthest from the root node. 
Therefore  Lemma \ref{lem:RIPtoL2}  yields the upper bound
\begin{align*}
    \|(\Delta_{\widetilde{e}})_{S_{\widetilde{e}}}\|_1 \leq 
    \frac{\delta^{\text{Non-root}}_{2 s^{\prime}}}{1-\delta^{\text{Non-root}}_{s^{\prime}}}
    \|\Delta_{\widetilde{e}}\|_1,
\end{align*}
as required.

\textbf{Edges connected to the root}
For edges connecting the root node so $\widetilde{e}  = (v,w) \in E$ such that $v = 1$ or $w = 1$, begin by  adding and subtracting $(x_1)_{S_{\widetilde{e}}}$ to decompose 
\begin{align*}
    \|(\Delta_{\widetilde{e}})_{S_{\widetilde{e}}}\|_1 
    \leq 
    \|(x_1)_{S_{\widetilde{e}}}\|_1
    + \|(x_1 + \Delta_{\widetilde{e}})_{S_{\widetilde{e}}}\|_1.
\end{align*}
Now if $|S_{\widetilde{e}}| \leq k^{\prime}$ we have from Lemma \ref{lem:RIPtoL2} the inequality $\|(x_1)_{S_{\widetilde{e}}}\|_1 \leq \sqrt{k^{\prime}} \|x_{S_{\widetilde{e}}}\|_2 \leq \frac{\delta^{\text{Root}}_{2k^{\prime}}}{1-\delta^{Root}_{k^{\prime}}} \|x_1\|_1$ since $x_1 \in \text{Ker}(A_\text{Root})\backslash \{0\}$. As well as from the fact that $x_1 + \Delta_{\widetilde{e}}  \in \text{Ker}(A_{\text{Non-root}})\backslash \{0\}$ the upper bound $\|(x_1 + \Delta_{\widetilde{e}})_{S_{\widetilde{e}}}\|_1 \leq \frac{\delta^{\text{Non-root}}_{2k^{\prime}}}{1-\delta^{\text{Non-root}}_{k^{\prime}}} \|x_1 + \Delta_{\widetilde{e}}\|_1$. Combining these bounds we get 
\begin{align*}
    \|(\Delta_{\widetilde{e}})_{S_{\widetilde{e}}}\|_1 
    & \leq  \frac{\delta^{\text{Root}}_{2k^{\prime}}}{1-\delta^{\text{Root}}_{k^{\prime}}} \|x_1\|_1
    + 
    \frac{\delta^{\text{Non-root}}_{2k^{\prime}}}{1-\delta^{\text{Non-root}}_{k^{\prime}}}
    \|x_1 + \Delta_{\widetilde{e}}\|_1\\
    & \leq 
    \Big(  \frac{\delta^{\text{Root}}_{2k^{\prime}}}{1-\delta^{Root}_{k^{\prime}}}
    + 
    \frac{\delta^{\text{Non-root}}_{2k^{\prime}}}{1-\delta^{\text{Non-root}}_{k^{\prime}}}
    \Big)
    \|x_1\|_1 
    + 
    \frac{\delta^{\text{Non-root}}_{2k^{\prime}}}{1-\delta^{\text{Non-root}}_{k^{\prime}}}
    \|\Delta_{\widetilde{e}}\|_1
\end{align*}

\textbf{Combining the upper bounds}
Let us now combine the upper bounds for $\|(x_1)_{S_1}\|_1$ and $\|(\Delta_{\widetilde{e}})_{S_{e}}\|_1$ with $e \in E$. Summing them up and noting that there are $\text{Deg}(1)$ edges connecting the root yields 
\begin{align*}
    \|(x_1)_{S_1}\|_1
    + 
    \sum_{e \in E} \|(\Delta_{e})_{S_{e}}\|_1
    & \leq 
    \underbrace{ 
     \Big( \frac{\delta^{\text{Root}}_{2k}}{1-\delta^{Root}_{k}}  + \text{Deg}(1)\frac{\delta^{\text{Root}}_{2k^{\prime}}}{1-\delta^{Root}_{k^{\prime}}}
    + 
    \text{Deg}(1) \frac{\delta^{\text{Non-root}}_{2k^{\prime}}}{1-\delta^{\text{Non-root}}_{k^{\prime}}}
    \Big)
    }_{\textbf{Multiplicative Term}}\\
    & \quad\quad\quad\quad\quad\quad\quad\quad\quad\quad\quad\quad\quad\quad\quad\quad \times \Big( \|x_1\|_1 
    + \sum_{e \in E} \|\Delta_e\|_1
    \Big)
\end{align*}
We then require $ \textbf{Multiplicative Term} \leq 1/2$ for the restricted Null space condition to be satisfied. Now, since $\delta^{\text{Root}}_{2k} \geq \delta^{\text{Root}}_{k}$ and $\delta^{\text{Non-root}}_{2k^{\prime}} \geq \delta^{\text{Non-root}}_{k^{\prime}}$, it is then sufficient for the Restricted Isometry Constants to satisfy the upper bounds 
\begin{align*}
    \delta^{\text{Root}}_{2s} \leq \frac{1}{3} \\
    \delta^{\text{Root}}_{2s^{\prime}} 
    \leq 
    \frac{1}{2 \text{Deg}(1)}\\
    \delta^{\text{Non-root}}_{2s^{\prime}} 
    \leq 
    \frac{1}{2\text{Deg}(1)} 
\end{align*}
Leveraging Theorem \ref{thm:RIPConc} and taking a union bound, this is then satisfied with probability greater than $1-\epsilon$ when  
\begin{align*}
    N_{\text{Root}} \geq 
    18 C  
    \big( \max\{s,\text{Deg}(1)^2 s^{\prime} \} \log(ed) + \text{Deg}(1)^2 \log(1/\epsilon) \big),
    \\
    N_{\text{Non-root}} 
    \geq 
    18 C \text{Deg}(1)^2 
    \big( s^{\prime} \log(ed) + \log(1/\epsilon) \big).
\end{align*}
This concludes the proof. 
\end{proof}

\subsection{Proof of Theorem \ref{thm:Starlike}}
\label{sec:proof:starlike}
We now present the proof of Theorem \ref{thm:Starlike}.  
\begin{proof}[Proof of Theorem \ref{thm:Starlike}]
Once again, recall it is sufficient to demonstrate the Restricted Null Space Property  \eqref{equ:DecentralisedPursit:Reformulated} is satisfied in this case. Following the proof of theorem \ref{thm:IdenticalMatrices}, let the restricted isometry constant of $A_1$ up-to sparsity level $k$ be denoted $\delta^{\text{Root}}_{k} \in [0,1)$. Meanwhile, let $\delta^{\text{Non-root}}_{k^{\prime}}$ now denote the \emph{maximum} restricted isometry constant of up-to sparsity level $k^{\prime}$ of the matrices associated to non-root agents i.e. $\{A_{v}\}_{v \in \backslash \{1\}}$. Furthermore, let $\widetilde{A}_{\text{Combined}} \in \mathbb{R}^{(n-1)N_{\text{Non-root}} \times d}$ be constructed from the row-wise concatenation of the non-root agent matrices $\{A_{v}\}_{v \in V \backslash \{1\}}$. Similarly, let $\delta^{\text{Combined}}_{\widetilde{k}} \in [0,1)$ denote the restricted isometry constant of $A_{\text{Combined}} : = \widetilde{A}_{\text{Combined}}/\sqrt{n-1}$ up to sparsity level $\widetilde{k}$. 

Following the proof of Theorem \ref{thm:IdenticalMatrices} we leverage that $x = (x_1,\Delta_1,\dots,\Delta_{|E|}) \in \text{Ker}(A) \backslash \{0\}$ to upper bound $\|(x_1)_{S_1}\|_1, \|(\Delta_e)\|_1$ for $e \in E$. In particular, we consider have three paragraphs: one associated to bounding $\|(x_1)_{S_1}\|_1$; one for bounding $\|(\Delta_e)\|_1$ for edges $e = (v,w) \in E$ not connected to the root $v,w \not=1$; and  one for bounding $\|(\Delta_e)\|_1$ for edges $e=(v,w)$ joined to the root $v=1$ or $w=1$. The fourth paragraph will then combined these bounds.

\textbf{Root Node} Since $x_1 \in \text{Ker}(A_1)$ we immediately have from Lemma \ref{lem:RIPtoL2} the upper bound for any $U \subset \{1,\dots,d\}$ such that $|U| = k$
\begin{align}
\label{equ:RootBound}
    \|(x_1)_{U}\|_2 \leq 
    \frac{\delta^{\text{Root}}_{2k}}{1-\delta^{\text{Root}}_{k}} 
    \frac{1}{\sqrt{k}}
    \|x_1\|_1.
\end{align}
Setting $U = S_1$ and recalling and following the proof of Theorem \ref{thm:IdenticalMatrices}, this immediately bounds $\|(x)_{S_{1}}\|_1 \leq \frac{\delta^{\text{Root}}_{2s}}{1-\delta^{\text{Root}}_{s}} \|x_1\|_1 $.

\textbf{Edges not connect to the root} Consider any edge $\widetilde{e} = (v,w) \in E$ not directly connected to the root i.e.\ $v,w \not= 1$. Without loss in generality suppose $w$ is the furthest from the root. This allows us to rewrite in terms of the difference
\begin{align*}
    \Delta_{\widetilde{e}} 
    = 
    \Big(  \sum_{e \in \pi(w)} \Delta_{e} \Big)
    - 
     \Big( \sum_{e \in \pi(v)} \Delta_{e} \Big).
\end{align*}
Each of the vectors are in potentially different null-spaces, and therefore, we bound each separately. Applying triangle inequality we then get 
\begin{align*}
    \|(\Delta_{\widetilde{e}})_{S_{\widetilde{e}}}\|_1
    & \leq 
    \Big\|
    \Big(  \sum_{e \in \pi(w)} \Delta_{e} 
    \Big)_{S_{\widetilde{e}}} 
    \Big\|_1 
    + 
    \Big\|
    \Big(  \sum_{e \in \pi(v)} \Delta_{e} 
    \Big)_{S_{\widetilde{e}}}
    \Big\|_1 \\
    & \leq 
    \Big\|
    \Big(  \sum_{e \in \pi(w)} \Delta_{e} 
    \Big)_{U_{1}} 
    \Big\|_1 
    + 
    \Big\|
    \Big(  \sum_{e \in \pi(v)} \Delta_{e} 
    \Big)_{U_{2}}
    \Big\|_1 \\
\end{align*}
where $U_1,U_2$ are the largest $s^{\prime}$ entries of $\sum_{e \in \pi(w)} \Delta_{e} $ and $\sum_{e \in \pi(v)} \Delta_{e} $ respectively. For the first term we have $x_1 + \sum_{e \in \pi(w)} \Delta_{e} \in \text{Ker}(A_{w})$ and therefore 
\begin{align*}
    (1-\delta^{\text{Non-root}}_{k^{\prime}})
    \Big\|
    \Big(  \sum_{e \in \pi(w)} \Delta_{e} 
    \Big)_{U_{1}} 
    \Big\|_2^2 & \leq 
    \Big\|
    A_{w}
    \Big(  \sum_{e \in \pi(w)} \Delta_{e}  
    \Big)_{U_1} 
    \Big\|_2^2\\
    & =
    - \Big\langle 
    A_w \Big(  \sum_{e \in \pi(w)} \Delta_{e} 
    \Big)_{U_{1}},
    A_w \Big(  \sum_{e \in \pi(w)} \Delta_{e} \Big)_{U_{1}^{c}}
    + 
    A_w x_1 \Big\rangle
\end{align*}
where we note that $A_{w}(x_1 + \sum_{e \in \pi(w)} \Delta_{e} ) = 0$ and therefore $A_{w}\Big(\sum_{e \in \pi(w)} \Delta_{e} \Big)_{U_{1}} = -A_{w}\Big(\sum_{e \in \pi(w)} \Delta_{e} \Big)_{U_{1}^c} - A_{w} x_1 $.
Following the shelling argument in the proof of Lemma \ref{lem:RIPtoL2} we can upper bound 
\begin{align*}
    \Big|\Big\langle 
    A_w \Big(  \sum_{e \in \pi(w)} \Delta_{e} 
    \Big)_{U_1},
    A_w \Big(  \sum_{e \in \pi(w)} \Delta_{e} \Big)_{U_1^{c}}
    \Big\rangle
    \Big|
    \leq 
    \frac{\delta^{\text{Non-root}}_{2s^{\prime}}}{\sqrt{s^{\prime}}} 
    \Big\|\Big(  \sum_{e \in \pi(w)} \Delta_{e} 
    \Big)_{U_1}
    \Big\|_2 \Big\|\sum_{e \in \pi(w)} \Delta_{e}\Big\|_1
\end{align*}
Upper bound the other inner product as
\begin{align*}
    \Big|\Big\langle 
    A_w \Big(  \sum_{e \in \pi(w)} \Delta_{e} 
    \Big)_{U_{1}},
    A_w x_1 \Big\rangle\Big|
    & \leq 
    \big\|A_w \Big(  \sum_{e \in \pi(w)} \Delta_{e} 
    \Big)_{U_{1}}\big\|_2 \|A_w x_1\|_2 
    \\
    & \leq 
    \sqrt{1 + \delta^{\text{Non-root}}_{s^{\prime}}} 
    \big\|
    \Big(  \sum_{e \in \pi(w)} \Delta_{e} 
    \Big)_{U_{1}}\big\|_2 \|A_w x_1\|_2 
\end{align*}
and dividing both sides by $(1-\delta^{\text{Non-root}}_{s^{\prime}}) \Big\|\Big(  \sum_{e \in \pi(w)} \Delta_{e} \Big)_{U_1} \Big\|_2$ then yields 
\begin{align*}
    \Big\|
    \Big(  \sum_{e \in \pi(w)} \Delta_{e} 
    \Big)_{U_1} 
    \Big\|_2
    \leq 
    \frac{\delta^{\text{Non-root}}_{2s^{\prime}}}{\sqrt{s^{\prime}}} \Big\|\sum_{e \in \pi(w)} \Delta_{e}\Big\|_1
    + 
    \frac{\sqrt{1+\delta^{\text{Non-root}}_{s^{\prime}}}}{1-\delta^{\text{Non-root}}_{s^{\prime}}}
    \|A_{w}x_1\|_2
\end{align*}
Repeating the steps above for the other node $v$ and going to $\ell_1$ norm from $\ell_2$ norm and bringing together the two bounds yields 
\begin{align}
\label{equ:EdgeBound1}
    \|(\Delta_{\widetilde{e}})_{S_{\widetilde{e}}}\|_1
    & \leq 
    \frac{\delta^{\text{Non-root}}_{2s^{\prime}}}{1-\delta^{\text{Non-root}}_{s^{\prime}}}
    \Big( \Big\|\sum_{e \in \pi(w)} \Delta_{e}\Big\|_1
    + 
    \Big\|\sum_{e \in \pi(v)} \Delta_{e}\Big\|_1
    \Big)\\
    & \quad\quad\quad\quad\quad\quad\quad 
    + 
    \frac{\sqrt{s^{\prime}} \sqrt{1 + \delta^{\text{Non-root}}_{s^{\prime}}}}{1-\delta^{\text{Non-root}}_{s^{\prime}}} 
    \big( \|A_{v}x_1\|_2 + \|A_w x_1 \|_2 \big)
    \nonumber
\end{align}

\textbf{Edges connecting to the root node} Consider an edge connected to the root node, that is, $\widetilde{e} = (v,w) \in E$ such that $v=1$ or $w = 1$. Without loss in generality, let us suppose that $w= 1$. We can then bound using Restricted Isometry Property 
\begin{align*}
    (1 - \delta_{s^{\prime}})^{\text{Non-root}}
    \|(\Delta_{\widetilde{e}})_{S_{\widetilde{e}}}\|_2^2
    & \leq
    \big\| A_{v}
    (\Delta_{\widetilde{e}})_{S_{\widetilde{e}}}
    \|_2^2\\
    & = 
    - \langle 
    A_v
    (\Delta_{\widetilde{e}})_{S_{\widetilde{e}}},
    A_v
    (\Delta_{\widetilde{e}})_{S_{\widetilde{e}}^c}
    + 
    A_v x_1 \rangle\\
    & \leq 
    \frac{\delta^{\text{Non-root}}_{2s^{\prime}} }{\sqrt{s^{\prime}}}
    \|
    (\Delta_{\widetilde{e}})_{S_{\widetilde{e}}}\|_2 
    \|\Delta_{\widetilde{e}}\|_1 
    + 
    \|A_v
    (\Delta_{\widetilde{e}})_{S_{\widetilde{e}}}\|_2
    \|A_v x_1 \|_2\\
    & \leq 
    \frac{\delta^{\text{Non-root}}_{2s^{\prime}} }{\sqrt{s^{\prime}}}
    \|
    (\Delta_{\widetilde{e}})_{S_{\widetilde{e}}}\|_2 
    \|\Delta_{\widetilde{e}}\|_1 
    + 
    \sqrt{1+\delta^{\text{Non-root}}_{s^{\prime}}}
    \|(\Delta_{\widetilde{e}})_{S_{\widetilde{e}}}\|_2
    \|A_v x_1 \|_2\\
\end{align*}
where for the equality we note that $x_1 + \Delta_{\widetilde{e}} \in \text{Ker}(A_v) \backslash \{0\}$ and therefore $A_{v}(\Delta_{\widetilde{e}})_{S_{\widetilde{e}}} = - A_{v}(\Delta_{\widetilde{e}})_{S_{\widetilde{e}}^c} - A_v x_1$. Meanwhile for the second inequality used a similar argument to previously. Dividing both sides by $(1-\delta^{\text{Non-root}}_{s^{\prime}}) \|(\Delta_{\widetilde{e}})_{S_{\widetilde{e}}}\|_2 $  and going to $\ell_1$ norm we then get 
\begin{align}
\label{equ:EdgeBound2}
    \|(\Delta_{\widetilde{e}})_{S_{\widetilde{e}}}\|_1
    \leq 
    \frac{\delta^{\text{Non-root}}_{2s^{\prime}} }{1 - \delta^{\text{Non-root}}_{s^\prime}} 
    \|\Delta_{\widetilde{e}}\|_1 
    + 
    \frac{\sqrt{s^{\prime}} \sqrt{1 + \delta^{\text{Non-root}}_{s^{\prime}}}}{1- \delta^{\text{Non-root}}_{s^{\prime}}} 
    \|A_v x_1 \|_2
\end{align}

\textbf{Combining upper bounds} Let us now combine the bounds on $\|(x)_{S_{1}}\|_1$ from \eqref{equ:RootBound}, as well as the bounds \eqref{equ:EdgeBound1} and \eqref{equ:EdgeBound2} for the edges $e \in E$. This yields  
\begin{align*}
    & \|(x_1)_{S_1}\|_1 
    + 
    \sum_{e \in E}\|(\Delta)_{S_{e}}\|_1  \\
    & = \|(x_1)_{S_{1}}\|_1 
    + 
    \sum_{e=(v,w) \in E :v,w \not=1 } \|(\Delta_{e} )_{S_{e}}\|_1
    + 
    \sum_{e=(v,w) \in E :v=1 \text{ or } w = 1} \|(\Delta_{e} )_{S_{e}}\|_1\\
    & \leq 
    \frac{\delta^{\text{Root}}_{2 s}}{1-\delta^{\text{Root}}_{s}}
    \|x_1 \|_1 \\
    &\quad\quad 
    +
    \frac{\delta^{\text{Non-root}}_{2s^{\prime}}}{1-\delta^{\text{Non-root}}_{s^{\prime}}}
    \underbrace{ 
    \sum_{e=(v,w) \in E :v,w \not=1 } 
    \Big( \Big\|\sum_{e \in \pi(w)} \Delta_{e}\Big\|_1
    + 
    \Big\|\sum_{e \in \pi(v)} \Delta_{e}\Big\|_1
    \Big)}_{\textbf{Term 1}} \\
    &\quad\quad
    \frac{\delta^{\text{Non-root}}_{2s^{\prime}}}{1-\delta^{\text{Non-root}}_{s^{\prime}}} \sum_{e \in E : v =1 \text{ or } w = 1} \|\Delta_{e}\|_1
  +  \underbrace{ 
  \frac{2 \sqrt{s^{\prime}} \sqrt{1 + \delta^{\text{Non-root}}_{s^{\prime}}}}{ 1-\delta^{\text{Non-root}}_{s^{\prime}}}
    \sum_{v \in V \backslash\{1\}}\|A_v x_1\|_2 }_{\textbf{Term 2}}.
\end{align*}
 Where we must now bound \textbf{Term 1} and \textbf{Term 2}. To bound \textbf{Term 1} we simply apply triangle inequality to get 
 \begin{align*}
     \textbf{Term 1} \leq 
     \sum_{e = (v,w) \in E :  v ,w\not= 1}
     2 \sum_{\widetilde{e} \in \pi(v) \cup \pi(w) }
     \|\Delta_{\widetilde{e}}\|_1
     \leq 2 \text{Deg}(V \backslash \{1\}) \text{Diam}(G) \sum_{e \in E} \|\Delta_{e}\|_1
 \end{align*}
where $\pi(v) \cup \pi(w)$ denotes the union without duplicates. We then note that the sum $\sum_{e = (v,w) \in E :  v ,w\not= 1} 2 \sum_{\widetilde{e} \in \pi(v) \cup \pi(w) } \dots $ can be seen as counting the number of times an edge is used on a path from any non-root node to the root. The edges which appear on most paths to the root are those directly connected to the root. The number of edges feeding into the edge directly connected to the root is then upper bounded by the max degree of non-root nodes times the graph diameter $\text{Deg}(V \backslash \{1\}) \text{Diam}(G)$.  
To bound \textbf{Term 2} we use Cauchy-Schwartz and recall the definition of $A_{\text{Combined}}$ to get
\begin{align*}
    \sum_{v \in V \backslash \{v\}} \|A_{v} x_1\|_2 
    \leq 
    \sqrt{n-1} \sqrt{ \sum_{v \in V \backslash \{v\}} \|A_{v} x_1\|_2^2 } 
    = (n-1) 
    \|A_{\text{Combined}} x_1\|_2
\end{align*}
Using the fact that $x_1 \in \text{Ker}(A_1)\backslash \{0\}$ as well as the Restricted Isometry Property of $A_{\text{Combined}}$, we have show the following upper bound for $\ell \geq 1$
\begin{align}
\label{equ:CombinedRIP}
    \textbf{Term 2}
    \leq
     \frac{6 \sqrt{(1 + \delta^{\text{Non-root}}_{s^{\prime}})(1 + \delta^{\text{Combined}}_{\ell})}}{ (1-\delta^{\text{Non-root}}_{s^{\prime}})(1-\delta^{\text{Root}}_{2\ell})}
    \frac{(n-1) \sqrt{s^{\prime}} }{\sqrt{\ell}}\big) \|x_1\|_1 
\end{align}
The proof of \eqref{equ:CombinedRIP} is then provided at the end.
Bringing everything together and collecting constants we get 
\begin{align*}
    & \|(x_1)_{S_1}\|_1 
    + 
    \sum_{e \in E}\|(\Delta)_{S_{e}}\|_1\\
    &   \leq 
    \underbrace{ 
    3 \max\Bigg\{
    \frac{\delta^{\text{Root}}_{2 s}}{1-\delta^{\text{Root}}_{s}},
    \frac{2 \text{Deg}(V \backslash \{1\}) \text{Diam}(G)  \delta^{\text{Non-root}}_{2s^{\prime}}}{1-\delta^{\text{Non-root}}_{s^{\prime}}},
    \frac{6 \sqrt{(1 + \delta^{\text{Non-root}}_{s^{\prime}})(1 + \delta^{\text{Combined}}_{\ell})}}{ (1-\delta^{\text{Non-root}}_{s^{\prime}})(1-\delta^{\text{Root}}_{2\ell})}
    \frac{(n-1) \sqrt{s^{\prime}} }{\sqrt{\ell}}\big) 
    \Bigg\}
    }_{\text{Multiplicative Term}}
    \\
    &
    \quad\quad\quad\quad\quad\quad\quad
    \quad\quad\quad\quad\quad\quad\quad
    \quad\quad\quad\quad\quad\quad\quad
    \quad\quad\quad\quad\quad\quad\quad
    \times
    \Big( 
    \|x_1\|_1 
     + 
    \sum_{e \in E}\|\Delta\|_1
    \Big)
\end{align*}
For the restricted nullspace property to be satisfied we must then ensure that $\textbf{Multiplicative Term} \leq 1/2$. This can then be ensured when setting $\ell = 156^2 (n-1)^2 s^{\prime}$ when the Restricted Isometry constants satisfy
\begin{align*}
    \delta^{\text{Root}}_{2s} \leq \frac{1}{4} \\
    \delta^{\text{Non-root}}_{2s^{\prime}} \leq \frac{1}{1 + 12 D} \\
    \delta^{\text{Root}}_{2\ell} \leq 1/2 \\
    \delta^{\text{Combined}}_{\ell} \leq 1
\end{align*}
Using theorem \ref{thm:RIPConc}, the conditions on $\delta^{\text{Root}}_{2s}, \delta^{\text{Root}}_{2\ell}$ and $\delta^{\text{Non-root}}_{2s^{\prime}}$ are ensured with probability  greater than $1-\epsilon$ when 
\begin{align*}
    N_{\text{Root}} \geq C \times 32 \times 156^2  \max\{n^2 s^{\prime},s\} \big(  \log(ed) + \log(1/\epsilon) \big)\\
    N_{\text{Non-root}} \geq C \times 13^2 \times \text{Deg}(V \backslash \{1\})^2 \text{Diam}(G) ^2 s^{\prime}
    \big(  \log(ed) + \log(n/\epsilon) \big)
\end{align*}
where  $\delta^{\text{Non-root}}_{2s^{\prime}}$   is the maximum restricted Isometry constant across the matrices $\{A_{v}\}_{v \in V \backslash \{1\}}$ and therefore, a union bound was taken. Meanwhile, for $\delta^{\text{Combined}}_{\ell}$, recall that the entries of $\sqrt{N_{\text{Non-root}} (n-1)} \times A_{\text{Combined}}$ are independent and sub-Gaussian i.e.\ $\widetilde{A}_{\text{Combined}}$ is the row-wise concatenation of $A_{v} = \widetilde{A}_{v}/\sqrt{N_{\text{Non-root}}}$ for $v \in V \backslash \{1\}$ where $\widetilde{A}_{v}$ has independent and identical sub-Gaussian entries.  Therefore, following Theorem \ref{thm:RIPConc} the condition on $\delta^{\text{Combined}}_{\ell}$ is then satisfied when 
\begin{align*}
    N_{\text{Non-root}} (n-1) \geq 
    C \times
    156^2 (n-1)^2 s^{\prime}
    \big( \log(ed) + \log(1/\epsilon) \big)
\end{align*}
Dividing both sides by $n-1$ and combining the conditions on $N_{\text{Non-root}}$ yields the result.

Let us now prove \eqref{equ:CombinedRIP}. 
Using \eqref{equ:RootBound} with $U$ being the largest $\ell$ entries of $x_1$ we have (since $x_1 \in \text{Ker}(A_1)\backslash \{0\}$) 
\begin{align*}
    \|x_1\|_2\leq \|(x_1)_{U^{c}}\|_2 + \|(x_1)_{U}\|_2 \leq \frac{1}{\sqrt{\ell}} \|x_1\|_1 + \frac{\delta^{\text{Root}}_{2 \ell}}{1-\delta^{\text{Root}}_{\ell}} \frac{1}{\sqrt{\ell}} \|x_1\|_1 = \big(1 + \frac{\delta^{\text{Root}}_{2 \ell}}{1-\delta^{\text{Root}}_{\ell}}\big) \frac{1}{\sqrt{\ell}} \|x_1\|_1
\end{align*} 
where have bounded using the shelling argument $U^{c} = B_1 \cup B_2 \cup \dots$ as $\|(x)_{U^{c}}\|_2 \leq \sum_{j \geq 1} \|(x)_{B_j}\|_2 \leq \frac{1}{\sqrt{\ell}} \|x\|_1$. That is $B_1$ is the largest $\ell$ entries of $x_1$ in $U^{c}$, $B_2$ is the largest $\ell$ entries in $(U \cup B_1)^{c}$ and so on. We then have $\|(x)_{B_{j-1}}\|_2 \leq \sqrt{\ell} \|(x)_{B_{j-1}}\|_{\infty} \leq \frac{1}{\sqrt{\ell}} \|(x)_{B_{j}}\|_1$.  Therefore we can bound with $c = \big(1 + \frac{\delta^{\text{Root}}_{2 \ell}}{1-\delta^{\text{Root}}_{\ell}}\big)$
\begin{align*}
    \|A_{\text{Combined}} x_1\|_2
    = 
    \|A_{\text{Combined}} \frac{x_1}{\|x_1\|_1} \|_2 \|x_1\|_1 
    & \leq 
    \Big( \max_{x: \|x\|_2 \leq \frac{c}{\sqrt{\ell}}, \|x\|_1 \leq 1 } \|A_{\text{Combined}} x\|_2
    \Big) 
    \|x_1\|_1
\end{align*}
where if we fix $x = \frac{x_1}{\|x_1\|_1}$ then it is clear $\|x\|_1 = 1$ and $\|x\|_2 = \frac{\|x_1\|_2}{\|x_1\|_1} \leq \frac{c}{\sqrt{\ell}}$.  We now study the maximum above. In particular, it can be rewritten since $c \geq 1$
\begin{align*}
    \max_{x: \|x\|_2 \leq \frac{c}{\sqrt{\ell}}, \|x\|_1 \leq 1 } \|A_{\text{Combined}} x\|_2
    & = \frac{c}{\sqrt{\ell}} \max_{x: \|x\|_2 \leq 1, \|x\|_1 \leq \frac{\sqrt{\ell}}{c}}
    \|A_{\text{Combined}} x\|_2\\
    & \leq 
    \frac{c}{\sqrt{\ell}} \max_{x: \|x\|_2 \leq 1, \|x\|_1 \leq \sqrt{\ell} }
    \|A_{\text{Combined}} x\|_2\\
    & = \frac{c}{\sqrt{\ell}}
    \max_{x \in \mathbb{B}_{2}(1) \cap \mathbb{B}_{1}(\sqrt{\ell}) }
    \|A_{\text{Combined}} x\|_2\\
\end{align*}
Using Lemma \ref{lem:balls} we can then bound 
\begin{align*}
    \max_{x \in \mathbb{B}_{2}(1) \cap \mathbb{B}_{1}(\sqrt{\ell}) }
    \|A_{\text{Combined}} x\|_2
    & \leq 
    \max_{x \in \mathbb{B}_{2}(3 ) \cap \mathbb{B}_{0}(\ell) }
    \|A_{\text{Combined}} x\|_2\\
    & \leq 
    3 \sqrt{1 + \delta_{\ell}^{\text{Combined}}}
\end{align*}
where at the end used the Restricted Isometry Property of $A_{\text{Combined}}$. Bringing everything together we get the bound for \textbf{Term 2} with $c  = 1 + \frac{\delta^{\text{Root}}_{2\ell}}{1-\delta^{\text{Root}}_{\ell}} \leq \frac{1}{1-\delta^{\text{Root}}_{2\ell}}$
\begin{align*}
    \textbf{Term 2}
    \leq
     \frac{6 \sqrt{(1 + \delta^{\text{Non-root}}_{s^{\prime}})(1 + \delta^{\text{Combined}}_{\ell})}}{ (1-\delta^{\text{Non-root}}_{s^{\prime}})(1-\delta^{\text{Root}}_{2\ell})}
    \frac{(n-1) \sqrt{s^{\prime}} }{\sqrt{\ell}}\big) \|x_1\|_1 
\end{align*}
as required.

\end{proof}

\section{Proofs for Noisy Case}
In this section we provide the proofs for the noisy setting. Section \ref{sec:BasisPursuitDenoisingProblem} begin by introducing the problem of Basis Pursuit Denoising. Section \ref{sec:TVBasisPursuitDenoisingProblem:proof} presents the proof of Theorem \ref{thm:TVBPD}. Section \ref{sec:Proof:Noisy} presents the proof for an intermediate lemma.

\subsection{Basis Pursuit Denoising}
\label{sec:BasisPursuitDenoisingProblem}
Let us begin by introducing Basis Pursuit Denoising. That is suppose $y = A x^{\star} + \epsilon$ for some noise $\epsilon \in \mathbb{R}^{n}$. The Basis Pursuit Denoising problem \citep{chen2001atomic} then considers replacing the equality with a bound on the $\ell_2$. Namely for $\eta \geq 0 $ 
\begin{align}
\label{equ:DenoisingBasisPursuit}
\min_{x } \|x\|_1 \text{ subject to } \|Ax - y \|_2 \leq \eta.
\end{align}
Naturally, the equality constraint $Ax=y$ in the noiseless setting has been swapped for an upper bound on the discrepancy $\|Ax - y\|_2$. To investigate guarantees for the solution to \eqref{equ:DenoisingBasisPursuit}, we consider the Robust Null Space Property, see for instance \citep{foucart2013invitation}. A matrix $A$ is said to satisfy the Robust Null Space Property for a set $S \subseteq \{1,\dots,p\}$ and parameters $\rho,\tau \geq 0$ if 
\begin{align}
\label{equ:RobustNullSpace}
\|x_{S}\|_1 \leq \rho \|x_{S^{c}}\|_1 + \tau \|Ax\|_2 
\text{ for all } x \in \mathbb{R}^{N}.
\end{align}
Given condition \eqref{equ:RobustNullSpace}, bounds on the $\ell_1$ estimation error between a solution to the Denoising Basis Pursuit problem \eqref{equ:DenoisingBasisPursuit} and the true underlying signal $x^{\star}$ can be obtained. That is, for any solution to \eqref{equ:DenoisingBasisPursuit}, $x \in \mathbb{R}^{p}$ with $y = Ax^{\star} + e$ where $\|e\|_2 \leq \eta$, we have
(see \citep{foucart2013invitation} with $z = x^{\star}$)
\vspace{-0.1cm}
\begin{align*}
\|x - x^{\star}\|_1 \leq \underbrace{ \frac{2(1+\rho)}{1-\rho}  \|(x^{\star})_{S^{c}}\|_1 }_{\text{Sparse Approximation}} + \underbrace{\frac{4 \tau}{1-\rho} \eta}_{\text{Noise}}.
\end{align*}
The first term above encodes that $x^{\star}$ is not exactly $s$ sparse, while the second term represents error from the noise. We now discuss the values taken by $\eta$ and $\tau$ in the case that $A$ has i.i.d.\ sub-Gaussian entries. Recall Theorem \ref{thm:RIPConc} that the scaled matrix $ A/\sqrt{N}$ in this case can satisfy a Restricted Isometry Property, and thus, it is natural to choose $\eta = \sqrt{N} \eta_{\text{Noise}}$ for $\eta_{\text{Noise}} \geq 0$ since the $\ell_2$ bound on the residuals in  \eqref{equ:DenoisingBasisPursuit} becomes $\| A x - y\|_2/\sqrt{N} \! \leq \! \eta_{\text{Noise}}$. We can then pick $\|e\|_2/\sqrt{N} \! \leq \! \eta_{\text{Noise}}$, which is an upper bound on the standard deviation of the noise. The Robust Null Space Property then holds in this case, see \citep[Theorem 4.22]{foucart2013invitation}, with $\tau \approx \sqrt{s}$, leading to a $\ell_1$ error bound of the order $\|x - x^{\star}\|_1 \! \lesssim \!  \|(x^{\star})_{S^{c}}\|_1 \!  + \! \eta_{\text{Noise}} \sqrt{s}$ (see \citep{wainwright2019high}).

\subsection{Proof of Theorem \ref{thm:TVBPD}}
\label{sec:TVBasisPursuitDenoisingProblem:proof}
We begin by recalling Section \ref{sec:Proof:reformulated} which reformulated the Total Variation Basis Pursuit problem \eqref{equ:TVBP} into a Basis Pursuit problem \eqref{equ:BasisPursuit:app}. In particular, we note that Total Variation Basis Pursuit Denoising \eqref{equ:DecentralisedDenoising} can be reformulated into Basis Pursuit Denoising \eqref{equ:DenoisingBasisPursuit} in a similar manner.  Let us suppose the root node signal $x^{\star}_1$  and the  $\{\Delta^{\star}_{e}\}_{e \in E}$ are approximately sparse and each agent $v \! \in \! V$ holds noisy samples $y_v \! \approx \! A_v x^{\star}_v$. Reformulating the Total Variation Basis Pursuit problem into a Basis Pursuit problem \eqref{equ:DecentralisedPursit:Reformulated} and bounding the $\ell_2$ norm of the residuals, then yields the \emph{Total Variation Basis Pursuit Denoising} problem 
\begin{align}
\label{equ:DecentralisedDenoising:Reformulated}
\min_{x_1, \Delta_e  \in E} \|x_1\|_1 & + \sum_{e \in E} \|\Delta_e \|_1 \text{ subject to }\\
 \sum_{v \in V} \| A_v \big(x_1 & + \sum_{e \in \pi(v)} \Delta_e \big) - y_v\|_2^2  \leq \eta^2 .
 \nonumber
\end{align}
Where $\eta^2$ now  upper bounds the squared $\ell_2$ norm of the noise summed across all of the nodes i.e.\ $\sum_{v \in V} \! \|A_v x^{\star}_v \! - \! y_v\|_2^2$. This is now in the form of \eqref{equ:DenoisingBasisPursuit} with an augmented matrix $A$ as in the reformulated problem described in Section \ref{sec:Proof:reformulated}. 

Let us now recall that $\delta^{\text{Non-root}}_k$ denotes the largest Restricted Isometry Constant of the matrices associated to the non-root agents $\{A_v\}_{v \in V \backslash \{1\} }$. Similarly, $\delta^{\text{Root}}_{k}$ denotes the  Restricted Isometry Constant associated to the root matrix $A_1$. 
The following theorem then gives, values for $\rho$ and $\tau$ for which the augmented matrix $A$ and sparsity set $S$ (described in Section \ref{sec:Proof:reformulated}) satisfies  the Robust Null Space Property \eqref{equ:RobustNullSpace}. 
\begin{lemma}
\label{equ:RobustNullSpace:Decentralised}
Consider the $A$ matrix and sparsity set $S$ as constructed in Section \ref{sec:Proof:reformulated}. Then $A$ satisfies the Robust Null Space Property with $\rho = \rho^{\prime}/(1-\rho^{\prime})$ and $\tau = \tau^{\prime} /(1-\rho^{\prime})$ where 
\begin{align*}
\rho^{\prime}  & = 4 \Big( \frac{n \delta^{\text{Non-root}}_{2s^{\prime}}}{1-\delta^{\text{Non-root}}_{s^{\prime}}} \vee \frac{\delta^{\text{Root}}_{2s}}{1-\delta^{\text{Root}}_{s}} \Big)
\quad \text{ and } \\ 
 \tau^{\prime} & = \frac{ \sqrt{1 + \delta^{\text{Non-root}}_{s^{\prime}}}}{1-\delta^{\text{Non-root}}_{s^{\prime}}} \vee \frac{ \sqrt{1 + \delta^{\text{Root}}_{s}}}{1-\delta^{\text{Root}}_{s}} \big( \sqrt{s} + \text{Deg}(G)\sqrt{n s^{\prime}} \big). 
\end{align*}
\end{lemma}
Note the parameter $\tau^{\prime}$ scales (up to a network degree \text{Deg}(G) factor) with the sparsity of the differences. Naturally we require $\rho^{\prime} < 1/2$, which can be ensured if each agent $\{A_v\}_{v \in V}$ have i.i.d.\ sub-Gaussian matrices. In particular, we require 
\begin{align*}
    \delta^{\text{Root}}_{2s} \leq \frac{1}{9}
    \text{  and  }
    \delta^{\text{Non-root}}_{2s^{\prime}} \leq \frac{1}{1+8n}
\end{align*}
Following Theorem \ref{thm:RIPConc} this can be ensured with probability greater than $1-\epsilon$ when 
\begin{align*}
    N_{\text{Root}} \geq 81 C \big(  2s\log(ed/s) + \log(n/\epsilon)\big) \\
    N_{\text{Non-root}} \geq 81 n^2 C \big(  2s^{\prime}\log(ed/s^{\prime}) + \log(n/\epsilon)\big)
\end{align*}
Choosing  $\eta = \sqrt{\sum_{v \in V} N_{v}} \eta_{\text{Noise}}$ where $\eta_{\text{Noise}} > 0$ upper bounds the noise standard deviation across all of the agents, the $\ell_1$ estimation error of the solution to \eqref{equ:DecentralisedDenoising} is then of the order 
\begin{align*}
 \|x_1 - x^{\star}_1\|_1 & + \sum_{e \in E} \|\Delta_e - \Delta^{\star}_e\|_1 
  \lesssim 
\underbrace{ \|(x^{\star})_{S^{c}}\|_1 }_{\text{Approximate Sparsity}} 
+ \underbrace{ (\sqrt{s} + \text{Deg}(G) \sqrt{n s^{\prime}}) \eta_{\text{Noise}}}_{\text{Noise}}.
\end{align*}
The error scales with the approximate sparsity of the true signal through $\|(x^{\star})_{S^{c}}\|_1$ and now the noise term with the effective sparsity $\sqrt{s} + \text{Deg}(G) \sqrt{n s^{\prime}}$. Picking $x^{\star}$ to be supported on $S$, the approximate sparsity term goes to zero, as required.

\subsection{Proof of Lemma \ref{equ:RobustNullSpace:Decentralised}}
\label{sec:Proof:Noisy}

We now set to show that the Robust Null Space Property \eqref{equ:RobustNullSpace} holds for some $\rho,\tau$. We note it suffices to show the following which is  equivalent to the Robust Null Space Property 
\begin{align*}
\|(x)_{S}\|_1 \leq \rho^{\prime} \|x\|_1 + \tau^{\prime} \|Ax\|_2 \text{ for all } x \in \mathbb{R}^{N}. 
\end{align*}
In particular, by adding $\rho \|(x)_{S}\|_1$ to both sides of the inequality for the Robust Null Space Property \eqref{equ:RobustNullSpace} and dividing by $1+\rho$, we see that if the above holds then the Robust Null Space Property holds with $\rho = \frac{\rho^{\prime}}{1-\rho^{\prime}}$ and $\tau = \tau^{\prime}/(1-\rho^{\prime})$.

The proof naturally follows the noiseless case (proof for Theorem \ref{thm:Starlike}) although with an additional error term owing the noise. To make the analysis clearer, the steps following the noiseless case from the proof of Theorem \ref{thm:Starlike}, are simplified.

We begin by controlling the $\ell_1$ norm of $x_1 + \sum_{e \in \pi(v)} \Delta_e = x_{v}$ for $v \in V$ restricted to subsets $U$.  Considering the subset $U$ of size $|U| \leq s^{\prime}$, and in particular, the set $U$ associated to the largest $s^{\prime}$ entries of $x_v$. Following the shelling argument used within the proof of Lemma \ref{lem:RIPtoL2}
decompose  $U^{c} = B_1 \cup B_2 \cup \dots$ where $B_1$ are the largest $s^{\prime}$ entries of $x_v$ within $U^{c}$, $B_2$ are the $s^{\prime}$ largest entries of $x_v$ in $(U \cup B_1)^{c}$  and so on. We can then upper bound
\begin{align*}
(1-\delta^{\text{Non-Root}}_{s^\prime})\|(x_{v})_{U}\|_2^2 
& \leq 
\|A_v  (x_{v})_{U}\|_2^2 \\
& = \Big\langle A_v  (x_{v})_{U},A_v \Big( x_{v} - \sum_{j \geq 1} (x_{v})_{B_j} \Big) \Big\rangle \\
& = \langle A_v  (x_{v})_{U},A_v  x_{v}\rangle - \sum_{j \geq 1} \langle A_v  (x_{v})_{U},A_v (x_{v})_{B_j}\rangle  \\
& \leq \sqrt{1+\delta^{\text{Non-root}}_{s^{\prime}}} \|(x_{v})_{U}\|_2 \|A_v  x_{v}\|_2 + \frac{\delta^{\text{Non-root}}_{2s^\prime}}{\sqrt{s^{\prime}}} \|(x_{v})_{U}\|_2 \|x_{v}\|_1
\end{align*}
where we used the Restricted Isometry Property of $A_v$ to upper bound inner product $\langle A_v  (x_{v})_{U},A_v  x_{v}\rangle \leq \|A_v  (x_{v})_{U}\|_2 \|A_v  x_{v}\|_2 \leq \sqrt{1+\delta^{\text{Non-root}}_{s^{\prime}}} \|(x_{v})_{U}\|_2 \|\|A_v  x_{v}\|_2  $ and followed the steps in the proof of Lemma \ref{lem:RIPtoL2} to upper bound  $\sum_{j \geq 1} \langle A_v  (x_{v})_{U},A_v (x_{v})_{B_j}\rangle   \leq \delta^{\text{Non-root}}_{2s^{\prime}} \|(x_{v})_{U}\|_2 \sum_{j \geq 1} \| (x_{v})_{B_j}\|_1 \leq \frac{1}{\sqrt{s^{\prime}}} \delta^{\text{Non-root}}_{2s^{\prime}} \|(x_{v})_{U}\|_2 \|x_{v}\|_1 $. Dividing both sides by $(1-\delta^{\text{Non-root}}_{s^\prime})\|(x_{v})_{U}\|_2$ we then get 
\begin{align*}
	\|(x_{v})_{U}\|_2\leq 
	\frac{\delta^{\text{Non-root}}_{2s^\prime}}{1-\delta_{s^{\prime}}}\frac{1}{ \sqrt{s^{\prime}}} 
	\|x_{v}\|_1
	+
	\frac{ \sqrt{1+\delta^{\text{Non-root}}_{s^{\prime}}}}{1 - \delta^{\text{Non-root}}_{s^{\prime}}} \|A_v  x_{v}\|_2 
\end{align*}
Using that $\|(x_{v})_{U}\|_2 \geq \frac{1}{\sqrt{s^{\prime}}} \|(x_{v})_{U}\|_1$  as well as simply upper bounding $\|x_{v}\|_1 = \| x_1 + \sum_{e \in \pi(v)} \Delta_e\|_1 \leq \|x_1\|_1 + \sum_{e \in \pi(v)} \|\Delta_e\|_1 \leq \|x_1\|_1 + \sum_{e \in E } \|\Delta_e\|_1$ we have
\begin{align}
\label{equ:l_1norm:noise:edge}
\big\| \big( x_1 + \sum_{e \in \pi(v)} \Delta_e \big)_{U}\big\|_1 
\leq 
\frac{\delta^{\text{Non-root}}_{2s^{\prime}}}{1-\delta_{s^{\prime}}} \big( \|x_1\|_1 + \sum_{e \in E} \|\Delta_e\|_1 \big) 
+ 
\frac{ \sqrt{1+\delta^{\text{Non-root}}_{s^{\prime}}}}{1 - \delta^{\text{Non-root}}_{s^{\prime}}} \sqrt{ s^{\prime}} \|A_v  x_{v}\|_2.
\end{align} 

For $e =\{v,w\}\in E$ we now set to bound $\|(\Delta_{e})_{S_{e}}\|_1$ where recall $S_e$ are the largest $s^{\prime}$ elements of $\Delta_e$. Suppose $w$ is closest to the root node. If not, swap the $v,w$ in the following. By adding and subtracting $\big( x_1 + \sum_{\widetilde{e} \in \pi(w) } \Delta_{\widetilde{e}} \big)_{S_{e}}$ we then get 
\begin{align*}
\|(\Delta_{e})_{S_{e}}\|_1 
& \leq 
\big\| \big( x_1 + \sum_{\widetilde{e} \in \pi(w) } \Delta_{\widetilde{e}} \big)_{S_{e}} \big\|_1 
+ 
\big\| \big( x_1 + \sum_{\widetilde{e} \in v } \Delta_{\widetilde{e}} \big)_{S_{e}} \big\|_1 \\
& \leq 
\frac{2 \delta^{\text{Non-root}}_{2s^{\prime}}}{1-\delta_{s^{\prime}}} \big( \|x_1\|_1 + \sum_{e \in E} \|\Delta_e\|_1 \big) 
+ 
\frac{ \sqrt{1+\delta^{\text{Non-root}}_{s^{\prime}}}}{1 - \delta^{\text{Non-root}}_{s^{\prime}}} \sqrt{ s^{\prime}} \big( \|A_v  x_{v}\|_2 + \|A_w  x_{w}\|_2 \big)
\end{align*} 
where on the second inequality we applied $\eqref{equ:l_1norm:noise:edge}$ twice.  Summing the above over all edges $e \in E$, we note 
$\|A_v  x_{v}\|_2$ for $v \in V$ appears at most the max degree of the graph, as such we get 
\begin{align*}
\sum_{e \in E} \|(\Delta_{e})_{S_{e}}\|_1  
& \leq 
\frac{2 n \delta^{\text{Non-root}}_{2s^{\prime}}}{1-\delta^{\text{Non-root}}_{s^{\prime}}} \big( \|x_1\|_1 + \sum_{e \in E} \|\Delta_e\|_1 \big) 
+ 
\frac{ \sqrt{1+\delta^{\text{Non-root}}_{s^{\prime}}}}{1 - \delta^{\text{Non-root}}_{s^{\prime}}} \text{Deg}(G) \sqrt{ s^{\prime}} \sum_{v \in V} \|A_v  x_{v}\|_2\\
& \leq 
\frac{2 n \delta^{\text{Non-root}}_{2s^{\prime}}}{1-\delta^{\text{Non-root}}_{s^{\prime}}} \big( \|x_1\|_1 + \sum_{e \in E} \|\Delta_e\|_1 \big) 
+ 
\frac{ \sqrt{1+\delta^{\text{Non-root}}_{s^{\prime}}}}{1 - \delta_{s^{\prime}}} \text{Deg}(G) \sqrt{ n s^{\prime}} \sqrt{ \sum_{v \in V} \|A_v  x_{v}\|_2^2}
\end{align*}
where on the final inequality we upper bounded $\sum_{v \in V} \|A_v  x_{v}\|_2 \leq \sqrt{n} \sqrt{\sum_{v \in V}  \|A_v  x_{v}\|_2^2}$. 

We now consider the bound for $\|(x_1)_U\|_1$ but for subsets $U$ of size up to $s$. Following an identical set of steps as for \eqref{equ:l_1norm:noise:edge}, but with $s^{\prime}$ swapped with $s$ and $\delta_{s^{\prime}}$ swapped with $\delta^{(1)}_{s}$, we get the upper bound 
\begin{align*}
\| (x_1)_{U} \|_1 
& \leq 
\frac{\delta^{\text{Root}}_{2s}}{1-\delta^{\text{Root}}_{s}} \big( \|x_1\|_1 + \sum_{e \in E} \|\Delta_e\|_1 \big) 
+ 
\frac{ \sqrt{1+\delta^{\text{Root}}_{s}}}{1 - \delta^{\text{Root}}_{s}} \sqrt{ s} \|A_1  x_{1}\|_2 \\
& 
\leq 
\frac{\delta^{\text{Root}}_{2s}}{1-\delta^{\text{Root}}_{s}} \big( \|x_1\|_1 + \sum_{e \in E} \|\Delta_e\|_1 \big) 
+ 
\frac{ \sqrt{1+\delta^{\text{Root}}_{s}}}{1 - \delta^{\text{Root}}_{s}} \sqrt{ s} \sum_{v \in V} \sqrt{ \|A_v  x_{v}\|_2^2} 
\end{align*}
where at the end we simply upper bounded $ \|A_1  x_{1}\|_2 =  \sqrt{ \|A_1  x_{1}\|_2^2} \leq \sqrt{  \sum_{v \in V} \|A_v x_v \|_2^2 }$. 
Picking $U = S_1$, adding together the upper bound for $\sum_{e \in E} \|(\Delta_{e})_{S_{e}}\|_1 $ and $\| (x_1)_{U} \|_1 $, and collecting terms then yields the result.

\end{document}